\documentclass[12pt]{amsart}
\usepackage{geometry} 
\usepackage{amscd}
\usepackage{amssymb}
\usepackage{amsmath}
\usepackage{amsthm}

\newtheorem{thm}{Theorem}[section]
\newtheorem{lemma}[thm]{Lemma}

\newtheorem{prop}[thm]{Proposition}

\theoremstyle{definition}

\newtheorem{rem}[thm]{Remark}

\newtheorem{defn}[thm]{Definition}

\makeatletter

\newcommand{\isom}{\overset{\sim}{\rightarrow}}

\title{Iwasawa theory of de Rham $(\varphi,\Gamma)$-modules 
over the Robba ring.}
\author{Kentaro Nakamura}
\date{} 

\begin{document}

\maketitle
\pagestyle{plain}
\footnote{2008 Mathematical Subject Classification 11F80 (primary), 11F85, 11S25 (secondary).
}
\begin{abstract}
The aim of this article is to study Bloch-Kato's exponential map and Perrin-Riou's big exponential map 
purely in terms of $(\varphi,\Gamma)$-modules over the Robba ring. 
We first generalize the definition of Bloch-Kato's exponential map for all the $(\varphi,\Gamma)$-modules 
without using Fontaine's rings $\bold{B}_{\mathrm{crys}}$, $\bold{B}_{\mathrm{dR}}$ of $p$-adic periods, and then generalize the construction of Perrin-Riou's big exponential map for all the de Rham $(\varphi,\Gamma)$-modules 
and prove that this map interpolates our Bloch-Kato's exponential map and the dual exponential map. 
Finally, we prove a theorem concerning the determinant of our big exponential map, which 
is a generalization of the theorem $\delta(V)$ of Perrin-Riou. The key ingredients for our study are
Pottharst's theory of the analytic Iwasawa cohomology and Berger's construction of $p$-adic differential equations associated to 
de Rham $(\varphi,\Gamma)$-modules.
\end{abstract}
\setcounter{tocdepth}{2}
\tableofcontents

\section{Introduction.}
\subsection{Introduction}
Let $p$ be a prime number, $K$  a finite extension of $\mathbb{Q}_p$ and $G_K$  the 
absolute Galois group of $K$. Let $\bold{B}_{\mathrm{crys}}$, $\bold{B}_e:=\bold{B}_{\mathrm{crys}}^{\varphi=1}$, $\bold{B}^+_{\mathrm{dR}}$ and $\bold{B}_{\mathrm{dR}}$ be 
Fontaine's rings of $p$-adic periods (\cite{Fo94}).

By the results of Fontaine (\cite{Fo90}), Cherbonnier-Colmez (\cite{CC98}) and 
Kedlaya (\cite{Ke04}), the category of $p$-adic representations of $G_K$ is naturally embedded in the category of 
$(\varphi,\Gamma_K)$-modules over the Robba ring $\bold{B}^{\dagger}_{\mathrm{rig},K}$. 
 The $(\varphi,\Gamma)$-modules 
corresponding to $p$-adic representations are called \'etale $(\varphi,\Gamma)$-modules.
 
The aim of this article is to study Bloch-Kato's 
exponential map and Perrin-Riou's big exponential map in the framework  of $(\varphi,\Gamma)$-modules. In particular, 
we generalize Perrin-Riou's big exponential map to all the de Rham $(\varphi,\Gamma)$-modules.
\subsection{Bloch-Kato's exponential map}
For a $p$-adic representation $V$ of $G_K$, Bloch-Kato (\cite{BK90}) defined a $\mathbb{Q}_p$-linear map 
$$\mathrm{exp}_{K,V}:=\delta_{1,V}:\bold{D}^K_{\mathrm{dR}}(V)\rightarrow \mathrm{H}^1(K, V),$$
where we put $\bold{D}^K_{\mathrm{dR}}(V):=(\bold{B}_{\mathrm{dR}}\otimes_{\mathbb{Q}_p}V)^{G_K}$, as the first connecting 
homomorphism of the long exact sequence 

\begin{align*}
0 \rightarrow& \mathrm{H}^0(K, V) \rightarrow  \mathrm{H}^0(K, \bold{B}_e\otimes_{\mathbb{Q}_p}V)\oplus \mathrm{H}^0(K, \bold{B}^+_{\mathrm{dR}}\otimes_{\mathbb{Q}_p}V)\rightarrow  \mathrm{H}^0(K, \bold{B}_{\mathrm{dR}}\otimes_{\mathbb{Q}_p}V)\\
  \xrightarrow{\delta_{1,V}} & \mathrm{H}^1(K, V)\rightarrow  \mathrm{H}^1(K, \bold{B}_e\otimes_{\mathbb{Q}_p}V)\oplus \mathrm{H}^1(K, \bold{B}^+_{\mathrm{dR}}\otimes_{\mathbb{Q}_p}V)\rightarrow  \mathrm{H}^1(K, \bold{B}_{\mathrm{dR}}\otimes_{\mathbb{Q}_p}V)\\
  \xrightarrow{\delta_{2,V}}  &\mathrm{H}^2(K, V)\rightarrow \mathrm{H}^2(K, \bold{B}_e\otimes_{\mathbb{Q}_p}V) \rightarrow 0
\end{align*}

associated to the short exact sequence 
obtained by tensoring $V$ with the so called Bloch-Kato's fundamental exact sequence
$$0\rightarrow \mathbb{Q}_p\xrightarrow{x\mapsto (x,x)}\bold{B}_{e}\oplus \bold{B}^+_{\mathrm{dR}}\xrightarrow{(x,y)\mapsto x-y}
\bold{B}_{\mathrm{dR}}\rightarrow 0.$$
When $V$ is a de Rham representation, Kato (\cite{Ka93a}) defined the dual exponential map 
$$\mathrm{exp}^{*}_{K,V^{\lor}(1)}:\mathrm{H}^1(K, V)\rightarrow \bold{D}^K_{\mathrm{dR}}(V)$$ 
using Tate's paring $\cup:\mathrm{H}^1(K, V)\times \mathrm{H}^1(K, V^{\lor}(1))\rightarrow \mathbb{Q}_p$ and 
the canonical paring $\bold{D}^K_{\mathrm{dR}}(V)\times \bold{D}^K_{\mathrm{dR}}(V^{\lor}(1))\rightarrow K$.
These maps describe the mysterious relationship between Galois objects and differential objects. 
 In fact, when $V=\mathbb{Q}_p(1)$ or 
 $V$ is the $p$-adic Tate module of an elliptic curve over $\mathbb{Q}$, Kato ( \cite{Ka93a}, \cite{Ka04}) 
 proved 
 that the values of $\mathrm{exp}^{*}_{\mathbb{Q}_p(\zeta_{p^n}), V^{\lor}(1-k)}$ for suitable $k\leqq 0$ at some special arithmetic elements (i.e. cyclotomic units or Kato's elements obtained from his Euler system
 ) can be described by using the special values of the $L$-functions associated to 
 cyclotomic twists of $V$.

In this article, we first generalize the above long exact sequence and the definition of Bloch-Kato's exponential and  
the dual exponential maps
for $(\varphi,\Gamma_K)$-modules over $\bold{B}^{\dagger}_{\mathrm{rig},K}$.

Fix a set $\{\zeta_{p^n}\}_{n\geqq 1}\subseteq \overline{K}$ such that $\zeta_p\not=1$, $\zeta_p^p=1$ and $\zeta_{p^{n+1}}^p=\zeta_{p^n}$ for each $n\geqq 1$. Set 
$K_n:=K(\zeta_{p^n})$, $K_{\infty}:=\cup_{n}K_n$ and $\Gamma_K:=\mathrm{Gal}(K_{\infty}/K)$. Let $t:=\mathrm{log}(1+T)\in \bold{B}^{\dagger}_{\mathrm{rig},K}$ be the period of $\mathbb{Q}_p(1)$ determined by $\{\zeta_{p^n}\}_{n\geqq 1}$ (see $\S2.1$ for the precise definition).

Let $D$ be a $(\varphi,\Gamma_K)$-module over $\bold{B}^{\dagger}_{\mathrm{rig},K}$. 
Taking the ``stalk at $\zeta_{p^n}-1$" ($n\geqq 1$), 
we can define $K_{\infty}[[t]]:=\cup_{n}K_n[[t]]$-modules $\bold{D}^+_{\mathrm{dif}}(D)$ and $\bold{D}_{\mathrm{dif}}(D):=\bold{D}^+_{\mathrm{dif}}(D)[1/t]$ with semi-linear 
$\Gamma_K:=\mathrm{Gal}(K_{\infty}/K)$-action. Using the $\varphi, \Gamma_K$-actions, we can define cohomologies 
$$\mathrm{H}^q(K, D),\,\, \mathrm{H}^q(K, D[1/t]), \,\,\mathrm{H}^q(K, \bold{D}^+_{\mathrm{dif}}(D))\,\,\text{ and } \,\,\mathrm{H}^q(K, \bold{D}_{\mathrm{dif}}(D))$$ which  correspond to $$\mathrm{H}^q(K, V) ,\,\, \mathrm{H}^q(K, \bold{B}_{e}\otimes_{\mathbb{Q}_p}V), \,\,
\mathrm{H}^q(K, \bold{B}^+_{\mathrm{dR}}\otimes_{\mathbb{Q}_p}V)\,\,\text{ and }\,\,\mathrm{H}^q(K, \bold{B}_{\mathrm{dR}}\otimes_{\mathbb{Q}_p}V)$$ respectively. 

Our first result is the following theorem (Theorem \ref{exp} and Theorem \ref{4.5}), which is the $(\varphi,\Gamma)$-module version of the above 
long exact sequence and its comparison with that of  \'etale-case.
\begin{thm}
\begin{itemize}
\item[(1)]
We have the following functorial exact sequence
\begin{align*}
0 \rightarrow& \mathrm{H}^0(K, D) \rightarrow  \mathrm{H}^0(K, D[1/t])\oplus \mathrm{H}^0(K, \bold{D}^+_{\mathrm{dif}}(D))
\rightarrow  \mathrm{H}^0(K, \bold{D}_{\mathrm{dif}}(D)) \\
   \xrightarrow{\delta_{1,D}}&\mathrm{H}^1(K, D)\rightarrow \mathrm{H}^1(K, D[1/t])\oplus \mathrm{H}^1(K, \bold{D}^+_{\mathrm{dif}}(D))
\rightarrow \mathrm{H}^1(K, \bold{D}_{\mathrm{dif}}(D)) \\
  \xrightarrow{\delta_{2,D}}&\mathrm{H}^2(K, D)\rightarrow \mathrm{H}^2(K, D[1/t]) \rightarrow 0.
\end{align*}
  \item[(2)]Let $D(V)$ be the $(\varphi,\Gamma_K)$-module over $\bold{B}^{\dagger}_{\mathrm{rig},K}$ associated to $V$. Then, 
  we have functorial isomorphisms
  \begin{itemize}
  \item[(i)]$\mathrm{H}^q(K, V)\isom \mathrm{H}^q(K, D(V))$,
   \item[(ii)]$\mathrm{H}^q(K, \bold{B}_e\otimes_{\mathbb{Q}_p}V)\isom \mathrm{H}^q(K, D(V)[1/t])$,
   \item[(iii)]$\mathrm{H}^q(K, \bold{B}^{(+)}_{\mathrm{dR}}\otimes_{\mathbb{Q}_p}V)\isom \mathrm{H}^q(K, \bold{D}^{(+)}_{\mathrm{dif}}(D(V)))$
   \end{itemize}
   for each $q\geqq 0$, and these comparison isomorphisms induce an isomorphism from the long exact sequence associated to $V$ to 
   that associated to $D(V)$.
  
  \end{itemize}

\end{thm}
\begin{rem}
The isomorphism of (i) is due to Liu (\cite{Li08}), and that of (iii) is due to Fontaine (\cite{Fo03}).
\end{rem}
\begin{rem}
We construct this long exact sequence purely in terms of $(\varphi,\Gamma)$-modules 
without using Fontaine's rings $\bold{B}_{\mathrm{crys}}$, $\bold{B}^+_{\mathrm{dR}}$ and 
$\bold{B}_{\mathrm{dR}}$. 
As will be shown in this article, this fact enables us to re-prove 
some results concerning Bloch-Kato's or Perrin-Riou's exponential maps more directly.

\end{rem}
\begin{rem}
In fact, in $\S$ 2.5, we prove the above comparison result (2)  in a more general setting.
In \cite{Ber08a}, Berger defined a notion of $B$-pairs using $\bold{B}_e$, $\bold{B}^+_{\mathrm{dR}}$ and $\bold{B}_{\mathrm{dR}}$, whose category naturally contains 
the category of $p$-adic representations of $G_K$, and established an equivalence of categories between the category of $B$-pairs and 
that of $(\varphi,\Gamma_K)$-modules over $\bold{B}^{\dagger}_{\mathrm{rig},K}$. In $\S$ 2.5, we prove the comparison isomorphisms for all the $B$-pairs (see Theorem \ref{4.5}).

\end{rem}
As in the case of $p$-adic representations, we define Bloch-Kato's exponential map of $D$ as the connecting homomorphism 
of the above exact sequence
$$\mathrm{exp}_{K, D}:=\delta_{1,D}:\bold{D}^K_{\mathrm{dR}}(D)\rightarrow \mathrm{H}^1(K, D),$$
where we put $\bold{D}^K_{\mathrm{dR}}(D):=\mathrm{H}^0(K,\bold{D}_{\mathrm{dif}}(D))$. When $D$ is a de Rham $(\varphi,\Gamma)$-module, 
then we also define the dual exponential map 
$$\mathrm{exp}^{*}_{K,D^{\lor}(1)}\mathrm{H}^1(K, D)\rightarrow \bold{D}^K_{\mathrm{dR}}(D)$$
in the same way as in the case of $p$-adic representations.

\subsection{Perrin-Riou's big exponential map}
To construct a $p$-adic $L$-function for a $p$-adic Galois representation $V$ coming from a motive, 
it is crucial to $p$-adically interpolate the special values of the complex $L$-functions associated 
to cyclotomic  twists of $V$. 
Since Bloch-Kato's exponential map and the dual exponential map relate some arithmetic elements in Galois cohomology groups with 
the special values of the $L$-functions, it is crucial to $p$-adically interpolate 
Bloch-Kato's exponential map and dual exponential map for the construction of the $p$-adic $L$-function and for relating 
the $p$-adic $L$-function with the Selmer group.

Let $\Lambda:=\mathbb{Z}_p[[\Gamma_K]]$ be the Iwasawa algebra of $\Gamma_K$,  
$\Lambda_{\infty}$ the $\mathbb{Q}_p$-valued distribution algebra of $\Gamma_K$ 
(see $\S3.1$ for the precise definition). For a $p$-adic representation $V$ of 
$G_K$, Perrin-Riou (\cite{Per92}) defined a $\Lambda$-module
$$\bold{H}^q_{\mathrm{Iw}}(K, V):=(\varprojlim_n \mathrm{H}^q(K_n, T))\otimes_{\mathbb{Z}_p}\mathbb{Q}_p,$$  called
the Iwasawa cohomology of $V$, where $T$ is a $G_K$-stable $\mathbb{Z}_p$-lattice of $V$ and the transition map is the corestriction map. 
This $\Lambda$-module $p$-adically interpolates $\mathrm{H}^q(L, V(k))$ for any 
$L=K, K_n$ and $k\in \mathbb{Z}$, i.e. we have 
a natural projection map 
$$\mathrm{pr}_{L,V(k)}:\bold{H}^q_{\mathrm{Iw}}(K, V)\rightarrow \mathrm{H}^q(L, V(k))$$
for each $L$ and $k$.
When $K$ is unramified over $\mathbb{Q}_p$ and $V$ is a crystalline representation of $G_K$, 
Perrin-Riou (\cite{Per94}) constructed a system of functorial $\Lambda_{\infty}$-morphisms
$$\Omega_{V,h}:(\Lambda_{\infty}\otimes_{\mathbb{Q}_p}\bold{D}^K_{\mathrm{crys}}(V))^{\widetilde{\Delta}=0}\rightarrow \Lambda_{\infty}\otimes_{\Lambda}(\bold{H}^1_{\mathrm{Iw}}(K, V)/\bold{H}^1_{\mathrm{Iw}}(K, V)_{\Lambda-\mathrm{torsion}})$$ 
for each $h\geqq 1$ such that $\mathrm{Fil}^{-h}\bold{D}^K_{\mathrm{dR}}(V)=\bold{D}^K_{\mathrm{dR}}(V)$, and proved that this interpolates $\mathrm{exp}_{L,V(k)}$ and $\mathrm{exp}^{*}_{L,V^{\lor}(1+k)}$ for any $L=K_n, K$ and for 
suitable $k$.  Here, the source of the map $\Omega_{V,h}$ is a $\Lambda_{\infty}$-module which 
$p$-adically interpolates 
$\bold{D}^L_{\mathrm{dR}}(V(k))$ for any $L$ and $k$.
This map $\Omega_{V,h}$ is the most important ingredient for her study of $p$-adic $L$-functions (\cite{Per95}). 

The main purpose of this article is to generalize the map $\Omega_{V,h}$ to all the de Rham $(\varphi,\Gamma)$-modules. For this generalization, the following  two notions are essential;
\begin{itemize}
\item[(1)]Pottharst's theory of the analytic Iwasawa cohomology,
\item[(2)]Berger's construction of $p$-adic differential equations associated to de Rham $(\varphi,\Gamma)$-modules.

\end{itemize}

As for (1), for  each $(\varphi,\Gamma_K)$-module $D$ over $\bold{B}^{\dagger}_{\mathrm{rig},K}$, 
Pottharst (\cite{Po12b}) defined a $\Lambda_{\infty}$-module $$\bold{H}^q_{\mathrm{Iw}}(K, D)$$ 
called the analytic Iwasawa cohomology as a generalization of the Iwasawa cohomology of $p$-adic representations. In fact, he proved that we have 
a functorial $\Lambda_{\infty}$-isomorphism 
$$\bold{H}^q_{\mathrm{Iw}}(K, D(V))\isom \Lambda_{\infty}\otimes_{\Lambda}\bold{H}^q_{\mathrm{Iw}}(K, V)$$
for each $p$-adic representation $V$.

As for (2), let $D$ be a de Rham $(\varphi,\Gamma)$-module. In order to interpolate $\bold{D}^L_{\mathrm{dR}}(D(k))$, we need to generalize the $\Lambda_{\infty}$-module
$\Lambda_{\infty}\otimes_{\mathbb{Q}_p}\bold{D}^K_{\mathrm{crys}}(V)$ for de Rham case. Our
 idea is to use Berger's $p$-adic differential equation 
$\bold{N}_{\mathrm{rig}}(D)$. Let $\nabla_0:=\frac{\mathrm{log}(\gamma)}{\mathrm{log}{\chi(\gamma)}}\in 
\Lambda_{\infty}$, where $\gamma\in \Gamma_K$ is a non-torsion element.  
For each $i\in \mathbb{Z}$, we define $\nabla_i:=\nabla_0-i\in \Lambda_{\infty}$. 
Let $\chi:G_K\rightarrow \mathbb{Z}_p^{\times}$ be the $p$-adic cyclotomic character. 
$\nabla_0$ acts on $D$  as a differential operator and acts on $\bold{B}^{\dagger}_{\mathrm{rig},\mathbb{Q}_p}$ by $t(1+T)\frac{d}{dT}$.

In \cite{Ber02},\cite{Ber08b}, for a de Rham $(\varphi,\Gamma_K)$-module $D$ over 
$\bold{B}^{\dagger}_{\mathrm{rig},K}$, 
Berger defined a $(\varphi,\Gamma_K)$-submodule $\bold{N}_{\mathrm{rig}}(D)\subseteq D[1/t]$ which satisfies that 
$\nabla_0(\bold{N}_{\mathrm{rig}}(D))\subseteq t\bold{N}_{\mathrm{rig}}(D)$. This condition enables us to define 
another better differential  operator 
$$\tilde{\partial}:=\nabla_0\otimes e_{-1}:\bold{N}_{\mathrm{rig}}(D)\rightarrow \bold{N}_{\mathrm{rig}}(D(-1)).$$
The map $\tilde{\partial}$ naturally induces a $\mathbb{Q}_p$-linear map
$$\tilde{\partial}:\bold{H}^1_{\mathrm{Iw}}(K,\bold{N}_{\mathrm{rig}}(D))\rightarrow \bold{H}^1_{\mathrm{Iw}}(K, \bold{N}_{\mathrm{rig}}(D(-1))).$$
In $\S 3.2$, we define a canonical projection map for each $L=K, K_n$,
$$T_{L}:\bold{H}^1_{\mathrm{Iw}}(K, \bold{N}_{\mathrm{rig}}(D))\rightarrow \bold{D}^L_{\mathrm{dR}}(D).$$

The main theorem of this article is the following (Theorem \ref{3.10}), which 
concerns with the existence of a $\Lambda_{\infty}$-morphism $\mathrm{Exp}_{D,h}$ for each  $h\in \mathbb{Z}_{\geqq 1}$ such that $\mathrm{Fil}^{-h}\bold{D}^K_{\mathrm{dR}}(D)=\bold{D}^K_{\mathrm{dR}}(D)$ which interpolates 
$\mathrm{exp}_{L, V(k)}$ for some $k\geqq -(h-1)$ and $\mathrm{exp}^*_{L, D^{\lor}(1-k)}$ for any $k\leqq -h$. 
\begin{thm}
Let $D$ be a de Rham $(\varphi,\Gamma_K)$-module over $\bold{B}^{\dagger}_{\mathrm{rig},K}$. 
Let $h\in \mathbb{Z}_{\geqq 1}$ such that $\mathrm{Fil}^{-h}\bold{D}^K_{\mathrm{dR}}(D)=\bold{D}^K_{\mathrm{dR}}(D)$. 
Then there exists a functorial $\Lambda_{\infty}$-linear map 
$$\mathrm{Exp}_{D, h}:\bold{H}^1_{\mathrm{Iw}}(K, \bold{N}_{\mathrm{rig}}(D))\rightarrow \bold{H}^1_{\mathrm{Iw}}(K, D)$$
such that, for any $x\in \bold{H}^1_{\mathrm{Iw}}(K, \bold{N}_{\mathrm{rig}}(D))$, 
\begin{itemize}
\item[(1)]if $k\geqq 1$ and there exists $x_k\in \mathrm{H}^1(K, \bold{N}_{\mathrm{rig}}(D(k)))$ such that $\widetilde{\partial}^k(x_k)=x$ or if $0\geqq k\geqq -(h-1)$ and $x_k:=\tilde{\partial}^{-k}(x)$, then 
$$\mathrm{pr}_{L,D(k)}(\mathrm{Exp}_{D,h}(x))=\frac{(-1)^{h+k-1}(h+k-1)!|\Gamma_{L, \mathrm{tor}}|}{p^{m(L)}}\mathrm{exp}_{L,D(k)}(T_{L}(x_k))$$
 for each $L=K, K_n$,
\item[(2)]if $-h\geqq k$, then
$$\mathrm{exp}^{*}_{L,D^{\lor}(1-k)}(\mathrm{pr}_{L,D(k)}(\mathrm{Exp}_{D, h}(x))=\frac{|\Gamma_{L,\mathrm{tor}}|}{(-h-k)!  p^{m(L)}}
T_L(\widetilde{\partial}^{-k}(x))$$
 for each $L=K, K_n$,

\end{itemize}
where we put $m(L):=\mathrm{min}\{v_p(\mathrm{log}(\chi(\gamma))|\gamma\in \Gamma_L\}$ for each $L=K, K_n$.

\end{thm}
\begin{rem}
The definition of $\mathrm{Exp}_{D,h}$ is strongly influenced by Berger's work (\cite{Ber03}) concerning the re-interpretation of 
Perrin-Riou's map in terms of $(\varphi,\Gamma)$-modules. 
In particular, this theorem is a generalization of Theorem 2.10 of \cite{Ber03}  to 
all the de Rham $(\varphi,\Gamma_K)$-modules over $\bold{B}^{\dagger}_{\mathrm{rig},K}$ 
for any $p$-adic field $K$.

\end{rem}

\begin{rem}
When $K$ is unramified and $V$ is crystalline, we can easily compare $\Lambda_{\infty}\otimes_{\mathbb{Q}_p}\bold{D}^K_{\mathrm{crys}}(V)$ with 
$\bold{H}^1_{\mathrm{Iw}}(K, \bold{N}_{\mathrm{rig}}(D(V)))$. 
Hence, we can also compare $\Omega_{V,h}$ with $\mathrm{Exp}_{D(V), h}$ by the
Berger's work above . Therefore, the maps $\mathrm{Exp}_{D,h}$ and their interpolation formulae 
can be regarded as a generalization of Perrin-Riou's theorem (Theorem 3.2.3 of \cite{Per94}) on the existence of 
$\Omega_{V,h}$ and 
their interpolation formulae to all the de Rham $(\varphi,\Gamma)$-modules. 
Moreover, Pottharst (\cite{Po12b}) generalized $\Omega_{V,h}$ (precisely, the inverse of 
$\Omega_{V,h}$ called big logarithm) to crystalline $(\varphi,\Gamma)$-modules using the theory of Wach modules. We can also compare Pottharst's map 
with our map. See $\S 3.5$ for more details about the comparison of our big exponential map with their ones  in crystalline case. On the other hands, Colmez (Theorem 7 of \cite{Col98}) generalized Perrin-Riou's map to all the de Rham $p$-adic representations by a completely different method. 
\end{rem}
\begin{rem}
In fact, Perrin-Riou and Comez also proved the uniqueness of their big exponential maps using 
the theory of ``tempered Iwasawa cohomologies". If we can generalize the theory of 
tempered Iwasawa cohomologies for $(\varphi,\Gamma)$-modules, it will be 
possible to prove the uniqueness of our map $\mathrm{Exp}_{D,h}$.
\end{rem}

Finally, we prove a theorem ( Theorem \ref{3.21}) concerning 
the determinant of $\mathrm{Exp}_{D, h}$. 
For a torsion co-admissible $\Lambda_{\infty}$-module $M$, denote by
$\mathrm{char}_{\Lambda_{\infty}}(M)$ the characteristic ideal of $M$, which is a principal ideal of 
$\Lambda_{\infty}$.
\begin{thm}$(\delta(D))$
Let $D$ be a de Rham $(\varphi,\Gamma_K)$-module over $\bold{B}^{\dagger}_{\mathrm{rig},K}$ of rank  $d$
with Hodge-Tate weights $\{h_1, h_2,\cdots,h_d\}$. 
For each $h\geqq 1$ such that $\mathrm{Fil}^{-h}\bold{D}^K_{\mathrm{dR}}(D)=\bold{D}^K_{\mathrm{dR}}(D)$,
we have the following equality of principal fractional ideals of $\Lambda_{\infty}$,
\begin{multline*}
\frac{1}{(\prod_{i=1}^d \prod_{j_i=0}^{h-h_i-1}
\nabla_{h_i+j_i})^{[K:\mathbb{Q}_p]}}\mathrm{det}_{\Lambda_{\infty}}(\bold{H}^1_{\mathrm{Iw}}(K,\bold{N}_{\mathrm{rig}}(D))\xrightarrow{\mathrm{Exp}_{D,h}} \bold{H}^1_{\mathrm{Iw}}(K, D)) \\
=
\mathrm{char}_{\Lambda_{\infty}}(\bold{H}^2_{\mathrm{Iw}}(K, D))(\mathrm{char}_{\Lambda_{\infty}}\bold{H}^2_{\mathrm{Iw}}(K, \bold{N}_{\mathrm{rig}}(D)))^{-1}.
\end{multline*}

\end{thm}
\begin{rem}
This theorem is a generalization of the theorem $\delta(V)$ which was conjectured by Perrin-Riou (\cite{Per94}) and was proved as a 
consequence of her reciprocity law conjecture $\mathrm{Rec}(V)$ proved by Colmez (\cite{Col98}), Kato-Kurihara-Tsuji(\cite{KKT96}), 
Benois (\cite{Ben00}) and Berger (\cite{Ber03}). The theorem $\delta(V)$ is very important in her works on $p$-adic $L$-functions. For example, this enables us to define 
the ``inverse of $\Omega_{V,h}$", which is a generalization of Coleman homomorphism and 
from which we can conjecturally define the $p$-adic L-functions associated to $V$. In the non-\'etale crystalline case, Pottharst also generalized the theorem $\delta(V)$ and proved his theorem 
$\delta(D)$ for all the crystalline 
$(\varphi,\Gamma)$-modules $D$  by reducing to the \'etale case using a slope filtration argument.
In $\S 3.5$, when $D$ is crystalline, we show that 
 our $\delta(D)$ is equivalent to their $\delta(V)$ or $\delta(D)$. Moreover, our proof does not use 
 $\mathrm{Rec}(V)$ and is  via a direct computation rather than by reducing to the \'etale case, hence gives a new and more direct proof of their theorems.
\end{rem}

Introducing non-\'etale $(\varphi,\Gamma)$-modules to Iwasawa theory was initiated by Pottharst 
in \cite{Po12a} and \cite{Po12b}, where he studied Iwasawa main conjecture for $p$-supersingular 
modular forms by generalizing the notion of Greenberg's Selmer groups using his theories 
of the analytic Iwasawa cohomology and of the big logarithm. Our interpolation formula of
the big exponential map might help study the values of the p-adic L-functions. Moreover, the author 
hopes that the results of this article will shed some light on Iwasawa theory or $p$-adic $L$-functions 
in the case of bad reductions. As another application of this article, in the next article (\cite{Na12}), 
the author generalize Kato's local $\varepsilon$-conjecture (\cite{Ka93b}), which is intimately related with Kato's 
generalized Iwasawa main conjecture (\cite{Ka93a}), 
for families of $(\varphi,\Gamma)$-modules over the Robba ring and prove the conjecture in some special cases using the results of this article.

\subsection*{Acknowledgement.}
The author would like to thank Kenichi Bannai for constantly encouraging the author.
He also would like to thank Ga\"etan Chenevier and Jonathan Pottharst for  discussing 
related topics on $(\varphi,\Gamma)$-modules over the Robba ring.


\subsection*{Notation.}
Let $p$ be a prime number. Let $K$  be a finite extension of $\mathbb{Q}_p$, $K_0$  the 
maximal unramified extension of $\mathbb{Q}_p$ in $K$,
$\overline{K}$  a fixed algebraic closure of $K$, $\mathbb{C}_p$  the $p$-adic completion of $\overline{K}$.
Let $v_p:\mathbb{C}_p^{\times}\rightarrow \mathbb{Q}$ be the valuation such that $v_p(p)=1$. 
Let 
$|-|_p:\mathbb{C}_p^{\times}\rightarrow \mathbb{Q}_{\geqq 0}$ be the $p$-adic absolute value such that 
$|p|_p:=1/p$.
Let $G_K:=\mathrm{Gal}(\overline{K}/K)$ be the absolute Galois group of $K$. 
We fix a set $\{\zeta_{p^n}\}_{n\geqq 1}\subseteq \overline{K}^{\times}$ such that 
$\zeta_p\not=1$ and $\zeta_p^p=1$ and $\zeta_{p^{n+1}}^p=\zeta_{p^n}$ for any $n\geqq 1$. 
We put $K_n:=K(\zeta_{p^n})$ ($n\geqq 1$) and $K_{\infty}:=\cup_{n\geqq 1}K_n$.
Let $\chi:G_K\rightarrow \mathbb{Z}_p^{\times}$ be the $p$-adic cyclotomic character (i.e. 
the character defined by the formula $g(\zeta_{p^n})=\zeta_{p^n}^{\chi(g)}$ for any $n\geqq 1$ and 
$g\in G_K$). 
We put $\Gamma_K:=G_K/\mathrm{Ker}(\chi)\isom \mathrm{Gal}(K_{\infty}/K)$. Denote by the same letter $\chi:\Gamma_K\hookrightarrow 
\mathbb{Z}_p^{\times}$ the map which is naturally induced by $\chi:G_K\rightarrow \mathbb{Z}_p^{\times}$. Define the base $e_1:=(\zeta_{p^n})_{n\geqq 1}\in\mathbb{Z}_p(1):=\varprojlim_n\mu_{p^n}(\overline{K})$. Set $e_k:=e_1^{\otimes k}\in \mathbb{Z}_p(k):=\mathbb{Z}_p(1)^{\otimes k}$ for each $k\in \mathbb{Z}$. In this article, we normalize Hodge-Tate weight such that  that of 
$\mathbb{Q}_p(1)$ is $1$. For a finite group $G$, let denote by $|G|$  the order of $G$.

\section{Bloch-Kato's exponential map for $(\varphi,\Gamma)$-modules}
In this section, we define Bloch-Kato's exponential and the dual exponential maps 
for $(\varphi,\Gamma)$-modules over the Robba ring. 
In $\S2.1$, we first recall the definition of $(\varphi,\Gamma)$-modules over the Robba ring. 
In $\S2.2$, we recall the definitions of some cohomology theories associated to 
$(\varphi,\Gamma)$-modules. The subsection $\S2.3$ is the main part of this section, 
where we generalize Bloch-Kato's fundamental exact sequence to all the 
$(\varphi,\Gamma)$-modules, and then define Bloch-Kato's exponential map for 
them and gives a explicit formula of this map. In $\S2.4$, we define the 
dual exponential map explicitly and then prove that this is the adjoint of our 
Bloch-Kato's exponential map. In the final subsection $\S2.5$, we compare our 
exponential map with classical Bloch-Kato's exponential map using the notion of $B$-pairs.

\subsection{$(\varphi,\Gamma)$-modules over the Robba ring}
In this subsection, we recall the definition of $(\varphi,\Gamma)$-modules 
over the Robba ring.

We first recall the definition of Fontaine's and Berger's rings of $p$-adic periods (\cite{Fo94}, \cite{Ber02}). Almost all rings in this paragraph are used only in $\S2.4$, where we compare 
our exponential map with classical Bloch-Kato's exponential map. Define  $\widetilde{\bold{E}}^+:=\varprojlim_{n\geqq 0}\mathcal{O}_{\mathbb{C}_p}/p$, where all the transition maps are the $p$-th power map.
The ring $\widetilde{\bold{E}}^+$ is equipped with a valuation $v_{\widetilde{\bold{E}}^+}$ defined by $v_{\widetilde{\bold{E}}^+}((\overline{x}_n)_{n\geqq 0}):=\lim_{n\rightarrow \infty}p^nv_p(x_n)$, where 
$x_n\in \mathcal{O}_{\mathbb{C}_p}$ is a lift of $\overline{x}_n\in \mathcal{O}_{\mathbb{C}_p}/p$. 
By this $v_{\widetilde{\bold{E}}^+}$, $\widetilde{\bold{E}}^+$ is a perfect complete valuation ring. We denote by $\widetilde{\bold{E}}:=\mathrm{Frac}(\widetilde{\bold{E}}^+)$ the fraction field of $\widetilde{\bold{E}}^+$.
We define $\varepsilon:=(\overline{\zeta}_{p^n})_{n\geqq 0}\in \widetilde{\bold{E}}^+$ for the fixed set $\{\zeta_{p^n}\}_{n\geqq 0}$.  We have 
$v_{\widetilde{\bold{E}}^+}(\varepsilon-1)=\frac{p}{p-1}$. Define 
$\widetilde{p}:=(\overline{p}_n)_{n\geqq 0}\in \widetilde{\bold{E}}^+$ where $p_0:=p$ and $p^p_{n+1}=p_n$ for any $n\geqq 0$. Let $\widetilde{\bold{A}}^+:=W(\widetilde{\bold{E}}^+), \widetilde{\bold{A}}:=W(\widetilde{\bold{E}})$ be the rings of Witt vectors of $\widetilde{\bold{E}}^+$ and $\widetilde{\bold{E}}$ respectively, 
which are naturally equipped with actions of $\varphi$ and $G_K$. For each $a\in \widetilde{\bold{E}}$, denote by $[a]\in \widetilde{\bold{A}}$  the Teichm\"uller lift of $a$. 
We have a natural $G_K$-equivariant surjective ring homomorphism  $\theta:\widetilde{\bold{A}}^+\rightarrow \mathcal{O}_{\mathbb{C}_p}$ such that 
$\theta([(\overline{x}_n)_{n\geqq 0}]):=\lim_{n\rightarrow \infty} x_n^{p^n}$, where $x_n\in \mathcal{O}_{\mathbb{C}_p}$ is a lift of $\overline{x}_n$.  We have 
$\mathrm{Ker}(\theta)=([\widetilde{p}]-p)$. We define $\bold{B}^+_{\mathrm{dR}}:=\varprojlim_{n\geqq 0}\widetilde{\bold{A}}^{+}[1/p]/(\mathrm{Ker}(\theta)[1/p])^n$ and define an 
element $t:=\mathrm{log}([\varepsilon])=\sum_{n\geqq 1}\frac{(-1)^{n-1}}{n}([\varepsilon]-1)^n\in \bold{B}^+_{\mathrm{dR}}$, then 
$\bold{B}^+_{\mathrm{dR}}$ is a discrete valuation ring with the maximal ideal $t\bold{B}^+_{\mathrm{dR}}$ and with the residue field $\mathbb{C}_p$. We have $\varphi(t)=pt$ and $\gamma(t)=\chi(\gamma)t$ for any $\gamma\in \Gamma_K$.
We put  $\bold{B}_{\mathrm{dR}}:=\mathrm{Frac}(\bold{B}^+_{\mathrm{dR}})=\bold{B}^+_{\mathrm{dR}}[1/t]$.
 These rings $\bold{B}^+_{\mathrm{dR}}$ and $\bold{B}_{\mathrm{dR}}$ are naturally equipped with $G_K$-actions. 
 Next, we define $\widetilde{\bold{B}}^{\dagger}_{\mathrm{rig}}$ and $\bold{B}_{\mathrm{max}}$. For each $0\leqq r\leqq s< +\infty$ such that $r,s\in \mathbb{Q}$, we define 
 a ring $\widetilde{\bold{A}}^{[r,s]}$ as the $p$-adic completion of $\widetilde{\bold{A}}^+[\frac{p}{[\varepsilon-1]^r}, \frac{[\varepsilon-1]^s}{p}]$.
  We define $\widetilde{\bold{B}}^{[r,s]}:=\widetilde{\bold{A}}^{[r,s]}[1/p]$, $\bold{B}^+_{\mathrm{max}}:=\widetilde{\bold{B}}^{[0,\frac{p-1}{p}]}$, 
 $\widetilde{\bold{B}}^{\dagger,r}_{\mathrm{rig}}:=\cap_{r\leqq s<+\infty} \widetilde{\bold{B}}^{[r,s]}$ and $\widetilde{\bold{B}}^{\dagger}_{\mathrm{rig}}:=\cup_{r}\widetilde{\bold{B}}^{\dagger,r}_{\mathrm{rig}}$ 
 and $\widetilde{\bold{B}}^+_{\mathrm{rig}}:=\cap_{s<+\infty}\widetilde{\bold{B}}^{[0,s]}$. These rings are equipped with $G_K$-actions. The ring $\bold{B}^+_{\mathrm{max}}$ is stable by $\varphi$ and $\varphi$ induces isomorphisms $\widetilde{\bold{B}}^{[r,s]}\isom \widetilde{\bold{B}}^{[pr,ps]}$, 
 $\widetilde{\bold{B}}^{\dagger,r}_{\mathrm{rig}}\isom \widetilde{\bold{B}}^{\dagger,pr}_{\mathrm{rig}}$ and $\widetilde{\bold{B}}^{\dagger}_{\mathrm{rig}}\isom \widetilde{\bold{B}}^{\dagger}_{\mathrm{rig}}$. 
 For each $n\geqq 0$, we put $r_n:=p^{n-1}(p-1)=1/v_p(\zeta_{p^n}-1)$.
 Then, we have a natural injection $\widetilde{\bold{B}}^{[r_0,r_0]}\hookrightarrow \bold{B}^+_{\mathrm{dR}}$ and 
 a $G_K$-equivariant injection $\iota_n:\widetilde{\bold{B}}^{\dagger,r_n}_{\mathrm{rig}}\xrightarrow{\varphi^{-n}}\widetilde{\bold{B}}^{\dagger,r_0}_{\mathrm{rig}}\hookrightarrow \widetilde{\bold{B}}^{[r_0,r_0]}\hookrightarrow \bold{B}^+_{\mathrm{dR}}$ for each $n\geqq 0$. The element $t$ is an element of $\widetilde{\bold{B}}^+_{\mathrm{rig}}$ and, since
 we have $\widetilde{\bold{B}}^+_{\mathrm{rig}}\subseteq \bold{B}^+_{\mathrm{max}}$ and $\widetilde{\bold{B}}^{+}_{\mathrm{rig}}\subseteq \widetilde{\bold{B}}^{\dagger}_{\mathrm{rig}}$, $t$ is also contained in $\bold{B}^+_{\mathrm{max}}$ and $\widetilde{\bold{B}}^{\dagger}_{\mathrm{rig}}$. We define $\bold{B}_{\mathrm{max}}:=\bold{B}^+_{\mathrm{max}}[1/t]$ and $\bold{B}_e:=\bold{B}_{\mathrm{max}}^{\varphi=1}=(\widetilde{\bold{B}}^+_{\mathrm{rig}}[1/t])^{\varphi=1}$. One has $\bold{B}_e=(\widetilde{\bold{B}}^{\dagger}_{\mathrm{rig}}[1/t])^{\varphi=1}$ by Lemma 1.1.7 of \cite{Ber08a}.

We next recall the definition of the Robba ring $\bold{B}^{\dagger}_{\mathrm{rig},K}\subseteq \widetilde{\bold{B}}^{\dagger}_{\mathrm{rig}}$. See \cite{Ber02} for more details. We set $T:=[\varepsilon]-1\in \widetilde{\bold{A}}^+$.
We first assume that $K=F$ is unramified over $\mathbb{Q}_p$. 
For each $r\in \mathbb{Q}_{>0}$, we define a subring $\bold{B}^{\dagger, r}_{\mathrm{rig}, F}$ of 
$\widetilde{\bold{B}}^{\dagger,r}_{\mathrm{rig}}$ by 
$$\bold{B}^{\dagger, r}_{\mathrm{rig}, F}:=\{f(T):=\sum_{n\in \mathbb{Z}}a_nT^{n}| a_n\in F\, \text{and}\, f(T) \,\text{is convergent on }\, 
p^{-1/r}\leqq |T|_p<1\}.$$
We define $\bold{B}^{\dagger}_{\mathrm{rig},F}:=\cup_{r>0}\bold{B}^{\dagger,r}_{\mathrm{rig},F}\subseteq \widetilde{\bold{B}}^{\dagger}_{\mathrm{rig}}$. We note that $t=\mathrm{log}(1+T)\in \bold{B}^{\dagger,r}_{\mathrm{rig},\mathbb{Q}_p}$ for any $r$.
As a subring of $\widetilde{\bold{B}}^{\dagger,r}_{\mathrm{rig}}$, 
this definition of $\bold{B}^{\dagger, r}_{\mathrm{rig}, F}$ does not depend on the choice of $T$, i.e does not depend on the choice of $\{\zeta_{p^n}\}_{n\geqq 0}$.
For general $K$,
we put $F:=K_0$, $e_K:=[K_{\infty}:F_{\infty}]$ and denote by $K_0'\subseteq K_{\infty}$ the maximal unramified extension 
of $F$ in $K_{\infty}$. Then, the theory of fields of norm enables us to define the subring $\bold{B}^{\dagger}_{\mathrm{rig},K}$ of $\widetilde{\bold{B}}^{\dagger}_{\mathrm{rig}}$ as a finite Galois extension of $\bold{B}^{\dagger}_{\mathrm{rig},K_0'}$ of 
degree $e_K$ and there exist $r(K)\in \mathbb{Q}_{>0}$ and $\pi_K\in \bold{B}^{\dagger,r(K)}_{\mathrm{rig},K}$ such that 
$\bold{B}^{\dagger}_{\mathrm{rig},K}=\cup_{r\geqq r(K)}\bold{B}^{\dagger,r}_{\mathrm{rig}, K}$ is the union of the subrings
$$\bold{B}^{\dagger,r}_{\mathrm{rig}, K}:=\{f(\pi_K)=\sum_{n\in \mathbb{Z}}a_n\pi_K^n|\, a_n\in F' \,\text{and}\, f(X) \,\text{ is convergent on }\, 
p^{-1/r e_K }\leqq |X|_p <1 \}$$ 
of $\widetilde{\bold{B}}^{\dagger,r}_{\mathrm{rig}}$.
For each finite extension $K\subseteq K'(\subseteq \overline{K})$, $\bold{B}^{\dagger}_{\mathrm{rig},K}$ is 
a subring of $\bold{B}^{\dagger}_{\mathrm{rig},K'}$. For any $n\geqq 1$, 
we have an equality $\bold{B}^{\dagger}_{\mathrm{rig},K}=\bold{B}^{\dagger}_{\mathrm{rig},K_n}$. 
The ring $\bold{B}^{\dagger}_{\mathrm{rig},K}$ is stable by the actions of $\varphi$ and $G_K$ on 
$\widetilde{\bold{B}}^{\dagger}_{\mathrm{rig}}$. More precisely, we have $\varphi(\bold{B}^{\dagger,r}_{\mathrm{rig},K})\subseteq \bold{B}^{\dagger,pr}_{\mathrm{rig},K}$ and the action of $G_K$ factors through that of $\Gamma_K$. 
When $K=F$ is unramified, these actions are explicitly defined by the following formulae, 
 for $f(T)=\sum_{n\in \mathbb{Z}}a_nT^n\in \bold{B}^{\dagger}_{\mathrm{rig},F}$ and 
$\gamma\in \Gamma_F$,
$$\varphi(f(T)):=\sum_{n\in \mathbb{Z}}\varphi(a_n)((T+1)^p-1)^n,\,\,\,\, \, \gamma(f(T)):=\sum_{n\in \mathbb{Z}}a_n((T+1)^{\chi(\gamma)}-1)^n.$$
Next, we define a $\mathbb{Q}_p$-linear map $\psi:
\bold{B}^{\dagger}_{\mathrm{rig},K}\rightarrow \bold{B}^{\dagger}_{\mathrm{rig},K}$ as follows. It is known that $\bold{B}^{\dagger}_{\mathrm{rig},K}$ can be written as a 
direct sum $\bold{B}^{\dagger}_{\mathrm{rig},K}=\bigoplus_{i=0}^{p-1}(T+1)^{i}\varphi(\bold{B}^{\dagger}_{\mathrm{rig},K})$, so each element 
$x\in \bold{B}^{\dagger}_{\mathrm{rig},K}$ is uniquely written as $x=\sum_{i=0}^{p-1}(T+1)^i\varphi(x_i)$, then we define $\psi$ by 
$$\psi:\bold{B}^{\dagger}_{\mathrm{rig},K}\rightarrow \bold{B}^{\dagger}_{\mathrm{rig},K}:x=\sum_{i=0}^{p-1}(T+1)^i\varphi(x_i)\mapsto x_0.$$ 
This operator $\psi$ satisfies that 
$\psi\varphi=\mathrm{id}$ and $\psi$ is surjective and commutes with the action of $\Gamma_K$. 
More precisely, if we define  $n(K):=\mathrm{min}\{n| r_n\geqq r(K)\}$, then we have 
$\psi(\bold{B}^{\dagger,r_{n+1}}_{\mathrm{rig},K})= \bold{B}^{\dagger,r_n}_{\mathrm{rig},K}$ for any 
$n\geqq n(K)$. For each $n\geqq n(K)$, the restriction of $\iota_n:\widetilde{\bold{B}}^{\dagger,r_n}_{\mathrm{rig}}\hookrightarrow \bold{B}^+_{\mathrm{dR}}$ to $\bold{B}^{\dagger,r_n}_{\mathrm{rig},K}$ factors through $K_n[[t]]\subseteq \bold{B}^+_{\mathrm{dR}}$, i.e. $\iota_n$ induces a 
$\Gamma_K$-equivariant injection
$$\iota_n: \bold{B}^{\dagger, r_n}_{\mathrm{rig},K}\hookrightarrow K_n[[t]].$$ 
When $K=F$ is unramified over $\mathbb{Q}_p$, 
$\iota_n:\bold{B}^{\dagger,r_n}_{\mathrm{rig},F}\hookrightarrow F_n[[t]]$ is explicitly defined by 
$$\iota_n(\sum_{m\in \mathbb{Z}}a_mT^m)
:=\sum_{m\in \mathbb{Z}}\varphi^{-n}(a_m)(\zeta_{p^n}\mathrm{exp}(t/p^n)-1)^m.$$ 
One has the following commutative diagrams
$$
\begin{CD}
 \bold{B}^{\dagger,r_n}_{\mathrm{rig},K}@>\iota_n >> K_n[[t]]\,\,\,\,\,\,\,\,\,\,\,\,\,\,\, @.\,\,\bold{B}^{\dagger,r_{n+1}}_{\mathrm{rig},K}@>\iota_{n+1}>> K_{n+1}[[t]] \\
  @VV \varphi V    @VV\text{can} V  @VV \psi V    @VV \frac{1}{p}\mathrm{Tr}_{K_{n+1}/K_n} V \\
\bold{B}^{\dagger,r_{n+1}}_{\mathrm{rig},K}@>\iota_{n+1} >> K_{n+1}[[t]]\,\,\,\,\,\,\,\,\,\,\,\,\,\,\,\,@.\,\,\bold{B}^{\dagger,r_{n}}_{\mathrm{rig},K}@>\iota_{n} >> K_{n}[[t]]. 
  \end{CD}
  $$
  where $\mathrm{can}:K_n[[t]]\hookrightarrow K_{n+1}[[t]]$ is the canonical injection and $\frac{1}{p}\mathrm{Tr}_{K_{n+1}/K_n}$ is defined by 
  $$\frac{1}{p}\mathrm{Tr}_{K_{n+1}/K_n}:K_{n+1}[[t]]
  \rightarrow K_n[[t]]: \sum_{m= 0}^{\infty}a_mt^m\mapsto \sum_{m=0}^{\infty}\frac{1}{p}\mathrm{Tr}_{K_{n+1}/K_n}(a_m)t^m.$$

\begin{defn}
We say that $D$ is a $(\varphi,\Gamma_K)$-module over $\bold{B}^{\dagger}_{\mathrm{rig},K}$ if 
\begin{itemize}
\item[(1)]$D$ is a finite free $\bold{B}^{\dagger}_{\mathrm{rig},K}$-module,
\item[(2)]$D$ is equipped with a $\varphi$-semi-linear map $\varphi:D\rightarrow D$ such that the linearization map 
$\varphi^{*}(D):=\bold{B}^{\dagger}_{\mathrm{rig},K}\otimes_{\varphi,\bold{B}^{\dagger}_{\mathrm{rig},K}}D\rightarrow D:a\otimes x\mapsto a\varphi(x)$ is isomorphism,
\item[(3)]$D$ is equipped with a continuous semi-linear action of $\Gamma_K$ which commutes with $\varphi$,
\end{itemize}
where semi-linear means that $\varphi(ax)=\varphi(a)\varphi(x)$ and $\gamma(ax)=\gamma(a)\gamma(x)$ for 
any $a\in \bold{B}^{\dagger}_{\mathrm{rig},K}, x\in D$ and $\gamma\in \Gamma_K$.

\end{defn}

Let $D$ be a $(\varphi,\Gamma_K)$-module over $\bold{B}^{\dagger}_{\mathrm{rig},K}$. 
For each $k$, we denote by $D(k):=D\otimes_{\mathbb{Q}_p}\mathbb{Q}_p(k)$ the $k$-th Tate 
twist of $D$.
For each 
finite extension $L$ of $K$, the restriction $D|_L$ of $D$ to $L$, which is a $(\varphi,\Gamma_L)$-module 
over $\bold{B}^{\dagger}_{\mathrm{rig},L}$, is defined by 
$$D|_L:=\bold{B}^{\dagger}_{\mathrm{rig},L}\otimes_{\bold{B}^{\dagger}_{\mathrm{rig},K}}D$$ and the actions 
of $\varphi$ and $\Gamma_L(\subseteq \Gamma_K)$ are defined by $\varphi(a\otimes x):=\varphi(a)\otimes \varphi(x), \gamma(a\otimes x):=\gamma(a)\otimes \gamma(x)$ 
for any $a\in \bold{B}^{\dagger}_{\mathrm{rig},L}, x\in D$ and $\gamma\in \Gamma_L$.
We define the dual $D^{\lor}$ of $D$ by 
$$D^{\lor}:=\mathrm{Hom}_{\bold{B}^{\dagger}_{\mathrm{rig},K}}(D, \bold{B}^{\dagger}_{\mathrm{rig},K})$$ and, for any 
$f\in D^{\lor}$ and $\gamma\in \Gamma_K$, $\gamma(f)\in D^{\lor}$ is defined by $\gamma(f)(x):=\gamma(f(\gamma^{-1}x))$ 
for any $x\in D$, and $\varphi(f)\in D^{\lor}$ is defined by $\varphi(f)(\sum_{i=1}^m a_i\varphi(x_i)):=
\sum_{i=1}^ma_i\varphi(f(x_i))$ for any $x=\sum_{i=1}^ma_i\varphi(x_i)\in D$ ($a_i\in \bold{B}^{\dagger}_{\mathrm{rig},K}, x_i\in D$).
Let $D_1, D_2$ be $(\varphi,\Gamma_K)$-modules over $\bold{B}^{\dagger}_{\mathrm{rig},K}$. We define 
the tensor product $D_1\otimes D_2$ by 
$$D_1\otimes D_2:=D_1\otimes_{\bold{B}^{\dagger}_{\mathrm{rig},K}}D_2$$ as  a $\bold{B}^{\dagger}_{\mathrm{rig},K}$-module with $\varphi$ and  $\Gamma_K$ acting diagonally.
Let $D$ be a $(\varphi,\Gamma_K)$-module over $\bold{B}^{\dagger}_{\mathrm{rig},K}$ of rank $d$. By Theorem 1.3.3 of \cite{Ber08b}, 
there exists 
a $n(D)\geqq n(K)$ and there exists a unique finite free 
$\bold{B}^{\dagger,r_{n(D)}}_{\mathrm{rig},K}$-submodule $D^{(n(D))}\subseteq D$ of 
rank $d$ which satisfies 
\begin{itemize}
\item[(1)]$\bold{B}^{\dagger}_{\mathrm{rig},K}\otimes_{\bold{B}^{\dagger,r_{n(D)}}_{\mathrm{rig},K}}D^{(n(D))}=D$,
\item[(2)]if we put $D^{(n)}:=\bold{B}^{\dagger,r_n}_{\mathrm{rig},K}\otimes_{\bold{B}^{\dagger,r_{n(D)}}_{\mathrm{rig},K}}D^{(n(D))}$ for  
each $n\geqq n(D)$, then $\varphi(D^{(n)})\subseteq D^{(n+1)}$ and the natural map 
$\bold{B}^{\dagger, r_{n+1}}_{\mathrm{rig}, K}\otimes_{\varphi, \bold{B}^{\dagger,r_n}_{\mathrm{rig},K}}D^{(n)}
\rightarrow D^{(n+1)}:a\otimes x\mapsto a\varphi(x)$ is isomorphism for any $n\geqq n(D)$.
\end{itemize}
Uniqueness of $D^{(n)}$ implies that  $D^{(n)}$ is preserved by $\Gamma_K$-action for any $n\geqq n(D)$.

Using $D^{(n)}$, we define $\bold{D}^+_{\mathrm{dif}}(D)$ and $\bold{D}_{\mathrm{dif}}(D)$ as follows.
For each $n\geqq n(D)$, we put 
$$\bold{D}^+_{\mathrm{dif}, n}(D):= K_n[[t]]\otimes_{\iota_n, \bold{B}^{\dagger,r_n}_{\mathrm{rig},K}} D^{(n)}\,\,\,(\text{resp}.\,\bold{D}_{\mathrm{dif},n}(D):=K_n((t))\otimes_{K_n[[t]]}\bold{D}^+_{\mathrm{dif},n}(D)),$$
which is a finite free $K_n[[t]]$\,\, (resp.$K_n((t))$)-module of rank $d$ with a semi-linear $\Gamma_K$-action.
Define a transition map 
$$\bold{D}^+_{\mathrm{dif},n}(D) \hookrightarrow
\bold{D}^+_{\mathrm{dif},n+1}(D) : f(t)\otimes x\mapsto f(t)\otimes \varphi(x),$$ and define a map 
$\bold{D}_{\mathrm{dif}, n}(D)\hookrightarrow \bold{D}_{\mathrm{dif}, n+1}(D)$ in the same way. 
Using these transition maps, we define 
$$\bold{D}^+_{\mathrm{dif}}(D):=\varinjlim_{n} \bold{D}^+_{\mathrm{dif}, n}(D)\,\, \, \,\,\,(\text{resp}.\, \bold{D}_{\mathrm{dif}}(D):=\varinjlim_n \bold{D}_{\mathrm{dif},n}(D)),$$
 this is a free $K_{\infty}[[t]]:=\cup_{n=1}^{\infty}K_n[[t]]$\,\,\,(resp. $K_{\infty}((t)):=\cup_{n=1}^{\infty} K_n((t))$)-module of rank $d$
 with a semi-linear $\Gamma_K$-action.
For each $n\geqq n(D)$, define a canonical $\Gamma_K$-equivariant injection 
$$\iota_n: D^{(n)}\hookrightarrow \bold{D}^+_{\mathrm{dif},n}(D): x\mapsto 1\otimes x.$$ 

\subsection{cohomologies of $(\varphi,\Gamma)$-modules}

In this subsection, we recall the definitions of some cohomology theories associated to $(\varphi,\Gamma)$-modules and the 
fundamental properties of them proved by Liu (\cite{Li08}). 

Let $\Delta_K\subseteq \Gamma_K$ be the $p$-torsion subgroup of $\Gamma_K$ 
which is trivial if $p\not=2$ and  at largest cyclic of order two if $p=2$. Choose 
$\gamma_K\in \Gamma_K$ whose image in $\Gamma_K/\Delta_K$ is a topological generator (this choice of $\Delta_K$ is useful for explicit formulas, but if desired one can reformulate everything to eliminate this choice).

For a $\Delta_K$-module $M$, we put $M^{\Delta_K}:=\{x\in M| \gamma'(x)=x $ for all $\gamma'\in \Delta_K \}$. 
For a $\mathbb{Z}[\Gamma_K]$-module $M$, we define a complex $C^{\bullet}_{\gamma_K}(M)$ concentrated in degree $[0,1]$ by 
$$C^{\bullet}_{\gamma_K}(M): [M^{\Delta_K}\xrightarrow{\gamma_K-1} M^{\Delta_K}] .$$
For a $\mathbb{Z}[\Gamma_K]$-module $M$ with a $\varphi$-action which commutes with the action of $\Gamma_K$,  
we define a complex $C^{\bullet}_{\varphi,\gamma_K}(M)$ concentrated in degree $[0,2]$ 
by 
$$C^{\bullet}_{\varphi,\gamma_K}(M):[M^{\Delta_K}\xrightarrow{d_1}M^{\Delta_K}\oplus M^{\Delta_K}\xrightarrow{d_2} M^{\Delta_K}]$$
with $d_1(x):=((\gamma_K-1)x,(\varphi-1)x)$ and $d_2(x,y):=(\varphi-1)x-(\gamma_K-1)y$.

Let $D$ be a $(\varphi,\Gamma_K)$-module over $\bold{B}^{\dagger}_{\mathrm{rig},K}$. We put $D[1/t]:=\bold{B}^{\dagger}_{\mathrm{rig},K}[1/t]\otimes_{\bold{B}^{\dagger}_{\mathrm{rig},K}}D$.
For each $q\in \mathbb{Z}_{\geqq 0}$, we define 
$$\mathrm{H}^{q}(K, D):=\mathrm{H}^q(C^{\bullet}_{\varphi,\gamma_K}(D)), \,\,\,\mathrm{H}^{q}(K, D[1/t])):=\mathrm{H}^q(C^{\bullet}_{\varphi,\gamma_K}(D[1/t]))$$ 
and 
$$\mathrm{H}^q(K, \bold{D}^+_{\mathrm{dif}}(D)):=\mathrm{H}^q(C^{\bullet}_{\gamma_K}(\bold{D}^+_{\mathrm{dif}}(D))), \,\,\,\mathrm{H}^{q}(K, \bold{D}_{\mathrm{dif}}(D))
:=\mathrm{H}^q(C^{\bullet}_{\gamma_K}(\bold{D}_{\mathrm{dif}}(D))).$$

These definitions are independent of the choice of $\gamma_K$. Namely, if $\gamma'_{K}\in \Gamma_K$  
is another one such that the image in $\Gamma_K/\Delta_K$ is a topological generator,  
then 
we have $\frac{\gamma_K'-1}{\gamma_K-1}\in \mathbb{Z}_p[[\Gamma_K/\Delta_K]]$  and 
have  the canonical isomorphism 
$$\mathrm{H}^q(C^{\bullet}_{\varphi,\gamma_K}(D))\isom \mathrm{H}^q(C^{\bullet}_{\varphi,\gamma'_K}(D))$$ given by the map which is 
induced by the following map of complexes
$$
\begin{CD}
C^{\bullet}_{\varphi,\gamma_K}(D):@. [D^{\Delta_K} @> d_1 >> D^{\Delta_K}\oplus D^{\Delta_K} @ > d_2 >> D^{\Delta_K}] \\
  @.@VV \mathrm{id} V    @VV \frac{\gamma_K'-1}{\gamma_K-1}\oplus \mathrm{id} V @VV \frac{\gamma_K'-1}{\gamma_K-1} V \\
  C^{\bullet}_{\varphi,\gamma'_K}(D):@. [D^{\Delta_K}@> d_1 >> D^{\Delta_K}\oplus D^{\Delta_K} @> d_2 >> D^{\Delta_K}],
  \end{CD}
  $$
  where we note that the $\mathbb{Z}_p[\Gamma_K/\Delta_K]$-module structure on $D^{\Delta_K}$ uniquely extends to a continuous $\mathbb{Z}_p[[\Gamma_K/\Delta_K]]$-module structure.
  
  For $(\varphi,\Gamma_K)$-modules $D_1, D_2$  over $\bold{B}^{\dagger}_{\mathrm{rig},K}$, we can define a cup product paring
  $$\cup:\mathrm{H}^{q_1}(K, D_1)\times \mathrm{H}^{q_2}(K, D_2)\rightarrow \mathrm{H}^{q_1+q_2}(K, D_1\otimes D_2). $$ 
  See $\S$ 2.1 of \cite{Li08} for the definition. 
  When $(q_1, q_2)=(0,1), (1,1)$, the paring $\cup$ are defined by
  $$\mathrm{H}^0(K, D_1)\times \mathrm{H}^1(K, D_2)\rightarrow \mathrm{H}^1(K, D_1\otimes D_2):(a, [x,y])\mapsto [a\otimes x, a\otimes y],$$
  $$\mathrm{H}^1(K, D_1)\times \mathrm{H}^1(K, D_2)\rightarrow \mathrm{H}^2(K, D_1\otimes D_2):([x,y], [x',y'])\mapsto [y\otimes \varphi(x')- x\otimes \gamma(y')].$$

 The following theorem was proved by Liu (\cite{Li08}) by reducing to the results of Herr (\cite{Her98}, \cite{Her01}) in the \'etale case.
 \begin{thm}
 Let $D$ be a $(\varphi,\Gamma_K)$-module over $\bold{B}^{\dagger}_{\mathrm{rig},K}$. Then 
 $\mathrm{H}^q(K, D)$ satisfies the following;
 \begin{itemize}
 \item[(0)]$\mathrm{H}^q(K, D)=0$ if $q\not= 0,1,2$,
 \item[(1)]for any $q$, $\mathrm{H}^q(K, D)$ is a finite dimensional $\mathbb{Q}_p$-vector space,
 \item[(2)]$\sum_{q=0}^2 (-1)^q\mathrm{dim}_{\mathbb{Q}_p}\mathrm{H}^q(K, D)=-[K:\mathbb{Q}_p]\mathrm{rank}(D)$,
 \item[(3)]we have  a canonical isomorphism $f_{\mathrm{tr}}:\mathrm{H}^2(K,\bold{B}^{\dagger}_{\mathrm{rig},K}(1))\isom \mathbb{Q}_p$ and 
 the following pairing $<,>$ is perfect for each $q=0,1,2$,
 $$<,>:\mathrm{H}^q(K, D)\times\mathrm{H}^{2-q}(K, D^{\lor}(1))\xrightarrow{\cup}\mathrm{H}^2(K, D\otimes D^{\lor}(1))\xrightarrow{\mathrm{ev}}\mathrm{H}^2(K, \bold{B}^{\dagger}_{\mathrm{rig},K}(1))\xrightarrow{f_{\mathrm{tr}}} \mathbb{Q}_p,$$
\end{itemize}
where
$\mathrm{ev}:\mathrm{H}^2(K, D\otimes D^{\lor}(1))\xrightarrow{\mathrm{ev}}\mathrm{H}^2(K, \bold{B}^{\dagger}_{\mathrm{rig},K}(1))$ is the map induced by the 
evaluation map $D\otimes D^{\lor}(1)\rightarrow \bold{B}^{\dagger}_{\mathrm{rig},K}(1):x\otimes (f\otimes e_1)\mapsto f(x)\otimes e_1$.
 
 \end{thm}
 \begin{proof}
 See Theorem 0.2 of \cite{Li08}
 \end{proof}
 \begin{rem}\label{normalization}
 We remark that Liu proved the existence of functorial comparison isomorphisms 
 $\mathrm{H}^q(K, V)\isom \mathrm{H}^q(K, D(V))$ for all the $p$-adic representations $V$ of $G_K$. 
 Then, the isomorphism $f_{\mathrm{tr}}$ is defined as the composition of 
 the inverse of the comparison isomorphism
  $\mathrm{H}^2(K, \mathbb{Q}_p(1))\isom \mathrm{H}^2(K, D(\mathbb{Q}_p(1)))=
  \mathrm{H}^2(K, \bold{B}^{\dagger}_{\mathrm{rig},K}(1))$ with Tate's trace $f'_{\mathrm{tr}}:\mathrm{H}^2(K,\mathbb{Q}_p(1))\isom \mathbb{Q}_p$.
  In this article, we normalize the isomorphism $f'_{\mathrm{tr}}:\mathrm{H}^2(K, \mathbb{Q}(1))\isom \mathbb{Q}_p$ such that Tate's paring 
 $$<,>:\mathrm{H}^1(\mathbb{Q}_p, \mathbb{Q}_p(1))\times \mathrm{H}^1(\mathbb{Q}_p,\mathbb{Q}_p)\xrightarrow{\cup} \mathrm{H}^2(\mathbb{Q}_p, \mathbb{Q}_p(1))\xrightarrow{f'_{\mathrm{tr}}}\mathbb{Q}_p$$
  satisfies that $<\kappa(a),\tau>=\tau(\mathrm{rec}_{\mathbb{Q}_p}(a))$ for any $a\in \mathbb{Q}_p^{\times}$ and $\tau\in \mathrm{Hom}(G^{\mathrm{ab}}_{\mathbb{Q}_p}, \mathbb{Q}_p)
 =\mathrm{H}^1(\mathbb{Q}_p,\mathbb{Q}_p)$, where $\kappa:\mathbb{Q}_p^{\times}\rightarrow 
 \mathrm{H}^1(\mathbb{Q}_p, \mathbb{Q}_p(1))$ is the Kummer map and 
 $\mathrm{rec}_{\mathbb{Q}_p}:\mathbb{Q}_p^{\times}\rightarrow G_{\mathbb{Q}_p}^{\mathrm{ab}}$ 
 is the reciprocity map of local class field theory.

 \end{rem}
 
 It is important to define the cohomology $\mathrm{H}^q(K, D)$ using $\psi$ instead of $\varphi$, which we recall below.
We define a complex $C^{\bullet}_{\psi,\gamma_K}(D)$ concentrated in degree $[0,2]$ by 
$$C^{\bullet}_{\psi,\gamma_K}(D):[D^{\Delta_K}\xrightarrow{d'_1}D^{\Delta_K}\oplus D^{\Delta_K}\xrightarrow{d'_2}
D^{\Delta_K}]$$
with $d'_1(x):=((\gamma_K-1)x,(\psi-1)x)$ and $d'_2(x,y):=(\psi-1)x-(\gamma_K-1)y$. 
We define a surjective map $C^{\bullet}_{\varphi,\gamma_K}(D)\rightarrow C^{\bullet}_{\psi,\gamma_K}(D)$ of complexes by 
$$
\begin{CD}
C^{\bullet}_{\varphi,\gamma_K}(D):@. [ D^{\Delta_K} @> d_1 >> D^{\Delta_K}\oplus D^{\Delta_K} @ > d_2 >> D^{\Delta_K}] \\
  @.@VV \mathrm{id} V    @VV \mathrm{id}\oplus(-\psi) V @VV -\psi V \\
  C^{\bullet}_{\psi,\gamma_K}(D): @. [D^{\Delta_K}@> d'_1 >> D^{\Delta_K}\oplus D^{\Delta_K} @> d'_2 >> D^{\Delta_K}]. 
  \end{CD}
  $$
  The kernel of this map is the complex $[0\rightarrow 0\oplus D^{\Delta_K,\psi=0}\xrightarrow{0\oplus (\gamma_K-1)}D^{\Delta_K,\psi=0}]$. Concerning this complex, 
   we have the following theorem.
  \begin{thm} \label{2.3}
 The map  $D^{\Delta_K,\psi=0}\xrightarrow{\gamma_K-1}D^{\Delta_K,\psi=0}$ is isomorphism.
  In particular, the map 
  $C^{\bullet}_{\varphi,\gamma_K}(D)\rightarrow C^{\bullet}_{\psi,\gamma_K}(D)$ defined above is quasi isomorphism.
      
  \end{thm}
  \begin{proof}
  For example, see Lemma 2.4 of \cite{Li08} in the \'etale case and see Theorem 2.6 of \cite{Po12b} for general case.
  \end{proof}

  Next, we recall the definition of crystalline or de Rham $(\varphi,\Gamma)$-modules.
  \begin{defn}
  For a $(\varphi,\Gamma_K)$-module $D$ over $\bold{B}^{\dagger}_{\mathrm{rig},K}$,  we define 
  $$ \bold{D}^K_{\mathrm{crys}}(D):=D[1/t]^{\Gamma_K=1},\,\,\, \bold{D}^K_{\mathrm{dR}}(D):=\bold{D}_{\mathrm{dif}}(D)^{\Gamma_K=1}.$$
  We define a decreasing filtration on $\bold{D}^{K}_{\mathrm{dR}}(D)$ by 
  $$\mathrm{Fil}^i\bold{D}^K_{\mathrm{dR}}(D):=\bold{D}^K_{\mathrm{dR}}(D)\cap t^i \bold{D}^+_{\mathrm{dif}}(D)\subseteq \bold{D}_{\mathrm{dif}}(D)$$ for 
  $i\in \mathbb{Z}$.
  
  \end{defn}
  Using cohomologies which we defined above, we have equalities 
  $$\bold{D}^K_{\mathrm{dR}}(D)=\mathrm{H}^0(K, \bold{D}_{\mathrm{dif}}(D)), \,\,\,\,
  \mathrm{Fil}^0\bold{D}^K_{\mathrm{dR}}(D)=\mathrm{H}^0(K, \bold{D}^+_{\mathrm{dif}}(D)), $$and 
  $$\bold{D}^K_{\mathrm{crys}}(D)^{\varphi=1}=\mathrm{H}^0(K, D[1/t]).$$
  As in the case of $p$-adic Galois representations of $G_K$, we have inequalities 
  $$\mathrm{dim}_{K_0}\bold{D}^K_{\mathrm{crys}}(D)\leqq\mathrm{dim}_K \bold{D}^K_{\mathrm{dR}}(D)\leqq \mathrm{rank}D.$$
   
  \begin{defn}
  Let $D$ be a $(\varphi,\Gamma_K)$-module over $\bold{B}^{\dagger}_{\mathrm{rig},K}$. 
  We say that $D$ is crystalline (resp. de Rham) if an equality 
  $\mathrm{dim}_{K_0}\bold{D}^K_{\mathrm{crys}}(D)=\mathrm{rank}(D)$ (resp. $\mathrm{dim}_K\bold{D}^K_{\mathrm{dR}}(D)=
  \mathrm{rank}(D))$ holds. We say that $D$ is potentially crystalline if there exists a finite extension $L$ of $K$ such that 
  $D|_L$ is crystalline $(\varphi,\Gamma_L)$-module over $\bold{B}^{\dagger}_{\mathrm{rig},L}$. 
  \end{defn}
  \begin{defn}
  Let $D$ be a de Rham $(\varphi,\Gamma_K)$-module over $\bold{B}^{\dagger}_{\mathrm{rig},K}$. 
  We call the set $\{h\in \mathbb{Z}| \mathrm{Fil}^{-h}\bold{D}^K_{\mathrm{dR}}(D)/\mathrm{Fil}^{-h+1}\bold{D}^K_{\mathrm{dR}}(D)\not=0\}$ 
  the Hodge-Tate weights of $D$.
  
  \end{defn}
  If $D$ is crystalline then $D$ is also de Rham by the above inequalities. Because 
  potentially de Rham implies de Rham by Hilbert 90, 
  if $D$ is potentially crystalline, then $D$ is de Rham. 
  If $D$ is potentially crystalline such that 
  $D|_L$ is crystalline for a finite Galois extension $L$ of $K$, 
  then $\bold{D}^L_{\mathrm{crys}}(D):=\bold{D}^L_{\mathrm{crys}}(D|_L)$ is naturally equipped with actions of 
   $\varphi$ and of $\mathrm{Gal}(L/K)$ and we have a 
   natural isomorphism $L\otimes_{L_0}\bold{D}^L_{\mathrm{crys}}(D)\isom \bold{D}^L_{\mathrm{dR}}(D)=L\otimes_K\bold{D}^K_{\mathrm{dR}}(D)$, i.e. 
   $\bold{D}^L_{\mathrm{crys}}(D)$ is naturally equipped with a structure of filtered $(\varphi,\mathrm{Gal}(L/K))$-module.
   
   \subsection{Bloch-Kato's exponential map for $(\varphi,\Gamma)$-modules}
   This subsection is the main part of this section. We define a map 
   $\mathrm{exp}_{K,D}:\bold{D}^K_{\mathrm{dR}}(D)\rightarrow \mathrm{H}^1(K, D)$, which is the $(\varphi,\Gamma)$-module analogue of 
   Bloch-Kato's exponential map.
  
  The following is the main theorem of this section, which is the $(\varphi,\Gamma)$-module 
  analogue of the long exact sequence, for a $p$-adic representation $V$ of $G_K$,
  \[
\begin{array}{ll}
0\rightarrow& \mathrm{H}^0(K, V) \rightarrow  \mathrm{H}^0(K, \bold{B}_e\otimes_{\mathbb{Q}_p}V)\oplus \mathrm{H}^0(K, \bold{B}^+_{\mathrm{dR}}\otimes_{\mathbb{Q}_p}V)
\rightarrow  \mathrm{H}^0(K, \bold{B}_{\mathrm{dR}}\otimes_{\mathbb{Q}_p}V) \\
  \xrightarrow{\delta_{1,V}}&\mathrm{H}^1(K, V)\rightarrow \mathrm{H}^1(K, \bold{B}_e\otimes_{\mathbb{Q}_p}V)\oplus \mathrm{H}^1(K, \bold{B}^+_{\mathrm{dR}}\otimes_{\mathbb{Q}_p}V)
\rightarrow \mathrm{H}^1(K, \bold{B}_{\mathrm{dR}}\otimes_{\mathbb{Q}_p}V) \\
  \xrightarrow{\delta_{2,V}}&  \mathrm{H}^2(K, V)\rightarrow \mathrm{H}^2(K, \bold{B}_e\otimes_{\mathbb{Q}_p}V) \rightarrow 0

\end{array}
  \]which is obtained by taking the cohomology long exact sequence 
  associated to the Bloch-Kato's fundamental short exact sequence 
  $$0\rightarrow V\rightarrow \bold{B}_e\otimes_{\mathbb{Q}_p}V\oplus \bold{B}^+_{\mathrm{dR}}\otimes_{\mathbb{Q}_p}V
  \rightarrow \bold{B}_{\mathrm{dR}}\otimes_{\mathbb{Q}_p}V\rightarrow 0.$$
  
  See $\S$ 2.5 for the comparison of 
  the above exact sequence with the below exact sequence.

\begin{thm}\label{exp}
Let $D$ be a $(\varphi,\Gamma_K)$-module over $\bold{B}^{\dagger}_{\mathrm{rig},K}$.
Then there exists a canonical functorial long exact sequence
\begin{align*}
0\rightarrow& \mathrm{H}^0(K, D) \rightarrow  \mathrm{H}^0(K, D[1/t])\oplus \mathrm{H}^0(K, \bold{D}^+_{\mathrm{dif}}(D))
\rightarrow  \mathrm{H}^0(K, \bold{D}_{\mathrm{dif}}(D)) \\
  \xrightarrow{\delta_{1,D}}&\mathrm{H}^1(K, D)\rightarrow \mathrm{H}^1(K, D[1/t])\oplus \mathrm{H}^1(K, \bold{D}^+_{\mathrm{dif}}(D))
\rightarrow \mathrm{H}^1(K, \bold{D}_{\mathrm{dif}}(D)) \\
  \xrightarrow{\delta_{2,D}} & \mathrm{H}^2(K, D)\rightarrow \mathrm{H}^2(K, D[1/t]) \rightarrow 0.
  \end{align*}

\end{thm}

\begin{proof}

To construct this exact sequence, we need to define  some more complexes. 
For each $n\geqq n(D)$, we define a complex with degree in $[0,2]$
$$\tilde{C}^{\bullet}_{\varphi,\gamma_K}(D^{(n)}): [(D^{(n)})^{\Delta_K}\xrightarrow{d_1} (D^{(n)})^{\Delta_K}\oplus (D^{(n+1)})^{\Delta_K}\xrightarrow{d_2} (D^{(n+1)})^{\Delta_K}]$$ 
with $d_1(x):=((\gamma_K-1)x,(\varphi-1)x)$ and $d_2((x,y)):=(\varphi-1)x-(\gamma_K-1)y$.
Define $\tilde{C}^{\bullet}_{\varphi,\gamma_K}(D^{(n)}[1/t])$ 
 in the same way. 
 We also define $\tilde{C}^{\bullet}_{\varphi,\gamma_K}(\bold{D}^+_{\mathrm{dif},n}(D))$ with degree in $[0.2]$ by 

$$[\prod_{m\geqq n}\bold{D}^+_{\mathrm{dif},m}(D)^{\Delta_K}
\xrightarrow{d_1'}
\prod_{m\geqq n}\bold{D}^+_{\mathrm{dif},m}(D)^{\Delta_K}\oplus \prod_{m\geqq n+1}\bold{D}^+_{\mathrm{dif},m}(D)^{\Delta_K}
\xrightarrow{d_2'}
\prod_{m\geqq n+1}\bold{D}^+_{\mathrm{dif}, m}(D)^{\Delta_K}]$$
with 
$$d_1'((x_m)_{m\geqq n}):=(((\gamma_K-1)x_m)_{m\geqq n}, (x_{m-1}-x_m)_{m\geqq n+1})$$
 and 
$$d_2'((x_m)_{m\geqq n}, (y_m)_{m\geqq n+1}):= ((x_{m-1}-x_m)-(\gamma_K-1)y_m)_{m\geqq n+1}.$$
Define $\tilde{C}^{\bullet}_{\varphi,\gamma_K}(\bold{D}_{\mathrm{dif},n}(D))=\cup_{k\geqq 
0}\tilde{C}^{\bullet}_{\varphi,\gamma_K}(\frac{1}{t^k}\bold{D}^+_{\mathrm{dif},n}(D))$.
We first prove the following lemma.
\begin{lemma}\label{2.8}
Let $D$ be a $(\varphi,\Gamma_K)$-module over $\bold{B}^{\dagger}_{\mathrm{rig},K}$. 
The following sequence is exact for any $n\geqq n(D)$
$$0\rightarrow (D^{(n)})^{\Delta_K}\xrightarrow{f_1}
(D^{(n)}[1/t])^{\Delta_K}\oplus \prod_{m\geqq n} \bold{D}^+_{\mathrm{dif},m}(D)^{\Delta_K}\xrightarrow{f_2}
\cup_{k\geqq 0}\prod_{m\geqq n} (\frac{1}{t^k}\bold{D}_{\mathrm{dif},m}(D))^{\Delta_K}\rightarrow 0 $$ 
where $f_1(x):=(x, (\iota_m(x))_{m\geqq n})$ and $f_2(x, (y_m)_{m\geqq n}):= (\iota_m(x)-y_m)_{m\geqq n}$.

\end{lemma}
\begin{proof}
Because the functor $M\mapsto M^{\Delta_K}$ is exact for $\mathbb{Q}[\Delta_K]$-modules, it suffices to show that the sequence 
$$0\rightarrow D^{(n)}\xrightarrow{f_1}
D^{(n)}[1/t]\oplus \prod_{m\geqq n} \bold{D}^+_{\mathrm{dif},m}(D)\xrightarrow{f_2}
\cup_{k\geqq 0}\frac{1}{t^k}\prod_{m\geqq n} \bold{D}^+_{\mathrm{dif},m}(D)\rightarrow 0 $$ is exact. 

That $f_1$ is injective and that $f_2\circ f_1=0$ 
are trivial by definition. 

If $(x, (y_m)_{m\geqq n})\in \mathrm{Ker}(f_2)$, then we have $\iota_m(x)=y_m\in \bold{D}^+_{\mathrm{dif}, m}(D)$ for any 
$m\geqq n$. Hence we have $x\in D^{(n)}$ by Proposition 1.2.2 of \cite{Ber08b} and so we have 
$(x, (y_m)_{m\geqq n})=f_1(x)\in \mathrm{Im}(f_1)$. 

Finally, we prove that $f_2$ is surjective. Because we have 
$D^{(n)}[1/t]=\cup_{k=1}^{\infty}\frac{1}{t^k}D^{(n)}$,  it suffices to show that the natural map 
$$\frac{1}{t^k}D^{(n)}\rightarrow \prod_{m\geqq n} \frac{1}{t^k}\bold{D}^{ +}_{\mathrm{dif},m}(D)/
\bold{D}^{ +}_{\mathrm{dif},m}(D)
: x \mapsto ( \overline{\iota_{m}(x)})_{m\geqq n}$$ is surjective for any $k\geqq 1$. Moreover,  
twisting by $t^k$, it suffices to show that  the map 
$$D^{(n)}\rightarrow \prod_{m\geqq n} \bold{D}^{+}_{\mathrm{dif},m}(D)/t^k\bold{D}^{+}_{\mathrm{dif},m}(D)
: x \mapsto ( \overline{\iota_{m}(x)})_{m\geqq n}$$ is surjective for any $k\geqq 1$. By  induction and by d\'evissage, it suffices 
to show that this map is surjective for $k=1$. Let $\{e_i\}_{i=1}^d$ be a basis of $D$ such that 
$D^{(n)}= \bold{B}^{\dagger, r_n}_{\mathrm{rig}, K}e_1\oplus \cdots \oplus \bold{B}^{\dagger, r_n}_{\mathrm{rig}, K}e_d$ for any $n\geqq n(D)$.
Then $\{\overline{\iota_m(e_i)}\}_{i=1}^d$ is a $K_m$-basis of $\bold{D}_{\mathrm{Sen},m}(D):=
\bold{D}^{+}_{\mathrm{dif},m}(D)/t \bold{D}^{+}_{\mathrm{dif},m}(D)$ for any $m\geqq n$ by Lemma 4.9 of \cite{Ber02}. Hence, 
for any $(y_m)_{m\geqq n}\in \prod_{m\geqq n} \bold{D}_{\mathrm{Sen},m}(D)$, 
there exist $a_{m,i}\in K_m $ ($m\geqq n, 1\leqq i\leqq d$) 
such that $y_m= \sum_{i=1}^d a_{m,i} \overline{\iota_m(e_i)}$ for any $m\geqq n$. 
Because the natural map $\bold{B}^{\dagger, r_n}_{\mathrm{rig}, K}/t \rightarrow \prod_{m\geqq n} K_m:x\mapsto (\overline{\iota_m(x)})_{m\geqq n} $ is 
isomorphism,  there exists $\{a_i\}_{1\leqq i\leqq d}\subseteq \bold{B}^{\dagger, r_n}_{\mathrm{rig}, K}$ such that 
$\overline{\iota_m(a_i)}=a_{m, i}$ for any $m\geqq n$ and $i$. Then we have $\overline{\iota_m(\sum_{i=1}^d a_ie_i)}=y_m$ for any $m\geqq n$. This
proves the surjection of $f_2$, hence proves the lemma.

\end{proof}

It is easy to see that the maps $f_1, f_2$ commute with the differentials of $\tilde{C}^{\bullet}_{\varphi,\gamma_K}(-)$.  Hence, for each 
$n\geqq n(D)$,
we obtain the following short exact sequence of complexes
$$0\rightarrow \tilde{C}^{\bullet}_{\varphi,\gamma_K}(D^{(n)})\xrightarrow{f_1} \tilde{C}^{\bullet}_{\varphi,\gamma_K}(D^{(n)}[1/t])\oplus 
\tilde{C}^{\bullet}_{\varphi,\gamma_K}(\bold{D}^+_{\mathrm{dif}, n}(D))\xrightarrow{f_2} \tilde{C}^{\bullet}_{\varphi,\gamma_K}(\bold{D}_{\mathrm{dif},n}(D))
\rightarrow 0.$$

We define a transition map
$$\tilde{C}^{\bullet}_{\varphi,\gamma_K}(\bold{D}^+_{\mathrm{dif}, n})\rightarrow 
\tilde{C}^{\bullet}_{\varphi,\gamma_K}(\bold{D}^+_{\mathrm{dif}, n+1}(D))$$ which is induced by the map 
$$\prod_{m\geqq n} \bold{D}^+_{\mathrm{dif}, m}(D)\rightarrow \prod_{m\geqq  n+1} \bold{D}^+_{\mathrm{dif}, m}:(x_m)_{m\geqq n}\mapsto (x_m)_{m\geqq n+1}.$$ 
We similarly define $\tilde{C}^{\bullet}_{\varphi,\gamma_K}(\bold{D}_{\mathrm{dif}, n})\rightarrow 
\tilde{C}^{\bullet}_{\varphi,\gamma_K}(\bold{D}_{\mathrm{dif}, n+1}(D))$. Taking the inductive limit with respect to $n\geqq n(D)$, we obtain 
the following short exact sequence of complexes
$$0\rightarrow C^{\bullet}_{\varphi,\gamma_K}(D)\rightarrow C^{\bullet}_{\varphi,\gamma_K}(D[1/t])\oplus 
\varinjlim_n \tilde{C}^{\bullet}_{\varphi,\gamma_K}(\bold{D}^+_{\mathrm{dif}, n}(D))\rightarrow 
\varinjlim_n \tilde{C}^{\bullet}_{\varphi, \gamma_K}(\bold{D}_{\mathrm{dif}, n}(D))\rightarrow 0$$
because we have $\varinjlim_{n}\tilde{C}^{\bullet}_{\varphi,\gamma_K}(D_1^{(n)})\isom C^{\bullet}_{\varphi,\gamma_K}(D_1)$ for 
$D_1=D, D[1/t]$.
Taking the cohomology long exact sequence, we obtain the following long exact sequence

\[
\begin{array}{ll}
0\rightarrow& \mathrm{H}^0(K, D) \rightarrow  \mathrm{H}^0(K, D[1/t])\oplus \mathrm{H}^0(\varinjlim_n \tilde{C}^{\bullet}_{\varphi,\gamma_K}(\bold{D}^+_{\mathrm{dif},n}(D)))
\rightarrow  \mathrm{H}^0(\varinjlim_n \tilde{C}^{\bullet}_{\varphi,\gamma_K}(\bold{D}_{\mathrm{dif},n}(D))) \\
   \xrightarrow{\delta_{1,D}}&\mathrm{H}^1(K, D)\rightarrow \mathrm{H}^1(K, D[1/t])\oplus \mathrm{H}^1(\varinjlim_n \tilde{C}^{\bullet}_{\varphi,\gamma_K}(\bold{D}^+_{\mathrm{dif},n}(D)))
\rightarrow \mathrm{H}^1(\varinjlim_n \tilde{C}^{\bullet}_{\varphi,\gamma_K}(\bold{D}_{\mathrm{dif},n}(D))) \\
  \xrightarrow{\delta_{2,D}} & \mathrm{H}^2(K, D)\rightarrow  \mathrm{H}^2(K, D[1/t]) \oplus\mathrm{H}^2(\varinjlim_n \tilde{C}^{\bullet}_{\varphi,\gamma_K}(\bold{D}^+_{\mathrm{dif},n}(D)))\rightarrow \mathrm{H}^2(\varinjlim_n \tilde{C}^{\bullet}_{\varphi,\gamma_K}(\bold{D}_{\mathrm{dif},n}(D))) \\
 \,\,\,\,\, \rightarrow & 0 .

\end{array}
  \]

Next, for $\bold{D}^{(+)}_{\mathrm{dif},n}(D)=\bold{D}^+_{\mathrm{dif},n}(D), \bold{D}_{\mathrm{dif},n}(D)$, 
define a map of complexes 
$$C^{\bullet}_{\gamma_K}(\bold{D}^{(+)}_{\mathrm{dif}, n}(D))\rightarrow \tilde{C}^{\bullet}_{\varphi,\gamma_K}(\bold{D}^{(+)}_{\mathrm{dif}, n}(D))$$ by 
$$C^0_{\gamma_K}(\bold{D}^{(+)}_{\mathrm{dif}, n}(D))\rightarrow 
\tilde{C}^0_{\varphi,\gamma_K}(\bold{D}^{(+)}_{\mathrm{dif}, n}(D)): x\mapsto 
(x_m)_{m\geqq n}\,\text{ where }\,x_m :=x\,\, (m\geqq n) ,$$ 
$$C^1_{\gamma_K}(\bold{D}^{(+)}_{\mathrm{dif}, n}(D))\rightarrow 
\tilde{C}^1_{\varphi,\gamma_K}(\bold{D}^{(+)}_{\mathrm{dif}, n}(D)): x\mapsto ((x_m)_{m\geqq n}, 0 )\,\text{ where }\, x_m :=x \,(m\geqq n).$$
 It is easy to check that the map $C^{\bullet}_{\gamma_K}(\bold{D}^{(+)}_{\mathrm{dif}, n}(D))\rightarrow \tilde{C}^{\bullet}_{\varphi,\gamma_K}(\bold{D}^{(+)}_{\mathrm{dif}, n}(D))$ is  
quasi isomorphism. Because we have $C^{\bullet}_{\gamma_K}(\bold{D}^{(+)}_{\mathrm{dif}}(D))=\varinjlim_n C^{\bullet}_{\gamma_K}(\bold{D}^{(+)}_{\mathrm{dif}, n}(D))$, 
we obtain an isomorphism
$$\mathrm{H}^q(K, \bold{D}^{(+)}_{\mathrm{dif}}(D))\isom  \mathrm{H}^q(\varinjlim_n \tilde{C}^{\bullet}_{\varphi,\gamma_K}(\bold{D}^{(+)}_{\mathrm{dif},n}(D))).$$

Combining 
the above isomorphisms  and the above long exact sequence, we obtain the following desired long exact sequence
\[
\begin{array}{ll}
0\rightarrow& \mathrm{H}^0(K, D) \rightarrow  \mathrm{H}^0(K, D[1/t])\oplus \mathrm{H}^0(K, \bold{D}^+_{\mathrm{dif}}(D))
\rightarrow  \mathrm{H}^0(K, \bold{D}_{\mathrm{dif}}(D)) \\
  \xrightarrow{\delta_{1,D}}&\mathrm{H}^1(K, D)\rightarrow \mathrm{H}^1(K, D[1/t])\oplus \mathrm{H}^1(K, \bold{D}^+_{\mathrm{dif}}(D))
\rightarrow \mathrm{H}^1(K, \bold{D}_{\mathrm{dif}}(D)) \\
  \xrightarrow{\delta_{2,D}}&  \mathrm{H}^2(K, D)\rightarrow \mathrm{H}^2(K, D[1/t]) \rightarrow 0.

\end{array}
  \]
  The functoriality of this exact sequence is trivial by construction.
  This finishes  the proof of  the theorem.

\end{proof}

\begin{defn}
Let $D$ be a $(\varphi, \Gamma_K)$-module over $\bold{B}^{\dagger}_{\mathrm{rig},K}$. Then we define a map 
$$\mathrm{exp}_{K,D}: \bold{D}^K_{\mathrm{dR}}(D)\rightarrow \mathrm{H}^1(K, D)$$ as the connecting map $\delta_{1,D}: 
\bold{D}^K_{\mathrm{dR}}(D)=\mathrm{H}^0(K, \bold{D}_{\mathrm{dif}}(D))\rightarrow \mathrm{H}^1(K, D)$ of the 
long exact sequence of  Theorem \ref{exp}. We call $\mathrm{exp}_{K,D}$ the exponential map 
of $D$.

\end{defn}
\begin{rem}
By definition, we have $\mathrm{exp}_{K,D}(\mathrm{Fil}^0\bold{D}_{\mathrm{dR}}(D))=0$. Hence 
$\mathrm{exp}_{K,D}$ induces a map 
$$\mathrm{exp}_{K,D}:\bold{D}^K_{\mathrm{dR}}(D)/\mathrm{Fil}^0\bold{D}^K_{\mathrm{dR}}(D)\rightarrow \mathrm{H}^1(K, D).$$
\end{rem}

To study the map $\mathrm{exp}_{K,D}$, it is useful to define $\mathrm{exp}_{K,D}$ in a more explicit way. The following 
lemma gives explicit definitions of $\mathrm{exp}_{K,D}$ and $\delta_{2,D}:\mathrm{H}^1(K, \bold{D}_{\mathrm{dif}}(D))\rightarrow \mathrm{H}^2(K, D)$.

\begin{lemma}\label{explicit}
\begin{itemize}
\item[(1)]
Let $x$ be an element of $\bold{D}^K_{\mathrm{dR}}(D)$. Take an $n\geqq n(D)$ such that 
$x\in \bold{D}_{\mathrm{dif}, n}(D)$ and take an $\widetilde{x}\in (D^{(n)}[1/t])^{\Delta_K}$ such that 
$\iota_{m}(\widetilde{x})-x\in \bold{D}^+_{\mathrm{dif}, m}(D)$ for any $m\geqq n$ (such an $\widetilde{x}$ 
exists by Lemma \ref{2.8}). Then we have 
$$\mathrm{exp}_{K,D}(x)=[(\gamma_K-1)\widetilde{x}, (\varphi-1)\widetilde{x}]\in \mathrm{H}^1(K, D).$$
\item[(2)]Let $[x]$ be an element of $\mathrm{H}^1(K, \bold{D}_{\mathrm{dif}}(D))$. 
Let $x\in \bold{D}_{\mathrm{dif},n}(D)^{\Delta_K}$ be a lift of $[x]$ for some $n\geqq n(D)$. 
Take  an $\widetilde{x}\in (D^{(n)}[1/t])^{\Delta_K}$ such that $\iota_m(\widetilde{x})-x\in \bold{D}^+_{\mathrm{dif},m}(D)$ 
for any $m\geqq n$ (such an $\widetilde{x}$ exists by Lemma \ref{2.8}). Then we have 
$$\delta_{2,D}([x])=[(\varphi-1)\widetilde{x}]\in \mathrm{H}^2(K, D).$$

\end{itemize}

\end{lemma}

\begin{proof}
These formulae directly follow from the proof of the above theorem and the construction of the snake lemma.

\end{proof}

\subsection{dual exponential map}
In this subsection, when $D$ is de Rham, we define a map 
$\mathrm{exp}^{*}_{K,D^{\lor}(1)}:\mathrm{H}^1(K, D)\rightarrow \mathrm{Fil}^0\bold{D}^K_{\mathrm{dR}}(D)$, which we call 
the dual exponential map of $D$.  Then, we prove that the map $\mathrm{exp}^{*}_{K,D}:\mathrm{H}^1(K, D^{\lor}(1))\rightarrow 
\mathrm{Fil}^0\bold{D}^K_{\mathrm{dR}}(D^{\lor}(1))$ is the adjoint of the map 
$\mathrm{exp}_{K,D}:\bold{D}^K_{\mathrm{dR}}(D)/\mathrm{Fil}^0\bold{D}^K_{\mathrm{dR}}(D)\rightarrow \mathrm{H}^1(K, D)$. 

Before defining $\mathrm{exp}^{*}_{K,D^{\lor}(1)}$, we prove some preliminary lemmas.
Let $D_1, D_2$ be $(\varphi,\Gamma_K)$-modules over $\bold{B}^{\dagger}_{\mathrm{rig},K}$.
We define a paring
$$\cup_{\mathrm{dif}}:\mathrm{H}^0(K, \bold{D}_{\mathrm{dif}}(D_1))\times \mathrm{H}^1(K, \bold{D}_{\mathrm{dif}}(D_2))
\rightarrow \mathrm{H}^1(K, \bold{D}_{\mathrm{dif}}(D_1\otimes D_2))$$ 
by $x\cup_{\mathrm{dif}} [y]:=[x\otimes y]$ for any $x\in \mathrm{H}^0(K, \bold{D}_{\mathrm{dif}}(D_1))$ and 
$[y]\in \mathrm{H}^1(K, \bold{D}_{\mathrm{dif}}(D_2))$.
\begin{lemma}\label{2.12}
The following two diagrams are commutative;
$$
\begin{CD}
\mathrm{H}^0(K, \bold{D}_{\mathrm{dif}}(D_1))@.\times @.\mathrm{H}^1(K,\bold{D}_{\mathrm{dif}}(D_2))@>\cup_{\mathrm{dif}} >> \mathrm{H}^1(K, \bold{D}_{\mathrm{dif}}(D_1\otimes D_2)) \\
@VV\mathrm{exp}_{K,D_1} V@. @ AA [x,y]\mapsto [\iota_n(x)] A @VV \delta_{2, D_1\otimes D_2} V \\
\mathrm{H}^1(K, D_1)@.\times @.\mathrm{H}^1(K, D_2) @> \cup >> \mathrm{H}^2(K, D_1\otimes D_2),
\end{CD}
$$
$$
\begin{CD}
\mathrm{H}^0(K, \bold{D}_{\mathrm{dif}}(D_1)) @.\times @. \mathrm{H}^1(K,\bold{D}_{\mathrm{dif}}(D_2)) @> \cup_{\mathrm{dif}} >> \mathrm{H}^1(K, \bold{D}_{\mathrm{dif}}(D_1\otimes D_2))\\
@AA a\mapsto \iota_n(a) A@. @ VV\delta_{2,D_2}V @VV \delta_{2,D_1\otimes D_2}V \\
\mathrm{H}^0(K, D_1)@.\times @.\mathrm{H}^2(K, D_2) @> \cup >> \mathrm{H}^2(K, D_1\otimes D_2).
\end{CD}
$$
In other words, 
 we have equalities 
$$\delta_{2,D_1\otimes D_2}(z\cup_{\mathrm{dif}}[\iota_n(x)])=\mathrm{exp}_{K,D_1}(z)\cup [x,y]$$
and 
$$ \delta_{2,D_1\otimes D_2}(\iota_n(a)\cup_{\mathrm{dif}}[b])=a\cup \delta_{2,D_2}([b])$$
for any $z\in \mathrm{H}^0(K,\bold{D}_{\mathrm{dif}}(D_1)), [x,y]\in \mathrm{H}^1(K, D_2)$ and 
$a\in \mathrm{H}^0(K, D_1), [b]\in \mathrm{H}^1(K, \bold{D}_{\mathrm{dif}}(D_2))$.

\end{lemma}

\begin{proof}
Here, we only prove the commutativity of the first diagram. We can prove 
the commutativity of the second diagram in a similar way. 

Take any $z\in \mathrm{H}^0(K, \bold{D}_{\mathrm{dif}}(D_1))$ and $[x,y]\in \mathrm{H}^1(K, D_2)$.
Take $n$ sufficiently large such that 
$z\in \bold{D}_{\mathrm{dif},n}(D_1)$ and $x\in D_2^{(n)}, y\in D_2^{(n+1)}$.
By Lemma \ref{explicit} (1), if we take $\tilde{z}\in (D_1^{(n)}[1/t])^{\Delta_K}$ such that $\iota_m(\tilde{z})-z\in \bold{D}^+_{\mathrm{dif},m}(D_1)$ for any 
$m\geqq n$, then we have 
$$\mathrm{exp}_{K,D_1}(z)=[(\gamma_K-1)\tilde{z},(\varphi-1)\tilde{z}].$$ Hence we have 
\[
\begin{array}{ll}
\mathrm{exp}_{K,D_1}(z)\cup [x,y]&=[(\gamma_K-1)\tilde{z},(\varphi-1)\tilde{z}]\cup [x,y]\\
                                                             &=[(\varphi-1)\tilde{z}\otimes \varphi(x)-(\gamma_K-1)\tilde{z}\otimes \gamma_K(y)]\in \mathrm{H}^2(K, D_1\otimes D_2).
                                                             
\end{array}
\]

On the other hand, under the natural quasi-isomorphism $C^{\bullet}_{\gamma_K}(\bold{D}_{\mathrm{dif},n}(D))\rightarrow 
\tilde{C}^{\bullet}_{\varphi,\gamma_K}(\bold{D}_{\mathrm{dif},n}(D))$ which is defined in the proof of Theorem \ref{exp}, the element $[\iota_n(x)]\in \mathrm{H}^1(K, \bold{D}_{\mathrm{dif},n}(D_2))$
is sent to $$[(\iota_m(x))_{m\geqq n}, (\iota_m(y))_{m\geqq n+1}]\in \mathrm{H}^1(\tilde{C}^{\bullet}_{\varphi,\gamma_K}(\bold{D}_{\mathrm{dif},n}(D_2))).$$  Hence,
the element $z\cup_{\mathrm{dif}}[\iota_n(x)]\in \mathrm{H}^1(K, \bold{D}_{\mathrm{dif},n}(D_1\otimes D_2))$ is sent to
$$[(z\otimes \iota_m(x))_{m\geqq n}, (z\otimes \iota_m(y))_{m\geqq n+1}]\in \mathrm{H}^1(\tilde{C}^{\bullet}_{\varphi,\gamma_K}(\bold{D}_{\mathrm{dif},n}(D_1\otimes D_2))).$$
The element $(\tilde{z}\otimes x,\tilde{z}\otimes y)\in (D^{(n)}_1\otimes D^{(n)}_2[1/t])^{\Delta_K}\bigoplus (D^{(n+1)}_1\otimes D^{(n+1)}_2[1/t])^{\Delta_K}$ 
satisfies that $((\iota_m(\tilde{z}\otimes x))_{m\geqq n},(\iota_m(\tilde{z}\otimes y))_{m\geqq n+1})-((z\otimes \iota_m(x))_{m\geqq n}, (z\otimes \iota_m(y))_{m\geqq n+1})
\in \prod_{m\geqq n}\bold{D}^+_{\mathrm{dif},m}(D_1\otimes D_2)\bigoplus \prod_{m\geqq n+1} \bold{D}^+_{\mathrm{dif},m}(D_1\otimes D_2)$. Hence, by the definition of the boundary map 
$\delta_{2,D_1\otimes D_2}$ , 
we obtain
$$\delta_{2,D_1\otimes D_2}(z\cup_{\mathrm{dif}}[\iota_n(x)])=[(\varphi-1)(\tilde{z}\otimes x)-(\gamma_K-1)(\tilde{z}\otimes y)].$$
Using the equality $(\varphi-1)x=(\gamma_K-1)y$, it is easy to show the equality 
$$(\varphi-1)\tilde{z}\otimes \varphi(x)-(\gamma_K-1)\tilde{z}\otimes \gamma_K(y)=
(\varphi-1)(\tilde{z}\otimes x)-(\gamma_K-1)(\tilde{z}\otimes y).$$ Hence we obtain the desired 
equality
$$\mathrm{exp}_{K,D_1}(z)\cup [x,y]=\delta_{2,D_1\otimes D_2}(z\cup_{\mathrm{dif}}[\iota_n(x)]).$$

\end{proof}

Let $D$ be a de Rham $(\varphi,\Gamma_K)$-module over $\bold{B}^{\dagger}_{\mathrm{rig},K}$. Then 
the natural map 
$$K_{\infty}((t))\otimes_K \bold{D}^K_{\mathrm{dR}}(D)\rightarrow \bold{D}_{\mathrm{dif}}(D): f(t)\otimes x\mapsto f(t)x$$ is isomorphism. We identify $K_{\infty}((t))\otimes_K \bold{D}^K_{\mathrm{dR}}(D)$ with  $\bold{D}_{\mathrm{dif}}(D)$ by this isomorphism. Then, it is easy to check that 
 the map
$$g_D:\bold{D}_{\mathrm{dR}}^K(D)\isom \mathrm{H}^1(K,\bold{D}_{\mathrm{dif}}(D))(\isom \mathrm{H}^1(C^{\bullet}_{\gamma_K}(K_{\infty}((t))\otimes_{K}\bold{D}^K_{\mathrm{dR}}(D))))$$
 define by 
 $$g_D(x):=[\mathrm{log}(\chi(\gamma_K))(1\otimes x)]\in \mathrm{H}^1(C^{\bullet}_{\gamma_K}(K_{\infty}((t))\otimes_{K}\bold{D}^K_{\mathrm{dR}}(D)))$$ is isomorphism and is
 independent of the chose of $\gamma_K$ up to the canonical isomorphism.   
We note that $\bold{D}^K_{\mathrm{dR}}(\bold{B}^{\dagger}_{\mathrm{rig},K}(1))=K\cdot\frac{1}{t}e_1$.

\begin{lemma}\label{2.13}
The following diagram is commutative,
$$
\begin{CD}
\bold{D}^K_{\mathrm{dR}}(\bold{B}^{\dagger}_{\mathrm{rig},K}(1)) @> = >> K\cdot\frac{1}{t}e_1 @> \frac{a}{t}e_1\mapsto a >> K \\
@VV  g_{\bold{B}^{\dagger}_{\mathrm{rig},K}(1)}V @.@VV \mathrm{Tr}_{K/\mathbb{Q}_p}V \\
\mathrm{H}^1(K, \bold{D}_{\mathrm{dif}}(\bold{B}^{\dagger}_{\mathrm{rig},K}(1)))@> \delta_{2,\bold{B}^{\dagger}_{\mathrm{rig},K}(1)} >>
\mathrm{H}^2(K, \bold{B}^{\dagger}_{\mathrm{rig},K}(1)) @> f_{\mathrm{tr}} >> \mathbb{Q}_p.
\end{CD}
$$

\end{lemma}
\begin {proof}
Using the trace and the corestriction, it suffices to show the lemma when $K=\mathbb{Q}_p$.
Assume $K=\mathbb{Q}_p$. We prove the lemma by comparing the cohomology of $(\varphi,\Gamma_{\mathbb{Q}_p})$-module $\bold{B}^{\dagger}_{\mathrm{rig},\mathbb{Q}_p}(1)$ with
the Galois cohomology of the $p$-adic representation $\mathbb{Q}_p(1)$ of $G_{\mathbb{Q}_p}$ as follows. Please see $\S$ 2.5 for 
some notations and definitions in the proof below.

Take the element $[\mathrm{log}(\chi(\gamma_K)),0]\in \mathrm{H}^1(\mathbb{Q}_p,
\bold{B}^{\dagger}_{\mathrm{rig},\mathbb{Q}_p})$. Then we have 
$$g_{\bold{B}^{\dagger}_{\mathrm{rig},\mathbb{Q}_p}(1)}(\frac{a}{t}e_1)=\frac{a}{t}e_1\cup_{\mathrm{dif}}[\iota_n(\mathrm{log}(\chi(\gamma_K)))]$$ by the definition of $g_D$ and $\cup_{\mathrm{dif}}$. 
Hence, by Lemma \ref{2.12}, we obtain an equality 
\[
\begin{array}{ll}
\delta_{2,\bold{B}^{\dagger}_{\mathrm{rig},\mathbb{Q}_p}(1)}(g_{\bold{B}^{\dagger}_{\mathrm{rig},\mathbb{Q}_p}(1)}(\frac{a}{t}e_1))
                     &=\delta_{2,,\bold{B}^{\dagger}_{\mathrm{rig},\mathbb{Q}_p}(1)}(\frac{a}{t}e_1\cup_{\mathrm{dif}}[\iota_n(\mathrm{log}(\chi(\gamma_K)))])\\
                     &=\mathrm{exp}_{\mathbb{Q}_p,\bold{B}^{\dagger}_{\mathrm{rig},\mathbb{Q}_p}(1)}(\frac{a}{t}e_1)\cup[\mathrm{log}(\chi(\gamma_K)),0]
 \end{array}
 \]                    
 for any $a\in \mathbb{Q}_p$. Hence, it suffices to show 
 $$f_{\mathrm{tr}}(\mathrm{exp}_{\mathbb{Q}_p,\bold{B}^{\dagger}_{\mathrm{rig},\mathbb{Q}_p}(1)}(\frac{a}{t}e_1))\cup [\mathrm{log}(\chi(\gamma_K)),0])=a.$$
 We note that  the comparison isomorphism 
 $$\mathrm{H}^1(\mathbb{Q}_p,\bold{B}^{\dagger}_{\mathrm{rig},\mathbb{Q}_p})\isom \mathrm{H}^1(\mathbb{Q}_p, W(\bold{B}^{\dagger}_{\mathrm{rig},\mathbb{Q}_p}))\isom \mathrm{H}^1(\mathbb{Q}_p, \mathbb{Q}_p)$$ (where $\mathbb{Q}_p$ on the right hand side is the trivial $p$-adic representation of $G_{\mathbb{Q}_p}$ ) 
 sends $[\mathrm{log}(\chi(\gamma)), 0]$ to the element $\mathrm{log}(\chi)\in 
 \mathrm{Hom}(G_{\mathbb{Q}_p}^{\mathrm{ab}},\mathbb{Q}_p)=\mathrm{H}^1(\mathbb{Q}_p,\mathbb{Q}_p)$ defined by $\mathrm{log}(\chi):G^{\mathrm{ab}}_{\mathbb{Q}_p}\rightarrow \mathbb{Q}_p:g\mapsto \mathrm{log}(\chi(g))$. Hence, 
 by Theorem \ref{4.5}, in particular, by the compatibility of $\mathrm{exp}_{K,D}$ with  $\mathrm{exp}_{K,W(D)}$ (we remark that we don't use any results in this subsection $\S$ 2.4 to prove this compatibility), it 
 suffices to show that Tate's paring satisfies 
 $$<\mathrm{exp}_{\mathbb{Q}_p,\mathbb{Q}_p}(\frac{a}{t}e_1), \mathrm{log}(\chi)>=a$$ for any $a\in \mathbb{Q}_p$.
 Because it is known that 
 $\kappa(b)=\mathrm{exp}_{\mathbb{Q}_p,\mathbb{Q}_p}(\frac{\mathrm{log}(b)}{t}e_1)$ for 
 any $b\in \mathbb{Z}_p^{\times}$ ($\kappa:\mathbb{Q}_p^{\times}
 \rightarrow \mathrm{H}^1(\mathbb{Q}_p,\mathbb{Q}_p(1))$ is the Kummer map), we obtain 
 \[
\begin{array}{ll}
<\mathrm{exp}_{\mathbb{Q}_p,\mathbb{Q}_p}(\frac{\mathrm{log}(b)}{t}e_1), \mathrm{log}(\chi)>&=<\kappa(b), \mathrm{log}(\chi)>\\
&=\mathrm{log}(b)
\end{array}
\]
for any $b\in \mathbb{Z}_p^{\times}$, 
where the last equality follows from Remark \ref{normalization}. This finishes to prove the lemma.

\end{proof}
Let $D$ be a $(\varphi,\Gamma_K)$-module over $\bold{B}^{\dagger}_{\mathrm{rig},K}$. 
We define a paring
\[
\begin{array}{ll}
<,>_{\mathrm{dif}}&:\mathrm{H}^0(K,\bold{D}_{\mathrm{dif}}(D))\times \mathrm{H}^1(K, \bold{D}_{\mathrm{dif}}(D^{\lor}(1)))\xrightarrow{\cup_{\mathrm{dif}}}
\mathrm{H}^1(K, \bold{D}_{\mathrm{dif}}(D\otimes D^{\lor}(1)))\\
&\xrightarrow{\mathrm{ev}}\mathrm{H}^1(K, \bold{D}_{\mathrm{dif}}(\bold{B}^{\dagger}_{\mathrm{rig},K}(1)))
\xrightarrow{g^{-1}_{\bold{B}^{\dagger}_{\mathrm{rig},K}(1)}} \bold{D}^K_{\mathrm{dR}}(\bold{B}^{\dagger}_{\mathrm{rig},K}(1)) \xrightarrow{\frac{a}{t}e_1\mapsto a} K.
\end{array}
\]
\begin{lemma}\label{2.15}
Let $D$ be a $(\varphi,\Gamma_K)$-module over $\bold{B}^{\dagger}_{\mathrm{rig},K}$. 
Then  the following diagrams 
$$
\begin{CD}
\mathrm{H}^0(K, \bold{D}_{\mathrm{dif}}(D))@.\times @.\mathrm{H}^1(K, \bold{D}_{\mathrm{dif}}(D^{\lor}(1)))@> {<,>_{\mathrm{dif}}} >> K\\
@VV \mathrm{exp}_{K,D}V @. @AA [x,y]\mapsto [\iota_n(x)]A @VV \mathrm{Tr}_{K/\mathbb{Q}_p}V \\
\mathrm{H}^1(K, D)@.\times @.\mathrm{H}^1(K,D^{\lor}(1)) @>> {<,>}> \mathbb{Q}_p
\end{CD}
$$
and 
$$
\begin{CD}
\mathrm{H}^0(K, \bold{D}_{\mathrm{dif}}(D))@.\times @.\mathrm{H}^1(K, \bold{D}_{\mathrm{dif}}(D^{\lor}(1)))@> {<,>_{\mathrm{dif}}} >> K\\
@AA x \mapsto \iota_n(x) A@. @VV \delta_{2,D^{\lor}(1)}V@VV \mathrm{Tr}_{K/\mathbb{Q}_p}V \\
\mathrm{H}^0(K, D)@.\times @.\mathrm{H}^2(K,D^{\lor}(1)) @>> {<,>}> \mathbb{Q}_p
\end{CD}
$$
are commutative.
In other words,  we have equalities

$$<\mathrm{exp}_{K,D}(z),[x,y]>=\mathrm{Tr}_{K/\mathbb{Q}_p}(<z,[\iota_n(x)]>_{\mathrm{dif}})$$
and
$$<a, \delta_{2,D^{\lor}(1)}([b])>=\mathrm{Tr}_{K/\mathbb{Q}_p}(<\iota_n(a),[b]>_{\mathrm{dif}})$$
for any $z\in \mathrm{H}^0(K, \bold{D}_{\mathrm{dif}}(D)), [x,y]\in \mathrm{H}^1(K, D^{\lor}(1))$ and 
$a\in \mathrm{H}^0(K, \bold{D}_{\mathrm{dif}}(D)), [b]\in \mathrm{H}^1(K, D^{\lor}(1))$.

\end{lemma}
\begin{proof}
This lemma follows from Lemma \ref{2.12} and Lemma \ref{2.13}.
\end{proof}

Let $D$ be a de Rham $(\varphi,\Gamma_K)$-module over $\bold{B}^{\dagger}_{\mathrm{rig},K}$. 

We define a map 
$$\mathrm{exp}^{*}_{K,D^{\lor}(1)}: \mathrm{H}^1(K ,D)\rightarrow \mathrm{Fil}^0\bold{D}^K_{\mathrm{dR}}(D)$$ 
as the composition of the natural map 
$$\mathrm{H}^1(K, D)\rightarrow \mathrm{H}^1(K, \bold{D}^+_{\mathrm{dif}}(D))\rightarrow 
\mathrm{H}^1(K, \bold{D}_{\mathrm{dif}}(D)):[x,y]\mapsto [\iota_n(x)]$$
(for sufficiently large $n$)
 with the inverse of the isomorphism 
 $$g_{D}:\bold{D}^K_{\mathrm{dR}}(D)\isom \mathrm{H}^1(K, \bold{D}_{\mathrm{dif}}(D)).$$
 Because 
we have $\bold{D}^+_{\mathrm{dif}}(D)=\mathrm{Fil}^0(K_{\infty}((t))\otimes_K \bold{D}^K_{\mathrm{dR}}(D))$, we can easily see that 
the image of $\mathrm{exp}^{*}_{K,D^{\lor}(1)}$ is contained in $\mathrm{Fil}^0\bold{D}^K_{\mathrm{dR}}(D)$. 
As in the case of $p$-adic Galois representations, the map $\mathrm{exp}^{*}_{K,D}$ is the adjoint of $\mathrm{exp}_{K,D}$ in the following sense.
We define a $K$-bi-linear perfect paring $[-,-]_{\mathrm{dR}}$ by
$$[-,-]_{\mathrm{dR}}:\bold{D}^K_{\mathrm{dR}}(D)\times \bold{D}^K_{\mathrm{dR}}(D^{\lor}(1))\xrightarrow{\mathrm{ev}} \bold{D}^K_{\mathrm{dR}}(\bold{B}^{\dagger}_{\mathrm{rig},K}(1))\xrightarrow{\frac{a}{t}e_1\mapsto a}
K$$ 
 where $\mathrm{ev}$ is the natural evaluation map. 
By the definition of $\mathrm{Fil}^i$, this paring induces a perfect paring 
$$[-,-]_{\mathrm{dR}}:\bold{D}^{K}_{\mathrm{dR}}(D)/\mathrm{Fil}^0\bold{D}^K_{\mathrm{dR}}(D)\times \mathrm{Fil}^0\bold{D}^K_{\mathrm{dR}}(D^{\lor}(1))\rightarrow K.$$
\begin{prop}\label{2.16}
Let $D$ be a de Rham $(\varphi,\Gamma_K)$-module over $\bold{B}^{\dagger}_{\mathrm{rig},K}$. 
For any $\overline{x}\in \bold{D}^K_{\mathrm{dR}}(D)/\mathrm{Fil}^0\bold{D}^K_{\mathrm{dR}}(D)$ and $y\in \mathrm{H}^1(K, D^{\lor}(1))$, the following equality holds
$$<\mathrm{exp}_{K,D}(\overline{x}),y>=\mathrm{Tr}_{K/\mathbb{Q}_p}([\overline{x},\mathrm{exp}^{*}_{K,D}(y)]_{\mathrm{dR}}).$$

\end{prop}

\begin{proof}
By Lemma \ref{2.15}, it suffices to show the equality 
$$[x,z]_{\mathrm{dR}}=<x,g_{D^{\lor}(1)}(z)>_{\mathrm{dif}}$$
for any $x\in \mathrm{H}^0(K, \bold{D}_{\mathrm{dif}}(D)), z\in \mathrm{H}^0(K, \bold{D}_{\mathrm{dif}}(D^{\lor}(1)))$; but this equality is trivial by definition.
\end{proof}

\subsection{comparison with Bloch-Kato's exponential map of $B$-pairs}
In this  subsection, we show that the long exact sequence of Theorem \ref{exp} associated to $D$ is isomorphic to the 
long exact sequence naturally defined from the cohomologies of the corresponding $B$-pair $W(D)$. 
In particular,  in the \'etale case, 
we show that the sequence of Theorem \ref{exp} is isomorphic to the long exact sequence induced from the Bloch-Kato's fundamental short exact sequence.

We first recall the definition of $B$-pairs  and the definition of  the functor from the category of $(\varphi,\Gamma)$-modules to the category of $B$-pairs which 
induces an equivalence between these categories, see \cite{Ber08a} for more details.

 The following definition is due to Berger (\cite{Ber08a}).
 \begin{defn}
 We say that a pair $W:=(W_e,W^+_{\mathrm{dR}})$ is a $B$-pair of $G_K$ if
 \begin{itemize}
 \item[(1)]$W_e$ is a finite free $\bold{B}_e$-module with a continuous semi-linear $G_K$-action.
 \item[(2)]$W^+_{\mathrm{dR}}$ is a $G_K$-stable finite $\bold{B}^+_{\mathrm{dR}}$-submodule of $W_{\mathrm{dR}}:=\bold{B}_{\mathrm{dR}}\otimes_{\bold{B}_e}W_e$ which generates
 $W_{\mathrm{dR}}$ as $\bold{B}_{\mathrm{dR}}$-module,
 \end{itemize}
 where semi-linear means that $g(ax)=g(a)g(x)$ for any $a\in \bold{B}_e$, $x\in W_e$ and $g\in G_K$.
 
 \end{defn}
 \begin{rem}
 Let $V$ be a $p$-adic representation of $G_K$. We define a $B$-pair 
 $$W(V):=(\bold{B}_e\otimes_{\mathbb{Q}_p}V, \bold{B}^+_{\mathrm{dR}}\otimes_{\mathbb{Q}_p}V).$$ 
 By Bloch-Kato's fundamental short exact sequence (\cite{BK90})
 $$0\rightarrow \mathbb{Q}_p\xrightarrow{x\mapsto (x,x)} \bold{B}_e\bigoplus \bold{B}^+_{\mathrm{dR}}\xrightarrow{(x,y)\mapsto x-y} \bold{B}_{\mathrm{dR}}\rightarrow 0,$$
 we can easily see that this functor $V\mapsto W(V)$ is fully faithful, hence we can view the category of $p$-adic representations of $G_K$ as a full subcategory of 
 the category of $B$-pairs of $G_K$.
 \end{rem}
 By the theorems of Fontaine, Cherbonnier-Colmez and Kedlaya, the category of $p$-adic representations of $G_K$ is equivalent to the 
 category of \'etale $(\varphi,\Gamma_K)$-modules over $\bold{B}^{\dagger}_{\mathrm{rig},K}$. Berger extended this categorical equivalence to 
 the equivalence between the category of $B$-pairs of $G_K$ with that of $(\varphi,\Gamma_K)$-modules over $\bold{B}^{\dagger}_{\mathrm{rig},K}$, which we recall below.
 
 We first note that we have a $(\varphi$, $G_{K})$-equivariant canonical injection $\bold{B}^{\dagger}_{\mathrm{rig},K}\hookrightarrow \widetilde{\bold{B}}^{\dagger}_{\mathrm{rig}}$.
 Let $D$ be a $(\varphi,\Gamma_K)$-module over $\bold{B}^{\dagger}_{\mathrm{rig},K}$ of rank $d$. For each $n\geqq n(D)$, we define 
$$W_e(D^{(n)}):=(\widetilde{\bold{B}}^{\dagger,r_n}_{\mathrm{rig}}[1/t]\otimes_{\bold{B}^{\dagger,r_n}_{\mathrm{rig},K}}D^{(n)})^{\varphi=1}.$$ 
 Since we have an isomorphism 
 $$\bold{B}^{\dagger,r_{n+1}}_{\mathrm{rig},K}\otimes_{\varphi,\bold{B}^{\dagger,r_n}_{\mathrm{rig},K}}D^{(n)}\isom D^{(n+1)}:a\otimes x\mapsto a\varphi(x)$$ and the map 
 $\varphi:\widetilde{\bold{B}}^{\dagger,r_n}_{\mathrm{rig}}\isom \widetilde{\bold{B}}^{\dagger,r_{n+1}}_{\mathrm{rig}}$ is isomorphism, 
 we obtain a natural  isomorphism 
  \begin{multline*}
  \widetilde{\bold{B}}^{\dagger,r_n}_{\mathrm{rig}}\otimes_{\bold{B}^{\dagger,r_n}_{\mathrm{rig},K}}D^{(n)}\xrightarrow{a\otimes x\mapsto \varphi(a)\otimes x}
 \widetilde{\bold{B}}^{\dagger,r_{n+1}}_{\mathrm{rig}}\otimes_{\varphi,\bold{B}^{\dagger,r_n}_{\mathrm{rig},K}}D^{(n)} \\
 \isom \widetilde{\bold{B}}^{\dagger,r_{n+1}}_{\mathrm{rig}}\otimes_{\bold{B}^{\dagger,r_{n+1}}_{\mathrm{rig},K}}(\bold{B}^{\dagger,r_{n+1}}_{\mathrm{rig},K}\otimes_{\varphi,\bold{B}^{\dagger,r_n}_{\mathrm{rig},K}}D^{(n)})
 \xrightarrow{a\otimes(b\otimes x)\mapsto a\otimes b\varphi(x)}\widetilde{\bold{B}}^{\dagger,r_{n+1}}_{\mathrm{rig}}\otimes_{\bold{B}^{\dagger,r_{n+1}}_{\mathrm{rig},K}}D^{(n+1)},
 \end{multline*}
 i.e. the map 
 $$\varphi:\widetilde{\bold{B}}^{\dagger,r_n}_{\mathrm{rig}}\otimes_{\bold{B}^{\dagger,r_n}_{\mathrm{rig},K}}D^{(n)}\rightarrow \widetilde{\bold{B}}^{\dagger,r_{n+1}}_{\mathrm{rig}}\otimes_{\bold{B}^{\dagger,r_{n+1}}_{\mathrm{rig},K}}D^{(n+1)}:a\otimes x\mapsto \varphi(a)\otimes \varphi(x)$$ is isomorphism. 
 Hence, we obtain the following diagram
 $$
 \begin{CD}
  0@>>>  W_e(D^{(n)})@>>>   \widetilde{\bold{B}}^{\dagger,r_n}_{\mathrm{rig}}[1/t]\otimes_{\bold{B}^{\dagger,r_n}_{\mathrm{rig},K}}D^{(n)}@ >\varphi-1>>
   \widetilde{\bold{B}}^{\dagger,r_{n+1}}_{\mathrm{rig}}[1/t]\otimes_{\bold{B}^{\dagger,r_{n+1}}_{\mathrm{rig},K}}D^{(n+1)}  \\
   @.@VV \varphi V    @VV  \varphi V  @ VV\varphi V \\
   0@>>>  W_e(D^{(n+1)})@>>>   \widetilde{\bold{B}}^{\dagger,r_{n+1}}_{\mathrm{rig}}[1/t]\otimes_{\bold{B}^{\dagger,r_{n+1}}_{\mathrm{rig},K}}D^{(n+1)}@ >\varphi-1>> 
   \widetilde{\bold{B}}^{\dagger,r_{n+2}}_{\mathrm{rig}}[1/t]\otimes_{\bold{B}^{\dagger,r_{n+2}}_{\mathrm{rig},K}}D^{(n+2)} .
    \end{CD}
 $$
 with exact rows. Hence, the map 
 $$\varphi:W_e(D^{(n)})\isom W_e(D^{(n+1)})$$ is also isomorphism. 
 We define $$W_e(D):=W_e(D^{(n)})$$ for any $n\geqq n(D)$. Using the isomorphism $\varphi:W_e(D^{(n)})\isom W_e(D^{(n+1)})$,
 $W_e(D)$ does not depend on the choice of $n$. One has that  $W_e(D)$ is a finite free $\bold{B}_e$-module of rank $d$ and
 the natural map $$\widetilde{\bold{B}}^{\dagger,r_n}_{\mathrm{rig}}[1/t]\otimes_{\bold{B}_e}W_e(D^{(n)})\rightarrow 
 \widetilde{\bold{B}}^{\dagger,r_n}_{\mathrm{rig}}[1/t]\otimes_{\bold{B}^{\dagger,r_n}_{\mathrm{rig},K}}D^{(n)}:a\otimes x\mapsto ax $$ is isomorphism by Proposition 2.2.6 of \cite{Ber08a}. Put 
 $$W_{\mathrm{dR}}(D):=\bold{B}_{\mathrm{dR}}\otimes_{\bold{B}_e}W_e(D).$$ 
 Using the isomorphism above, we obtain an isomorphism 
 \begin{multline*}
 W_{\mathrm{dR}}(D)\isom\bold{B}_{\mathrm{dR}}\otimes_{\bold{B}_e}W_e(D^{(n)})\isom 
 \bold{B}_{\mathrm{dR}}\otimes_{\iota_n,\widetilde{\bold{B}}^{\dagger,r_n}_{\mathrm{rig}}[1/t]}(\widetilde{\bold{B}}^{\dagger,r_n}_{\mathrm{rig}}[1/t]\otimes_{\bold{B}_e}W_e(D^{(n)}))\\
 \isom \bold{B}_{\mathrm{dR}}\otimes_{\iota_n,\widetilde{\bold{B}}^{\dagger,r_n}_{\mathrm{rig}}[1/t]}(\widetilde{\bold{B}}^{\dagger,r_n}_{\mathrm{rig}}[1/t]\otimes_{\bold{B}^{\dagger,r_n}_{\mathrm{rig},K}}D^{(n)})\isom \bold{B}_{\mathrm{dR}}\otimes_{\iota_n,\bold{B}^{\dagger,r_n}_{\mathrm{rig},K}}D^{(n)}.
 \end{multline*}
  We define a $\bold{B}^+_{\mathrm{dR}}$-submodule 
  $$W^+_{\mathrm{dR}}(D):=\bold{B}^+_{\mathrm{dR}}\otimes_{\iota_n,\bold{B}^{\dagger,r_n}_{\mathrm{rig},K}}D^{(n)}$$ of $W_{\mathrm{dR}}(D)$. Using the isomorphism 
  $$\bold{B}^+_{\mathrm{dR}}\otimes_{\iota_n,\bold{B}^{\dagger,r_n}_{\mathrm{rig},K}}D^{(n)}\isom \bold{B}^+_{\mathrm{dR}}\otimes_{\iota_{n+1},\bold{B}^{\dagger,r_{n+1}}_{\mathrm{rig},K}}D^{(n+1)}:
  a\otimes x\mapsto a\otimes \varphi(x),$$ $W^+_{\mathrm{dR}}(D)$ also does not depend on the choice of $n$. 
  Hence, we obtain a $B$-pair $W(D):=(W_e(D), W^+_{\mathrm{dR}}(D))$. 
  
  The main theorem of \cite{Ber08a} is the following.
  \begin{thm}
  The functor $D\mapsto W(D)$ is exact and gives  an equivalence of categories  between 
  the category of $(\varphi,\Gamma_K)$-modules over $\bold{B}^{\dagger}_{\mathrm{rig},K}$ and the category 
  of $B$-pairs of $G_K$. Moreover, if we restrict this functor to \'etale $(\varphi,\Gamma_K)$-modules, this gives an equivalence of 
  categories between the category of \'etale $(\varphi,\Gamma_K)$-modules over $\bold{B}^{\dagger}_{\mathrm{rig},K}$ and the category of $p$-adic representations 
  of  $G_K$.

  \end{thm}
  \begin{proof}
  This is Theorem 2.2.7 and Proposition 2.2.9 of \cite{Ber08a}.
  \end{proof}
  \begin{rem}\label{4.4}
  The inverse functor $D(-)$ of $W(-)$ is defined as follows, see $\S$ 2 of \cite{Ber08a} for the proof. Let 
  $W=(W_e,W^+_{\mathrm{dR}})$ be a $B$-pair of $G_K$ of rank $d$. For  each $n\geqq 1$, we first define 
  $$\widetilde{D}^{(n)}(W):=\{x\in \widetilde{\bold{B}}^{\dagger,r_n}_{\mathrm{rig}}[1/t]\otimes_{\bold{B}_e} W_e| \iota_m(x)\in W^+_{\mathrm{dR}} \text{ for any } m\geqq n\}.$$
  Berger showed that $\widetilde{D}(W):=\varinjlim_{n}\widetilde{D}^{(n)}(W)$ is a finite free $\widetilde{\bold{B}}^{\dagger}_{\mathrm{rig}}$-module of rank $d$ with $(\varphi,G_K)$-action. Then, $D(W)$ is defined as the unique $(\varphi,\Gamma_K)$-submodule $D(W)\subseteq \widetilde{D}(W)^{\mathrm{Ker}(\chi)}$ over $\bold{B}^{\dagger}_{\mathrm{rig},K}$ such that $\widetilde{\bold{B}}^{\dagger}_{\mathrm{rig}}\otimes_{\bold{B}^{\dagger}_{\mathrm{rig},K}}D(W)\isom \widetilde{D}(W)$.

  \end{rem}

Next, we recall the definition of Galois cohomology of $B$-pairs, see $\S$ 2 of \cite{Na09} and the appendix of \cite{Na10} for details.
For a continuous $G_K$-module $M$ and for each $q\geqq 0$, we denote by
$$C^q(G_K, M):=\{c: G_K^{\times q}\rightarrow M\,\,\text{ continuous map} \}$$ the set of $q$ continuous cochains (when 
$q=0$, we define $G_K^{\times 0}:=\{\text{one point}\}$). As usual, we define
the  map 
$$\delta_q:C^q(G_K, M)\rightarrow C^{q+1}(G_K, M)$$
 by 
 \begin{multline*}
  \delta_q(c)(g_1,g_2,\cdots,g_{q+1}):=g_1c(g_2,\cdots,g_{q+1})+(-1)^{q+1}c(g_1,g_2,\cdots,g_q)\\
+\sum_{s=1}^q(-1)^sc(g_1,\cdots,g_{s-1},g_sg_{s+1},g_{s+2},\cdots,g_{q+1})
\end{multline*}
 and define the continuous cochain 
complex concentrated in degree $[0,+\infty)$ by
$$C^{\bullet}(G_K, M):=[C^0(G_K,M)\xrightarrow{\delta_0}C^1(G_K, M)\xrightarrow{\delta_1}\cdots].$$ 
We define
 $$\mathrm{H}^q(K, M):=\mathrm{H}^q(C^{\bullet}(G_K, M)).$$

For a $B$-pair $W:=(W_e, W^+_{\mathrm{dR}})$,
we denote by
$$C^{\bullet}(G_K, W):=\mathrm{Cone}(C^{\bullet}(G_K, W_e)\oplus C^{\bullet}(G_K, W^+_{\mathrm{dR}})
\xrightarrow{(c_e,c_{\mathrm{dR}})\mapsto c_e-c_{\mathrm{dR}}} C^{\bullet}(G_K, W_{\mathrm{dR}}))[-1]$$  the degree $(-1)$-shift of the mapping cone 
of the map of complexes 
$$C^{\bullet}(G_K, W_e)\oplus C^{\bullet}(G_K, W^+_{\mathrm{dR}})
\rightarrow C^{\bullet}(G_K, W_{\mathrm{dR}}):(c_e,c_{\mathrm{dR}})\mapsto c_e-c_{\mathrm{dR}}.$$ 
We define
 $$\mathrm{H}^q(K, W):=\mathrm{H}^q(C^{\bullet}(G_K, W)).$$ 

By the definition of mapping cone, we have the following long exact sequence.
 \begin{align*}
0\rightarrow& \mathrm{H}^0(K,W)\rightarrow \mathrm{H}^0(K, W_e)\oplus \mathrm{H}^0(K, W^+_{\mathrm{dR}})\rightarrow 
\mathrm{H}^0(K, W_{\mathrm{dR}})\\
\xrightarrow{\delta_{1,W}} &\mathrm{H}^1(K, W)\rightarrow \mathrm{H}^1(K, W_e)\oplus\mathrm{H}^1(K, W^+_{\mathrm{dR}})\rightarrow 
\mathrm{H}^1(K, W_{\mathrm{dR}})\\
 \xrightarrow{\delta_{2,W}} &\mathrm{H}^2(K, W)\rightarrow \mathrm{H}^2(K, W_e)\rightarrow 0,
\end{align*}
where the vanishings of $\mathrm{H}^q(K, W^+_{\mathrm{dR}}), \mathrm{H}^q(K, W_{\mathrm{dR}}), \mathrm{H}^{q+1}(K, W)$ and 
$\mathrm{H}^{q+1}(K, W_e)$ for any $q\geqq 2$ are proved in \cite{Na10}.

We define
$$\bold{D}^K_{\mathrm{dR}}(W):=\mathrm{H}^0(K, W_{\mathrm{dR}}),$$ and 
we define
$$\mathrm{exp}_{K,W}:=\delta_{1,W}:\bold{D}^K_{\mathrm{dR}}(W)\rightarrow \mathrm{H}^1(K, W)$$ as the first boundary map of the 
above exact sequence.

When $W=W(V)$, since we have a short exact sequence 
$$0\rightarrow V\rightarrow \bold{B}_e\otimes_{\mathbb{Q}_p}V\oplus \bold{B}^+_{\mathrm{dR}}\otimes_{\mathbb{Q}_p}V\rightarrow \bold{B}_{\mathrm{dR}}\otimes_{\mathbb{Q}_p}V\rightarrow 
0$$ by Bloch-Kato, we have a canonical quasi-isomorphism 
$$C^{\bullet}(G_K, V)\rightarrow C^{\bullet}(G_K, W(V)).$$ 
This quasi-isomorphism gives an 
isomorphism $$\mathrm{H}^q(K, V)\isom \mathrm{H}^q(K, W(V))$$
for each $q$.  By this isomorphism, the above exact sequence for $W=W(V)$ is 
equal to the exact sequence
 \begin{align*}
0\rightarrow & \mathrm{H}^0(K,V)\rightarrow \mathrm{H}^0(K, \bold{B}_e\otimes_{\mathbb{Q}_p}V)\oplus \mathrm{H}^0(K, \bold{B}^+_{\mathrm{dR}}\otimes_{\mathbb{Q}_p}V)\rightarrow 
\mathrm{H}^0(K, \bold{B}_{\mathrm{dR}}\otimes_{\mathbb{Q}_p}V)\\
\xrightarrow{\delta_{1,V}}& \mathrm{H}^1(K, V)\rightarrow \mathrm{H}^1(K, \bold{B}_e\otimes_{\mathbb{Q}_p}V)\oplus\mathrm{H}^1(K, \bold{B}^+_{\mathrm{dR}}\otimes_{\mathbb{Q}_p}V)\rightarrow 
\mathrm{H}^1(K, \bold{B}_{\mathrm{dR}}\otimes_{\mathbb{Q}_p}V)\\
 \xrightarrow{\delta_{2,V}}& \mathrm{H}^2(K, V)\rightarrow \mathrm{H}^2(K, \bold{B}_e\otimes_{\mathbb{Q}_p}V)\rightarrow 0,
\end{align*}
obtained from Bloch-Kato's fundamental short exact sequence.


The main result of this subsection is the following.

\begin{thm}\label{4.5}
Let $D$ be a $(\varphi,\Gamma_K)$-module over $\bold{B}^{\dagger}_{\mathrm{rig},K}$. 
For each $q\geqq 0$, there exist the following functorial isomorphisms
\begin{itemize}
\item[(1)]$\mathrm{H}^q(K, D)\isom \mathrm{H}^q(K, W(D))$,
\item[(2)]$\mathrm{H}^q(K, D[1/t])\isom \mathrm{H}^q(K, W_e(D))$,
\item[(3)]$\mathrm{H}^q(K, \bold{D}^+_{\mathrm{dif}}(D))\isom \mathrm{H}^q(K, W^+_{\mathrm{dR}}(D))$,
\item[(4)]$\mathrm{H}^q(K, \bold{D}_{\mathrm{dif}}(D))\isom \mathrm{H}^q(K, W_{\mathrm{dR}}(D))$.
\end{itemize}
Moreover, these isomorphisms give an isomorphism between the exact sequence associated to $D$ in
 Theorem \ref{exp} and 
that associated to $W(D)$ defined above.

\end{thm}
\begin{proof}
We proved  (1) in Theorem 5.11 of \cite{Na10}.
Since we have $W^+_{\mathrm{dR}}(D)=\bold{B}^+_{\mathrm{dR}}\otimes_{K_{\infty}[[t]]}\bold{D}^+_{\mathrm{dif}}(D)$, 
then (3) follows Theorem 2.14 of \cite{Fo03}.
(4) follows from (3) since we have $W_{\mathrm{dR}}(D)=\varinjlim_{n\geqq 0} \frac{1}{t^n}W^+_{\mathrm{dR}}(D)$ and 
$\bold{D}_{\mathrm{dif}}(D)=\varinjlim_{n\geqq 0} \frac{1}{t^n}\bold{D}^+_{\mathrm{dif}}(D)$. 

We prove (2) using (1). 
By (1),  we have $$\mathrm{H}^q(K, D[1/t])\isom \varinjlim_{n}\mathrm{H}^q(K,\frac{1}{t^n}D)
\isom \varinjlim_{n}\mathrm{H}^q(K, W(\frac{1}{t^n}D)).$$ 
Since we have $W(\frac{1}{t^n}D)\isom (W_e(D), \frac{1}{t^n}W^+_{\mathrm{dR}}(D))$ for each $n\geqq 0$, 
we obtain 
 $$\varinjlim_{n\geqq 0}C^{\bullet}(G_K, W^+_{\mathrm{dR}}(\frac{1}{t^n}D))=\varinjlim_{n\geqq 0}C^{\bullet}(G_K,
\frac{1}{t^n}W^+_{\mathrm{dR}}(D))=C^{\bullet}(G_K, W_{\mathrm{dR}}(D)).$$ 
Hence, we obtain an isomorphism
\begin{multline*}
  \varinjlim_{n\geqq 0}\mathrm{H}^q(K, W(\frac{1}{t^n}D))\isom \mathrm{H}^q(\varinjlim_{n\geqq 0}\mathrm{C}^{\bullet}(G_K, 
W(\frac{1}{t^n}(D)))\\
\isom \mathrm{H}^q(\mathrm{Cone}(C^{\bullet}(G_K, W_e(D))\oplus C^{\bullet}(G_K, W_{\mathrm{dR}}(D))
\xrightarrow{(c_e,c_{\mathrm{dR}})\mapsto c_e-c_{\mathrm{dR}}}C^{\bullet}(G_K, W_{\mathrm{dR}}(D)))[-1]).
\end{multline*}
Since we have the following short exact sequence of complexes
  \begin{multline*}
0\rightarrow C^{\bullet}(G_K, W_e(D))\xrightarrow{x\mapsto (x,x)}C^{\bullet}(G_K, W_e(D))\oplus
C^{\bullet}(G_K, W_{\mathrm{dR}}(D))\\
\xrightarrow{(x,y)\mapsto x-y}C^{\bullet}(G_K, W_{\mathrm{dR}}(D))\rightarrow 0, 
\end{multline*}
we obtain a natural isomorphism 
 \begin{multline*}
\mathrm{H}^q(K, W_e(D))\\
\isom \mathrm{H}^q(\mathrm{Cone}(C^{\bullet}(G_K, W_e(D))\oplus C^{\bullet}(G_K, W_{\mathrm{dR}}(D))
\xrightarrow{(c_e,c_{\mathrm{dR}})\mapsto c_e-c_{\mathrm{dR}}}C^{\bullet}(G_K, W_{\mathrm{dR}}(D)))[-1]),
\end{multline*}
which proves (2).

Next, we prove that  the isomorphisms $(1)$ to $(4)$ of the theorem give an isomorphism of  the corresponding  long exact sequences. 
Since the other commutativities are easy to check,
it suffices to show that the following two diagrams are commutative
\begin{itemize}
\item[(i)]$$
 \begin{CD}
  \mathrm{H}^0(K, \bold{D}_{\mathrm{dif}}(D))@>\isom >>  \mathrm{H}^0(K, W_{\mathrm{dR}}(D))   \\
   @VV \mathrm{exp}_{K,D} V @VV \mathrm{exp}_{K,W(D)} V \\
   \mathrm{H}^1(K, D)@>\isom>>  \mathrm{H}^1(K, W(D)),
       \end{CD}
 $$
 \item[(ii)]$$
 \begin{CD}
 \mathrm{H}^1(K,\bold{D}_{\mathrm{dif}}(D)) @>\isom>> \mathrm{H}^1(K, W_{\mathrm{dR}}(D))\\
 @VV \delta_{2,D} V @VV \delta_{2, W(D)} V \\
 \mathrm{H}^2(K, D) @> \isom >> \mathrm{H}^2(K, W(D)).
 \end{CD}
 $$
 \end{itemize}

We first prove the commutativity of (i). For simplicity, we assume that
 $\Gamma_K$ has a topological generator $\gamma_K$, the general case can be proved similarly using the argument of $\S$ 2.1 of \cite{Li08}.
 If we denote by
 $\mathrm{Ext}^1(\bold{B}^{\dagger}_{\mathrm{rig},K}, D)$ (resp. $\mathrm{Ext}^1((\bold{B}_e,\bold{B}^+_{\mathrm{dR}}), W(D))$ )
 the group of extension classes of $\bold{B}^{\dagger}_{\mathrm{rig},K}$ by $D$ (resp. 
 of the trivial $B$-pair $(\bold{B}_e,\bold{B}^+_{\mathrm{dR}})$ by $W(D)$), then we have the following canonical isomorphisms
 $$h_D:\mathrm{H}^1(K, D)\isom \mathrm{Ext}^1(\bold{B}^{\dagger}_{\mathrm{rig},K},D),\,\,\,h_{W(D)}:\mathrm{H}^1(K, W(D))\isom \mathrm{Ext}^1((\bold{B}_e,\bold{B}^+_{\mathrm{dR}}), W(D))$$ 
 by $\S$ 2.1 of \cite{Li08} and by $\S$ 2.1 of \cite{Na09},
 and we have the following commutative diagram
 $$
 \begin{CD}
 \mathrm{H}^1(K, D) @> \isom >> \mathrm{H}^1(K, W(D)) \\
 @VV h_D V @VV h_{W(D)}V \\
 \mathrm{Ext}^1(\bold{B}^{\dagger}_{\mathrm{rig},K},D) @> \isom >> \mathrm{Ext}^1((\bold{B}_e,\bold{B}^+_{\mathrm{dR}}), W(D))
\end{CD}
 $$
 by Theorem 5.11 of \cite{Na10}, where the isomorphism 
 $$\mathrm{Ext}^1(\bold{B}^{\dagger}_{\mathrm{rig},K}, D)\isom \mathrm{Ext}^1((\bold{B}_e,\bold{B}^+_{\mathrm{dR}}), W(D))$$ is given by 
 $$[0\rightarrow \bold{B}^{\dagger}_{\mathrm{rig},K}\rightarrow D'\rightarrow D\rightarrow 0]\mapsto [0\rightarrow (\bold{B}_e,\bold{B}^+_{\mathrm{dR}})\rightarrow W(D')\rightarrow W(D)\rightarrow 0],$$i.e. given by applying the functor $W(-)$. 
 
 Let $a\in \mathrm{H}^0(K, \bold{D}_{\mathrm{dif}}(D))=\bold{D}_{\mathrm{dif}}(D)^{\gamma_K=1}\isom \mathrm{H}^0(K,W_{\mathrm{dR}}(D))$. 
 By the above diagram, it suffices to show that the functor $W(-)$ sends 
 the extension corresponding to $\mathrm{exp}_{K,D}(a)$ to  the extension corresponding to $\mathrm{exp}_{K,W(D)}(a)$.
 
 Take $n$ sufficiently large such that 
 $a\in \bold{D}_{\mathrm{dif},n}(D)^{\gamma_K=1}$. By (1) of Lemma \ref{explicit}  and by the definition of the isomorphism 
 $\mathrm{H}^1(K, D)\isom \mathrm{Ext}^1(\bold{B}^{\dagger}_{\mathrm{rig},K}, D)$, if we take $\tilde{a}\in D^{(n)}[1/t]$ such that 
$\iota_m(\tilde{a})-a\in \bold{D}^+_{\mathrm{dif},m}(D)$ for any $m\geqq n$, the extension $D_a$ corresponding to $\mathrm{exp}_{K,D}(a)$ is 
 explicitly defined by 
$$[0\rightarrow D\xrightarrow{x\mapsto (x,0)} D\oplus \bold{B}^{\dagger}_{\mathrm{rig},K}e\xrightarrow{(x,ye)\mapsto y} \bold{B}^{\dagger}_{\mathrm{rig},K}\rightarrow 0]$$
 such that 
 $$\varphi((x,ye)):=(\varphi(x)+\varphi(y)(\gamma_K-1)\tilde{a}, \varphi(y)e),\,\,\, \gamma_K((x,ye)):=(\gamma_K(x)+\gamma_K(y)(\varphi-1)\tilde{a},\gamma_K(y)e)$$
 for any $(x,ye)\in D\oplus \bold{B}^{\dagger}_{\mathrm{rig},K}e$.
 
 On the other hands, the extension 
 $$W_a:=(W_{e,a}:=W_e(D)\oplus \bold{B}_ee_{\mathrm{crys}}, W_{\mathrm{dR},a}^+:=W^+_{\mathrm{dR}}(D)
 \oplus \bold{B}^+_{\mathrm{dR}}e_{\mathrm{dR}})$$
 corresponding to $\mathrm{exp}_{K,W(D)}(a)$ is defined by 
 $$g(x+ye_{\mathrm{crys}}):=g(x)+g(y)e_{\mathrm{crys}}, \,\,\,g(x'+y'e_{\mathrm{dR}}):=g(x')+g(y')e_{\mathrm{dR}}$$
 for any $g\in G_K, x\in W_e(D), x'\in W^+_{\mathrm{dR}}(D), y\in \bold{B}_e, y'\in \bold{B}^+_{\mathrm{dR}}$ and 
 the inclusion $$W^+_{\mathrm{dR},a}\hookrightarrow \bold{B}_{\mathrm{dR}}\otimes _{\bold{B}_e} W_{e,a}=W_{\mathrm{dR}}(D)\oplus \bold{B}_{\mathrm{dR}}e_{\mathrm{crys}}$$ is defined by 
 $$x+ye_{\mathrm{dR}}\mapsto (x+ya)+ye_{\mathrm{crys}}\,\, (x\in W^+_{\mathrm{dR}}(D), y\in \bold{B}^+_{\mathrm{dR}}).$$
 
 Let's show that $D(W_a)\isom D_a$ as an extension. 
 By the definition of $\tilde{a}$, we can easily check that 
 $\widetilde{D}^{(n)}(W_a)$ defined in Remark \ref{4.4} is isomorphic to 
 $$\widetilde{\bold{B}}^{\dagger,r_n}_{\mathrm{rig}}\otimes_{\bold{B}^{\dagger,r_n}_{\mathrm{rig},K}}D^{(n)}\oplus \widetilde{\bold{B}}^{\dagger,r_n}_{\mathrm{rig}}(e_{\mathrm{crys}}+\tilde{a})\subseteq 
 \widetilde{\bold{B}}^{\dagger,r_n}_{\mathrm{rig}}[1/t]\otimes_{\bold{B}_e} W_{e,a}$$ 
 and that $\widetilde{D}(W_a)$ contains a $(\varphi,\Gamma_K)$-module $D\oplus \bold{B}^{\dagger}_{\mathrm{rig},K}(e_{\mathrm{crys}}+\tilde{a})$ over $\bold{B}^{\dagger}_{\mathrm{rig},K}$, which is easily seen to be isomorphic to $D_a$, hence finishing the proof of the commutativity of (i).

 Finally, we prove the commutativity of (ii).
By Lemma \ref{2.15} (we note that we can show the $B$-pairs analogue of Lemma \ref{2.15} in the same way), it suffices to show the following 
 diagram is commutative
 \begin{itemize}
  \item[(ii)']
 $$
 \begin{CD}
 \mathrm{H}^0(K, \bold{D}_{\mathrm{dif}}(D^{\lor}(1))) @> \isom >>\mathrm{H}^0(K, W_{\mathrm{dR}}(D^{\lor}(1)))\\
 @ AA x\mapsto \iota_n(x)A @ AA\mathrm{can} A\\
 \mathrm{H}^0(K, D^{\lor}(1))@> \isom >>\mathrm{H}^0(K, W(D^{\lor}(1))),
 \end{CD}
 $$
 \end{itemize}
 but the commutativity of this diagram is trivial. We finish the proof of the theorem.

\end{proof}

\section{Perrin-Riou's big exponential map for de Rham $(\varphi,\Gamma)$- modules}
This section is the main part of this article. 
For any de Rham $(\varphi,\Gamma)$-module $D$, we construct a system of maps $\{\mathrm{Exp}_{D,h}\}_{h\gg0}$, which we call big exponential maps, and prove their important properties, i.e 
their interpolation formulae and the theorem $\delta(D)$. First two subsection is for preliminary. 
In $\S3.1$, we recall Pottharst's theory of the analytic Iwasawa cohomologies (\cite{Po12b}). In $\S3.2$, we recall Berger's construction of $p$-adic differential equations associated to de Rham $(\varphi,\Gamma)$-modules (\cite{Ber02}, \cite{Ber08b}). The next two subsection is the main part of this article. In $\S3.3$, we define the maps $\{\mathrm{Exp}_{D,h}\}_{h\gg0}$ and prove their interpolation formulae. In $\S3.4$, we formulate and prove
the theorem $\delta(D)$. In the final subsection $\S3.5$, we compare our big exponential maps and our theorem $\delta(D)$ with Perrin-Riou's or Pottharst's ones in the crystalline case.

\subsection{analytic Iwasawa cohomology}
In this subsection, we recall the results of Pottharst concerning 
analytic Iwasawa cohomologies of $(\varphi,\Gamma)$-modules over the Robba ring  (\cite{Po12b}) . 

Let $\Lambda:=\mathbb{Z}_p[[\Gamma_K]]:=\varprojlim_n \mathbb{Z}_p[\Gamma_K/\Gamma_{K_n}]$ be the Iwasawa 
algebra of $\Gamma_K$. If we decompose $\Gamma_K$ by $\Gamma_K\isom \Gamma_{K,\mathrm{tor}}\times \Gamma_{K,\mathrm{free}}$, where 
$\Gamma_{K,\mathrm{tor}}\subseteq \Gamma_K$ is the torsion subgroup of $\Gamma_K$ and 
$\Gamma_{K,\mathrm{free}}=\Gamma_K\cap \chi^{-1}(1+2p\mathbb{Z}_p)$ , then we have an 
isomorphism $\Lambda\isom \mathbb{Z}_p[\Gamma_{K,\mathrm{tor}}]\otimes_{\mathbb{Z}_p}\mathbb{Z}_p[[\Gamma_{K,\mathrm{free}}]]$. If we take a topological generator 
$\gamma\in \Gamma_{K,\mathrm{free}}$, then we also have a $\mathbb{Z}_p[\Gamma_{K,\mathrm{tor}}]$-algebra  isomorphism 
$\mathbb{Z}_p[\Gamma_{K,\mathrm{tor}}]\otimes_{\mathbb{Z}_p}\mathbb{Z}_p[[\Gamma_{K,\mathrm{free}}]]\isom \mathbb{Z}_p[\Gamma_{K,\mathrm{tor}}]\otimes_{\mathbb{Z}_p}
\mathbb{Z}_p[[T]]$ define by $1\otimes[\gamma]\mapsto 1\otimes (1+T)$.

 Let $\mathfrak{m}\subseteq \Lambda$ be the Jacobson 
radical of  $\Lambda$. 
For each $n\geqq 1$, we set $\Lambda_{n}:=\Lambda[\frac{\mathfrak{m}^n}{p}]^{\wedge}$  the $p$-adic completion of $\Lambda[\frac{\mathfrak{m}^n}{p}]$, which is an affinoid algebra over $\mathbb{Q}_p$. The natural map $\Lambda[\frac{\mathfrak{m}^{n+1}}{p}]\rightarrow \Lambda[\frac{\mathfrak{m}^n}{p}]$ induces 
a continuous map $\Lambda_{n+1}\rightarrow \Lambda_n$ for each $n\geqq 1$. We set $\Lambda_{\infty}:=\varprojlim_n \Lambda_n$. If we fix an isomorphism 
$\Lambda\isom \mathbb{Z}_p[\Gamma_{K,\mathrm{tor}}]\otimes_{\mathbb{Z}_p}\mathbb{Z}_p[[T]]$ as above, then this isomorphism 
naturally extends to a $\mathbb{Q}_p[\Gamma_{K,\mathrm{tor}}]$-algebra isomorphism $\Lambda_{\infty}\isom \mathbb{Q}_p[\Gamma_{K,\mathrm{tor}}]\otimes_{\mathbb{Q}_p}\bold{B}^+_{\mathrm{rig},\mathbb{Q}_p}$ , where the ring $\bold{B}^+_{\mathrm{rig},\mathbb{Q}_p}$ is defined by 
$$\bold{B}^+_{\mathrm{rig},\mathbb{Q}_p}:= \{f(T)=\sum_{n=0}^{\infty}a_nT^n| \,  a_n\in \mathbb{Q}_p\, \text{ and}\, f(T) \,\text{ is convergent on }\, 
0\leqq |T| <1\}.$$ 
We remark that the above isomorphism $\Lambda_{\infty}\isom \mathbb{Q}_p[\Gamma_{K,\mathrm{tor}}]\otimes_{\mathbb{Q}_p}\bold{B}^+_{\mathrm{rig},\mathbb{Q}_p}$ also depends on the choice of a topological generator $\gamma$ and 
highly non-canonical, and is only used to help the reader to understand the structure of $\Lambda_{\infty}$.

We define $\Lambda_n[\Gamma_K]$-modules $\widetilde{\Lambda}_n$  and $\widetilde{\Lambda}_{n}^{\iota}$ by 
$\widetilde{\Lambda}_n=\widetilde{\Lambda}_{n}^{\iota}=\Lambda_n$ as  $\Lambda_n$-module and $\gamma(\lambda):=[\gamma]\cdot\lambda$, 
$\gamma(\lambda'):=[\gamma^{-1}]\cdot\lambda'$ for 
$\lambda\in \widetilde{\Lambda}_n, \lambda'\in\widetilde{\Lambda}_{n}^{\iota}$ and $\gamma\in \Gamma_K$. 

Let $D$ be a $(\varphi,\Gamma_K)$-module over $\bold{B}^{\dagger}_{\mathrm{rig},K}$. For each $n\geqq 1$,  we define a $(\varphi,\Gamma_K)$-module $D\widehat{\otimes}_{\mathbb{Q}_p}\widetilde{\Lambda}^{\iota}_n$ over 
$\bold{B}^{\dagger}_{\mathrm{rig},K}\widehat{\otimes}_{\mathbb{Q}_p} \Lambda_n$ as follows (see $\S$ 2 of \cite{Po12b} for more precise definition).
We define $D\widehat{\otimes}_{\mathbb{Q}_p}\widetilde{\Lambda}^{\iota}_n:=D\widehat{\otimes}_{\mathbb{Q}_p}\Lambda_n$ 
as a $\bold{B}^{\dagger}_{\mathrm{rig},K}\widehat{\otimes}_{\mathbb{Q}_p}\Lambda_n$-module and define 
$\varphi(x\widehat{\otimes}\lambda):=\varphi(x)\widehat{\otimes}\lambda$, $\psi(x\widehat{\otimes}\lambda):=
\psi(x)\widehat{\otimes}\lambda$ and $\gamma(x\widehat{\otimes}\lambda):=\gamma(x)\widehat{\otimes}[\gamma^{-1}]\cdot\lambda$  for 
$x\in D, \lambda\in \Lambda_n$ and $\gamma\in \Gamma_K$. 

For each $n\geqq 1$, we define two complexes 
$C^{\bullet}_{\varphi,\gamma_K}(D\widehat{\otimes}_{\mathbb{Q}_p}\widetilde{\Lambda}^{\iota}_n)$ and
$C^{\bullet}_{\psi,\gamma_K}(D\widehat{\otimes}_{\mathbb{Q}_p}\widetilde{\Lambda}^{\iota}_n)$ and define the natural map 
of complexes $C^{\bullet}_{\varphi,\gamma_K}(D\widehat{\otimes}_{\mathbb{Q}_p}\widetilde{\Lambda}^{\iota}_n)\rightarrow C^{\bullet}_{\psi,\gamma_K}(D\widehat{\otimes}_{\mathbb{Q}_p}\widetilde{\Lambda}^{\iota}_n)$ in the same way as those for $D$. 
We define
$\mathrm{H}^q(K, D\widehat{\otimes}_{\mathbb{Q}_p}\widetilde{\Lambda}^{\iota}_n):=\mathrm{H}^q(C^{\bullet}_{\varphi,\gamma_K}(D\widehat{\otimes}_{\mathbb{Q}_p}\widetilde{\Lambda}^{\iota}_n))$, which is a $\Lambda_n$-module. The natural map $\Lambda_{n+1}\rightarrow \Lambda_n$ induces a natural map $D\widehat{\otimes}_{\mathbb{Q}_p}\widetilde{\Lambda}^{\iota}_{n+1}
\rightarrow D\widehat{\otimes}_{\mathbb{Q}_p}\widetilde{\Lambda}^{\iota}_n$, and this map 
induces $\mathrm{H}^q(K, D\widehat{\otimes}_{\mathbb{Q}_p}\widetilde{\Lambda}^{\iota}_{n+1})\rightarrow \mathrm{H}^q(K, D\widehat{\otimes}_{\mathbb{Q}_p}\widetilde{\Lambda}^{\iota}_n)$. 
Following \cite{Po12b}, we define the analytic Iwasawa cohomology of $D$ as follows.
\begin{defn}
Let $D$ be a $(\varphi,\Gamma_K)$-module over $\bold{B}^{\dagger}_{\mathrm{rig},K}$, 
$q\geqq 0$ an integer. We define the 
$q$-th analytic Iwasawa cohomology of $D$
by $$\bold{H}^q_{\mathrm{Iw}}(K, D):=\varprojlim_n\mathrm{H}^q(K, D\widehat{\otimes}_{\mathbb{Q}_p}\widetilde{\Lambda}^{\iota}_n),$$
which is a $\Lambda_{\infty}$-module.

\end{defn}

Because we have a decomposition
$\mathbb{Q}_p[\Gamma_{K,\mathrm{tor}}]=\oplus_{\eta\in \widehat{\Gamma}_{K,\mathrm{tor}}}\mathbb{Q}_p\alpha_{\eta}$, 
where $\widehat{\Gamma}_{K,\mathrm{tor}}$ is the character group of $\Gamma_{K,\mathrm{tor}}$ and $\alpha_{\eta}$ is the idempotent corresponding to $\eta$, 
we also have $\Lambda_{\infty}=\bigoplus_{\eta\in \hat{\Gamma}_{K,\mathrm{tor}}}
\Lambda_{\infty}\alpha_{\eta}$ and each $\Lambda_{\infty}\alpha_{\eta}$ is non-canonically isomorphic 
to $\bold{B}^+_{\mathrm{rig},\mathbb{Q}_p}$.
Let $M$ be a $\Lambda_{\infty}$-module.  Using this decomposition,
we obtain a decomposition $M=\bigoplus_{\eta\in \widehat{\Gamma}_{K,\mathrm{tor}}}M_{\eta}$, where we define $M_{\eta}:=\alpha_{\eta}M$ which is a $\Lambda_{\infty}\alpha_{\eta}$-module. 
For a $\bold{B}^+_{\mathrm{rig},\mathbb{Q}_p}$-module $N$, we define $N_{\mathrm{tor}}:=\{x\in N| ax=0 \text{ for some non zero }\,a\in \bold{B}^+_{\mathrm{rig},\mathbb{Q}_p}\}$. 
For a $\Lambda_{\infty}$-module $M$, we define $M_{\mathrm{tor}}:=\bigoplus_{\eta\in \widehat{\Gamma}_{K,\mathrm{tor}}}(M_{\eta})_{\mathrm{tor}}$.

 As for the fundamental properties of $\bold{H}^q_{\mathrm{Iw}}(K, D)$, Pottharst proved the following  results, which is a generalization of Perrin-Riou's results (\cite{Per94}) in 
the case of $p$-adic Galois representations. 

\begin{thm}\label{3.2}
Let $D$ be a $(\varphi,\Gamma_K)$-module over $\bold{B}^{\dagger}_{\mathrm{rig},K}$ of rank $d$. Then we have the following,
\begin{itemize}
\item[(1)]For each $n\geqq 1$ and $q\geqq 0$, the natural map 
$$\mathrm{H}^q(K, D\widehat{\otimes}_{\mathbb{Q}_p}\widetilde{\Lambda}^{\iota}_{n+1})
\rightarrow \mathrm{H}^q(K, D\widehat{\otimes}_{\mathbb{Q}_p}\widetilde{\Lambda}^{\iota}_n)$$ induces an isomorphism of $\Lambda_n$-modules
$$\mathrm{H}^q(K, D\widehat{\otimes}_{\mathbb{Q}_p}\widetilde{\Lambda}^{\iota}_{n+1})\otimes_{\Lambda_{n+1}}\Lambda_n\isom 
\mathrm{H}^q(K, D\widehat{\otimes}_{\mathbb{Q}_p}\widetilde{\Lambda}^{\iota}_n),$$
\item[(2)]$\bold{H}^q_{\mathrm{Iw}}(K,D)=0$ if $q\not= 1,2$,
\item[(3)]$\bold{H}^1_{\mathrm{Iw}}(K, D)_{\mathrm{tor}}$ and $\bold{H}^2_{\mathrm{Iw}}(K, D)$ are finite dimensional 
$\mathbb{Q}_p$-vector spaces,
\item[(4)]$\bold{H}^1_{\mathrm{Iw}}(K, D)/\bold{H}^1_{\mathrm{Iw}}(K,D)_{\mathrm{tor}}$ is a finite free $\Lambda_{\infty}$-module of rank $d[K:\mathbb{Q}_p]$.

\end{itemize}

\end{thm}
\begin{proof}
This is Theorem 2.6 and Proposition 2.9 of \cite{Po12b}.

\end{proof}


Let $A$ be a $\mathbb{Q}_p$-affinoid algebra, 
 $\delta:\Gamma_K\rightarrow A^{\times}$ a continuous homomorphism. 
We define $A(\delta):=A e_{\delta}$ the free rank one $A$-module with the base $e_{\delta}$ 
with an $A$-linear 
$\Gamma_K$ action by $\gamma(e_{\delta}):=\delta(\gamma)e_{\delta}$ for 
 $\gamma\in \Gamma_K$.
Then the continuous $\mathbb{Q}_p$-algebra homomorphism 
$f_{\delta}:\Lambda_{\infty}\rightarrow A$ which is defined by 
$f_{\delta}([\gamma]):=\delta(\gamma)^{-1}$ for any $\gamma\in \Gamma_K$ induces 
the isomorphism 
$$D\hat{\otimes}_{\mathbb{Q}_p}(\widetilde{\Lambda}_{\infty}^{\iota}\otimes_{\Lambda_{\infty},f_{\delta}}A)\isom D\hat{\otimes}_{\mathbb{Q}_p}A(\delta):
x\hat{\otimes}(\lambda\otimes a)\mapsto x\hat{\otimes} f_{\delta}(\lambda)a e_{\delta}$$
of $(\varphi,\Gamma_K)$-modules over $\bold{B}^{\dagger}_{\mathrm{rig},K}\hat{\otimes}_{\mathbb{Q}_p}A$.
This isomorphism induces the canonical projection map 
$$\bold{H}^q_{\mathrm{Iw}}(K, D)\rightarrow \mathrm{H}^q(K, D\hat{\otimes}_{\mathbb{Q}_p}(\widetilde{\Lambda}_{\infty}^{\iota}\otimes_{\Lambda_{\infty},f_{\delta}}A))\isom 
\mathrm{H}^q(K, D\hat{\otimes}_{\mathbb{Q}_p}A(\delta)).$$

For each  $L=K_n$ ($n\geqq 1$) or $L=K$, $k\in \mathbb{Z}$ and $q\geqq 0$, as a special case 
of the above projection map, 
we define the canonical map
$$\mathrm{pr}_{L, D(k)}:\bold{H}^q_{\mathrm{Iw}}(K, D)\rightarrow \mathrm{H}^q(L, D(k))$$ 
as follows. 
First, define the continuous homomorphism 
$\delta_{L}:\Gamma_K\rightarrow \mathbb{Q}_p[\Gamma_K/\Gamma_L]^{\times}: 
\gamma\mapsto [\overline{\gamma}]^{-1}$, then for each $k\in \mathbb{Z}$, we obtain 
the projection map 
$$\bold{H}^q_{\mathrm{Iw}}(K, D)\rightarrow \mathrm{H}^q(K, D\otimes_{\mathbb{Q}_p}
\mathbb{Q}_p[\Gamma_K/\Gamma_L](\delta_L\chi^{k}))$$ 
associated to the character $\delta_L\chi^{k}$. Using the isomorphism 
$D\otimes_{\mathbb{Q}_p}\mathbb{Q}_p[\Gamma_K/\Gamma_L](\delta_L\chi^k)\isom 
D(k)\otimes_{\mathbb{Q}_p}\mathbb{Q}_p\widetilde{[\Gamma_K/\Gamma_L]}^{\iota}: 
x\otimes a e_{\delta_L\chi^k}\mapsto (x\otimes e_k)\otimes a$ (where $\mathbb{Q}_p\widetilde{[\Gamma_K/\Gamma_L]}^{\iota}$ is defined similarly as $\widetilde{\Lambda}_{\infty}^{\iota}$)
, we define 
\[
\begin{array}{ll}
\mathrm{pr}_{L, D(k)}:\bold{H}^q_{\mathrm{Iw}}(K, D)&\rightarrow \mathrm{H}^q(K, D\otimes_{\mathbb{Q}_p}
\mathbb{Q}_p[\Gamma_K/\Gamma_L](\delta_L\chi^{k}))\\
&\isom \mathrm{H}^q(K, D(k)\otimes_{\mathbb{Q}_p}\mathbb{Q}_p\widetilde{[\Gamma_K/\Gamma_L]}^{\iota})\\
&\isom \mathrm{H}^q(L, D(k)),
\end{array}
\]
where the last isomorphism is the canonical one induced by the Shapiro's lemma (see  Theorem 2.2 of \cite{Li08}).

For each $k\in \mathbb{Z}$, we define a canonical isomorphism
$$f_{D,k}:\bold{H}^q_{\mathrm{Iw}}(K, D)\isom \bold{H}^q_{\mathrm{Iw}}(K, D(k))$$
of $\mathbb{Q}_p$-vector spaces
as follows. We first define continuous $\mathbb{Q}_p$-algebra isomorphisms
$f_k:\Lambda_0\isom \Lambda_0$ ( $\Lambda_{0}=\Lambda,\Lambda_n,\Lambda_{\infty}$) by
 $f_k([\gamma]):=\chi(\gamma)^{-k}[\gamma]$ for 
any $\gamma\in \Gamma_K$. Using $f_k$, for each $n\geqq 1$, we define a continuous $\bold{B}^{\dagger}_{\mathrm{rig},K}$-linear isomorphism
 $D\widehat{\otimes}_{\mathbb{Q}_p}\widetilde{\Lambda}_n^{\iota}\isom D(k)\widehat{\otimes}_{\mathbb{Q}_p}\widetilde{\Lambda}_n^{\iota}
:x\widehat{\otimes}\lambda\mapsto (x\otimes e_k)\widehat{\otimes}f_k(\lambda)$.
 This map commutes with $\varphi$ and $\Gamma_K$-actions, hence 
induces an  isomorphism $C^{\bullet}_{\varphi,\gamma_K}(D\widehat{\otimes}_{\mathbb{Q}_p}\widetilde{\Lambda}_n^{\iota})
\isom C^{\bullet}_{\varphi,\gamma_K}(D(k)\widehat{\otimes}_{\mathbb{Q}_p}\widetilde{\Lambda}_n^{\iota})$ of complexes of $\mathbb{Q}_p$-vector spaces, hence also 
induces an isomorphism $f_{D,k}:\bold{H}^q_{\mathrm{Iw}}(K, D)\isom \bold{H}^q_{\mathrm{Iw}}(K, D(k))$ of $\mathbb{Q}_p$-vector spaces for each $q$.

Using the $\psi$-complex $C^{\bullet}_{\psi,\gamma_K}(D\widehat{\otimes}_{\mathbb{Q}_p}\widetilde{\Lambda}^{\iota}_{n})$, 
we can describe $\bold{H}^q_{\mathrm{Iw}}(K, D)$ in a more explicit way as follows.

\begin{thm}\label{3.3}
Let $D$ be a $(\varphi,\Gamma_K)$-module over $\bold{B}^{\dagger}_{\mathrm{rig},K}$.

\begin{itemize}
\item[(1)]For each $n\geqq 1$, the map 
$$(\gamma_K-1):(D\widehat{\otimes}_{\mathbb{Q}_p}\widetilde{\Lambda}^{\iota}_n)^{\Delta_K,\psi=0}
\rightarrow (D\widehat{\otimes}_{\mathbb{Q}_p}\widetilde{\Lambda}^{\iota}_n)^{\Delta_K,\psi=0}$$ is isomorphism. In particular, the natural map 
$$C^{\bullet}_{\varphi,\gamma_K}(D\widehat{\otimes}_{\mathbb{Q}_p}\widetilde{\Lambda}^{\iota}_{n})\rightarrow C^{\bullet}_{\psi,\gamma_K}(D\widehat{\otimes}_{\mathbb{Q}_p}\widetilde{\Lambda}^{\iota}_{n})$$ is quasi-isomorphism.

\item[(2)]
The complex 
$$C^{\bullet}_{\psi}(D):[D\xrightarrow{\psi-1}D]$$ 
of $\Lambda_{\infty}$-modules concentrated in degree 
$[1, 2]$ calculates $\mathrm{H}^{q}_{\mathrm{Iw}}(K, D)$, i.e. 
we have  functorial isomorphisms of $\Lambda_{\infty}$-modules
$$\iota_D:D^{\psi=1}\isom \bold{H}^1_{\mathrm{Iw}}(K, D)$$ 
and 
$$D/(\psi-1)D\isom \bold{H}^2_{\mathrm{Iw}}(K, D).$$
\end{itemize}

\end{thm}
\begin{proof}
This is Theorem 2.6  of \cite{Po12b}.
\end{proof}
\begin{rem}
Let $D$ be a $(\varphi,\Gamma_K)$-module over $\bold{B}^{\dagger}_{\mathrm{rig},K}$. Then one has  that the structure of 
$\mathbb{Q}_p[\Gamma_K]$-module on $D$ uniquely extends to a structure of continuous $\Lambda_{\infty}$-module (
see Proposition 2.13 of \cite{Ch12}).

\end{rem}

We define $p_{\Delta_K}:=\frac{1}{|\Delta_K|}\sum_{g\in \Delta_K}g\in 
\mathbb{Q}[\Delta_K]$, $\mathrm{log}_0(a):=\frac{\mathrm{log}(a)}{p^{v_p(\mathrm{log}(a))}}\in \mathbb{Z}_p^{\times}$ for any $a\in \mathbb{Z}_p^{\times}$. 
For $q=1$,  the isomorphism 
$$\iota_D:D^{\psi=1}\isom \bold{H}^1_{\mathrm{Iw}}(K, D)$$ is defined as the composition of the 
following isomorphisms,
\[
\begin{array}{ll}
\iota_D:D^{\psi=1}&\isom \varprojlim_{n}((D\widehat{\otimes}_{\mathbb{Q}_p}\widetilde{\Lambda}^{\iota}_{n})^{\Delta_K}/(\gamma_K-1)(D\widehat{\otimes}_{\mathbb{Q}_p}\widetilde{\Lambda}^{\iota}_{n})^{\Delta_K})^{\psi=1}\\
&\isom \varprojlim_n \mathrm{H}^1(C^{\bullet}_{\psi,\gamma_K}(D\hat{\otimes}_{\mathbb{Q}_p}\widetilde{\Lambda}_n^{\iota})) \\
&\isom \varprojlim_n \mathrm{H}^1(C^{\bullet}_{\varphi,\gamma_K}(D\hat{\otimes}_{\mathbb{Q}_p}\widetilde{\Lambda}_n^{\iota}))=\bold{H}^1_{\mathrm{Iw}}(K, D),
\end{array}
\]
 where each isomorphism is defined as follows.
 The first isomorphism is defined by 
  $x\mapsto (|\Gamma_{K, \mathrm{tor}}|\mathrm{log}_0(\chi(\gamma_K))\overline{p_{\Delta_K}(x\widehat{\otimes} 1)})_{n\geqq 1,}$ for any $x\in D^{\psi=1}$. The second isomorphism 
  is defined as the limit of
$$((D\widehat{\otimes}_{\mathbb{Q}_p}\widetilde{\Lambda}^{\iota}_{n})^{\Delta_K}/(\gamma_K-1)(D\widehat{\otimes}_{\mathbb{Q}_p}\widetilde{\Lambda}^{\iota}_{n})^{\Delta_K})^{\psi=1}\isom 
\mathrm{H}^1(C^{\bullet}_{\psi,\gamma_K}(D\widehat{\otimes}_{\mathbb{Q}_p}\widetilde{\Lambda}^{\iota}_{n})): 
\overline{x}\mapsto [x, y],$$ 
where $x\in(D\widehat{\otimes}_{\mathbb{Q}_p}\widetilde{\Lambda}^{\iota}_{n})^{\Delta_K}$ is a lift of $\overline{x}$ and 
 $y\in (D\widehat{\otimes}_{\mathbb{Q}_p}\widetilde{\Lambda}^{\iota}_{n})^{\Delta_K}$ is an
element such that $(\psi-1)x=(\gamma_K-1)y$. The third isomorphism is induced by Theorem 
\ref{3.3} (1). 

For each $k\in \mathbb{Z}$, we have the following commutative diagram
$$
\begin{CD}
D^{\psi=1}@> \iota_D>> \bold{H}^1_{\mathrm{Iw}}(K,D)  \\
  @VV x\mapsto x\otimes e_k V@VV f_{D,k} V   \\
D(k)^{\psi=1}@> \iota_{D(k)}>> \bold{H}^1_{\mathrm{Iw}}(K, D(k)). 
  \end{CD}
  $$

\subsection{$p$-adic differential equations associated to de Rham $(\varphi,\Gamma)$-modules}
In this subsection, we recall the results of Berger concerning the construction of 
$p$-adic differential equations associated to 
de Rham $(\varphi,\Gamma)$-modules. 
Let $D$ be a de Rham $(\varphi,\Gamma_K)$-module over $\bold{B}^{\dagger}_{\mathrm{rig},K}$. Then 
 we have an isomorphism $K_n((t))\otimes_{K}\bold{D}^K_{\mathrm{dR}}(D)\isom \bold{D}_{\mathrm{dif},n}(D)$ for each $n\geqq n(D)$. Hence 
$K_n[[t]]\otimes_K \bold{D}^K_{\mathrm{dR}}(D)$ is a $\Gamma_K$-stable $K_n[[t]]$-lattice of 
$\bold{D}_{\mathrm{dif}, n}(D)$ for each $n\geqq n(D)$. 
Define $\nabla_0:=\frac{\mathrm{log}(\gamma)}{\mathrm{log}(\chi(\gamma))}\in \Lambda_{\infty}$ where 
$\gamma$ is a non-torsion element of $\Gamma_K$, which is independent of the choice of $\gamma$. For each $i\in\mathbb{Z}$, we define $\nabla_i:=\nabla_0-i\in \Lambda_{\infty}$.
The operator $\nabla_0$ satisfies the Leibnitz rule
$\nabla_0(fx)=\nabla_0(f)x+f\nabla_0(x)$ for any $f\in \bold{B}^{\dagger}_{\mathrm{rig},K}$, $x\in D$. 
When $K=F$ is unramified over $\mathbb{Q}_p$, then we have
$\nabla_0(f(T))=t(T+1)\frac{df(T)}{dT}$ for $f(T)\in \bold{B}^{\dagger}_{\mathrm{rig},F}$.  For the case of 
general
$K$, let $P(X)\in \bold{B}^{\dagger}_{\mathrm{rig},K_0'}[T]$ be the monic minimal polynomial of $\pi_K\in \bold{B}^{\dagger}_{\mathrm{rig},K}$ 
over $\bold{B}^{\dagger}_{\mathrm{rig},K_0'}$. Calculating $0=\nabla_0(P(\pi_K))$, we obtain 
$\nabla_0(\pi_K)=-\frac{1}{\frac{dP}{dX}(\pi_K)}\nabla_0(P)(\pi_K)$, where
 we define $\nabla_0(P)(X):=\sum_{i=0}^m \nabla_0(a_i)X^m$ for any
$P(X)=\sum_{i=0}^ma_mX^m\in B^{\dagger}_{\mathrm{rig},K_0'}[X]$. We denote by
$\widehat{\Omega}_{\bold{B}^{\dagger}_{\mathrm{rig},K}/K_0'}$ 
the continuous differentials. Then one has $\widehat{\Omega}_{\bold{B}^{\dagger}_{\mathrm{rig},K}/K_0'}= \bold{B}^{\dagger}_{\mathrm{rig},K}dT$ by the \'etaleness of the inclusion $\bold{B}^{\dagger}_{\mathrm{rig},K_0'}\subseteq \bold{B}^{\dagger}_{\mathrm{rig},K}$.


\begin{thm}\label{3.5}
Let $D$ be a de Rham $(\varphi,\Gamma_K)$-module over $\bold{B}^{\dagger}_{\mathrm{rig},K}$ of rank $d$. 
For each $n\geqq n(D)$, we define
$$\bold{N}^{(n)}_{\mathrm{rig}}(D):=\{x\in D^{(n)}[1/t]| \iota_m(x)\in K_m[[t]]\otimes_K 
\bold{D}^K_{\mathrm{dR}}(D)\,\text{for any }\,m\geqq n\}.$$ 
Then $\bold{N}_{\mathrm{rig}}(D):=\varinjlim_n \bold{N}^{(n)}_{\mathrm{rig}}(D)$ 
is a $(\varphi,\Gamma_K)$-module over $\bold{B}^{\dagger}_{\mathrm{rig},K}$ of rank $d$ which satisfies the 
following,
\begin{itemize}
\item[(1)]$\bold{N}_{\mathrm{rig}}(D)[1/t]=D[1/t]$,
\item[(2)]$\bold{D}^+_{\mathrm{dif}, n}(\bold{N}_{\mathrm{rig}}(D))= K_n[[t]]\otimes_K \bold{D}^K_{\mathrm{dR}}(D)$ for any 
$n\geqq n (D)$,
\item[(3)]$\nabla_0(\bold{N}_{\mathrm{rig}}(D))\subseteq t\bold{N}_{\mathrm{rig}}(D)$.
\end{itemize}
In fact, the properties (1) and (2) uniquely characterize $\bold{N}_{\mathrm{rig}}(D)$.

\end{thm}
\begin{proof}
See, for example, Theorem 5.10 of \cite{Ber02}, Theorem 3.2.3 of \cite{Ber08b}.
\end{proof}

By the condition (3) in the above theorem, we can define a differential operator
$$\partial:=\frac{1}{t}\nabla_0: \bold{N}_{\mathrm{rig}}(D)
\rightarrow \bold{N}_{\mathrm{rig}}(D)$$ which satisfies that $\partial \varphi=p\varphi\partial$ 
and $\partial \gamma=\chi(\gamma)\gamma\partial $ for any $\gamma\in \Gamma_K$. 
In particular, we can equip 
$\bold{N}_{\mathrm{rig}}(D)$ with a structure of 
a $p$-adic differential equation over $\bold{B}^{\dagger}_{\mathrm{rig},K}$ with Frobenius structure 
by $$\bold{N}_{\mathrm{rig}}(D)\rightarrow \bold{N}_{\mathrm{rig}}(D)\otimes_{\bold{B}^{\dagger}_{\mathrm{rig},K}}\widehat{\Omega}_{\bold{B}^{\dagger}_{\mathrm{rig},K}/K_0'}: x\mapsto \partial(x)dT,$$\
where we define $\varphi(dT):=pdT$ and $\gamma(dT):=\chi(\gamma)dT$ for any $\gamma\in \Gamma_K$.

Moreover, 
because we have an isomorphism 
\begin{align*}
\bold{N}_{\mathrm{rig}}(D(-1))&\isom \bold{N}_{\mathrm{rig}}(D)\otimes_{\bold{B}^{\dagger}_{\mathrm{rig},K}} \bold{N}_{\mathrm{rig}}(\bold{B}^{\dagger}_{\mathrm{rig},K}(-1))\\
&=\bold{N}_{\mathrm{rig}}(D)\otimes_{\bold{B}^{\dagger}_{\mathrm{rig},K}}t\bold{B}^{\dagger}_{\mathrm{rig},K}(-1)
=t\bold{N}_{\mathrm{rig}}(D)(-1),
\end{align*}
we obtain a $\varphi$-equivariant map 
$$\widetilde{\partial}: \bold{N}_{\mathrm{rig}}(D)\rightarrow \bold{N}_{\mathrm{rig}}(D(-1)): 
x\mapsto \nabla_0(x)\otimes e_{-1}.$$

\subsection{construction of $\mathrm{Exp}_{D,h}$ for de Rham $(\varphi,\Gamma)$-modules}
This subsection is the main part of this article. We generalize Perrin-Riou's big exponential map to all the de Rham 
$(\varphi,\Gamma)$-modules, and prove that this map interpolates the exponential map and the dual exponential map of cyclotomic twists of a given $(\varphi,\Gamma)$-module.

We first prove the following easy lemma. We remark that a stronger version (in the crystalline case) 
appears in $\S$2.2 of \cite{Ber03}.

\begin{lemma}\label{3.6}

Let $D$ be a de Rham $(\varphi,\Gamma_K)$-module over $\bold{B}^{\dagger}_{\mathrm{rig},K}$ and let
$h\in \mathbb{Z}_{\geqq 1}$ such that $\mathrm{Fil}^{-h}\bold{D}^K_{\mathrm{dR}}(D)=\bold{D}^K_{\mathrm{dR}}(D)$. 
Then we have $$\nabla_{h-1}\cdot \nabla_{h-2} \cdots \nabla_1\cdot \nabla_0(\bold{N}_{\mathrm{rig}}(D))
\subseteq D.$$
\end{lemma}
\begin{proof}
By (3) of Theorem \ref{3.5} and by the formula $\nabla_i(t^ix)=t^i\nabla_0(x)$ for each
 $i\in \mathbb{Z}$, we obtain an inclusion $\nabla_{h-1}\cdots \nabla_0(\bold{N}_{\mathrm{rig}}(D))\subseteq 
t^h \bold{N}_{\mathrm{rig}}(D)$. Hence, it suffices to show that 
$t^h\bold{N}_{\mathrm{rig}}(D)$ is contained in $D$. By (2) of Theorem \ref{3.5},
we have $\bold{D}^+_{\mathrm{dif}, n}(t^h \bold{N}_{\mathrm{rig}}(D))=t^h K_n[[t]]\otimes_{K}\bold{D}^K_{\mathrm{dR}}(D)$ for each $n\geqq n(D)$. 
Hence, by the assumption on $h$, $t^h K_n[[t]]\otimes_{K}\bold{D}^K_{\mathrm{dR}}(D)$ is contained in $\mathrm{Fil}^0(K_n((t))\otimes_K\bold{D}^K_{\mathrm{dR}}(D))=\bold{D}^+_{\mathrm{dif}, n}(D)$ for any $n\geqq n(D)$. Hence $t^h\bold{N}_{\mathrm{rig}}(D)$ is also contained in $D$.

\end{proof}

\begin{defn}
Let $D$ be a de Rham $(\varphi,\Gamma_K)$-module over $\bold{B}^{\dagger}_{\mathrm{rig},K}$ and let $h\in\mathbb{Z}_{\geqq 1}$ 
such that 
$\mathrm{Fil}^{-h}\bold{D}^K_{\mathrm{dR}}(D)=\bold{D}^K_{\mathrm{dR}}(D)$. Then we define a 
$\Lambda_{\infty}$-linear map 

$$\mathrm{Exp}_{D, h}: \bold{N}_{\mathrm{rig}}(D)^{\psi=1}\rightarrow \bold{H}^1_{\mathrm{Iw}}(K, D): x\mapsto 
\iota_D (\nabla_{h-1}\cdots\nabla_0(x)),$$ where $\iota_D:
 D^{\psi=1}\isom \bold{H}^1_{\mathrm{Iw}}(K, D)$ is the isomorphism defined in Theorem \ref{3.3}.

\end{defn}
\begin{rem}
This definition is strongly influenced by the work of Berger (\cite{Ber03}), where he re-constructed Perrin-Riou's big exponential 
map using $(\varphi,\Gamma)$-modules over the Robba ring. Using the work in 
$\S$ 3.5 below, comparing $\bold{D}^K_{\mathrm{crys}}(D)\otimes_{\mathbb{Q}_p}(\bold{B}^+_{\mathrm{rig},\mathbb{Q}_p})^{\psi=0}$  with $\bold{N}_{\mathrm{rig}}(D)^{\psi=1}$, one sees that 
in the crystalline \'etale case this our map is essentially the same as Berger's, reinterpreted in
terms of $\bold{N}_{\mathrm{rig}}(D)^{\psi=1}$. Therefore, we regard our map as a generalization of his.

\end{rem}

Next, we define a projection  map for each $L=K$ or $L=K_n$ ($n\geqq 1$)
$$T_L: \bold{N}_{\mathrm{rig}}(D)^{\psi=1}\rightarrow \bold{D}^{L}_{\mathrm{dR}}(D)$$
as follows. Because we have $\psi(\bold{N}^{(m+1)}_{\mathrm{rig}}(D))\subseteq \bold{N}^{(m)}_{\mathrm{rig}}(D)$ for any sufficiently large
$m$, we have an equality $\bold{N}_{\mathrm{rig}}(D)^{\psi=1}=\bold{N}^{(m)}_{\mathrm{rig}}(D)^{\psi=1}$ for any $m\gg0$. 
Let $n\geqq 1$ be any integer. We take a sufficiently large $m\geqq n$ as above.  Then we define $T_L$ for $L=K_n$ or $L=K$ by
\begin{align*}
T_{L}: \bold{N}_{\mathrm{rig}}(D)^{\psi=1}=\bold{N}^{(m)}_{\mathrm{rig}}(D)^{\psi=1}&\xrightarrow{\iota_m} K_m[[t]]\otimes_{K}\bold{D}^{K}_{\mathrm{dR}}(D)\\
&\xrightarrow{t\mapsto 0} \bold{D}^{K_m}_{\mathrm{dR}}(D)\xrightarrow{\frac{1}{[K_m: L]}\mathrm{Tr}_{K_m/L}} \bold{D}^{L}_{\mathrm{dR}}(D).
\end{align*}
Because we have a commutative diagram
$$
\begin{CD}
\bold{N}^{(m+1)}_{\mathrm{rig}}(D)@> \iota_{m+1}>> K_{m+1}[[t]]\otimes_K\bold{D}^K_{\mathrm{dR}}(D) @> t\mapsto 0 >> \bold{D}^{K_{m+1}}_{\mathrm{dR}}(D)  \\
  @VV\psi V@VV \frac{1}{p}\mathrm{Tr}_{K_{m+1}/K_m} V  @VV \frac{1}{p}\mathrm{Tr}_{K_{m+1}/K_m} V  \\
\bold{N}^{(m)}_{\mathrm{rig}}(D)@> \iota_{m}>> K_{m}[[t]]\otimes_K\bold{D}^K_{\mathrm{dR}}(D) @> t\mapsto 0 >> \bold{D}^{K_{m}}_{\mathrm{dR}}(D) , 
  \end{CD}
  $$
the definition of $T_L$ does not depend on the choice of $m\gg n$.

The following lemma directly follows from the definition.
\begin{lemma}\label{3.9}
Let $D$ be a de Rham $(\varphi,\Gamma_K)$-module over $\bold{B}^{\dagger}_{\mathrm{rig},K}$ and let $h\in \mathbb{Z}_{\geqq 1}$ such that 
$\mathrm{Fil}^{-h}\bold{D}^K_{\mathrm{dR}}(D)=\bold{D}^K_{\mathrm{dR}}(D)$. Then we have the following.
\begin{itemize}
\item[(1)]$\mathrm{Exp}_{D,h+1}=\nabla_{h}\mathrm{Exp}_{D, h}$.
\item[(2)]The following diagram is commutative
$$
\begin{CD}
\bold{N}_{\mathrm{rig}}(D(1))^{\psi=1}@> \widetilde{\partial}>> \bold{N}_{\mathrm{rig}}(D)^{\psi=1}  \\
  @VV\mathrm{Exp}_{D(1), h+1} V @VV \mathrm{Exp}_{D, h} V \\
\bold{H}^1_{\mathrm{Iw}}(K, D(1))@> f_{D(1), -1} >> \bold{H}^1_{\mathrm{Iw}}(K, D), 
  \end{CD}
  $$
  where $f_{D(1),-1}:\bold{H}^1_{\mathrm{Iw}}(K, D(1))\isom \bold{H}^1_{\mathrm{Iw}}(K, D)$ is the canonical isomorphism 
  defined in $\S  3.1$.

\end{itemize}

\end{lemma}

The main theorem of this paper is the following, which says that 
$\mathrm{Exp}_{D,h}$ interpolates $\mathrm{exp}_{L,D(k)}$ for suitable $k\geqq -(h-1)$ and $\mathrm{exp}^{*}_{L,D^{\lor}(1-k)}$ for any $k\leqq -h$  for any $L=K_n, K$. 
According to the comparison of $\bold{D}^K_{\mathrm{crys}}(D)\otimes_{\mathbb{Q}_p}(\bold{B}^+_{\mathrm{rig},\mathbb{Q}_p})^{\psi=0}$ with $\bold{N}_{\mathrm{rig}}(D)^{\psi=1}$ provided in $\S$ 3.5, we see this theorem as a generalization of Berger's theorem (Theorem 2.10 of \cite{Ber03}) in the crystalline \'etale case. 


\begin{thm}\label{3.10}
For any  $h\in \mathbb{Z}_{\geqq 1}$ such that $\mathrm{Fil}^{-h}\bold{D}^K_{\mathrm{dR}}(D)=
\bold{D}^K_{\mathrm{dR}}(D)$, 
$\mathrm{Exp}_{D,h}$ satisfies the following formulae. 
\begin{itemize}
\item[(1)]If $k\geqq 1$ and there exists $x_k\in \bold{N}_{\mathrm{rig}}(D(k))^{\psi=1}$ such that $\widetilde{\partial}^k(x_k)=x$, or if $0\geqq k\geqq -(h-1)$ and $x_k:=\widetilde{\partial}^{-k}(x)$, then 
$$\mathrm{pr}_{L, D(k)}(\mathrm{Exp}_{D,h}(x))=\frac{(-1)^{h+k-1}(h+k-1)!|\Gamma_{L, \mathrm{tor}}|}{p^{m(L)}}\mathrm{exp}_{L,D(k)}(T_L(x_k)),$$
\item[(2)]if $-h\geqq k$, then
$$\mathrm{exp}^{*}_{L,D^{\lor}(1-k)}(\mathrm{pr}_{L,D(k)}(\mathrm{Exp}_{D, h}(x))=\frac{|\Gamma_{L,\mathrm{tor}}|}{(-h-k)!  p^{m(L)}}
T_L(\widetilde{\partial}^{-k}(x)),$$

\end{itemize}
for any $L=K, K_n$ ($n\geqq 1$), where we put $m(L):=\mathrm{min}\{v_p(\mathrm{log}(\chi(\gamma))|\gamma\in \Gamma_L\}$.

\end{thm}

\begin{proof}
We first prove (1).
By Lemma \ref{3.9}, it suffices to show (1) for $k=0$. Moreover, 
since we have a commutative diagram

$$
\begin{CD}
\bold{D}^{K_m}_{\mathrm{dR}}(D)@> \mathrm{exp}_{K_m,D}>> \mathrm{H}^1(K_m, D) \\
  @VV\mathrm{Tr}_{K_m/L} V @VV \mathrm{cor}_{K_m/L} V \\
\bold{D}^L_{\mathrm{dR}}(D)@> \mathrm{exp}_{L,D}>> \mathrm{H}^1(L, D), 
  \end{CD}
  $$
 for each 
$L=K_n, K$ ($m\geqq n$) (where $\mathrm{cor}_{K_m/L}$ is the corestriction map) and since we have an 
equality 
$$[K_n:L]\frac{|\Gamma_{K_n,\mathrm{tor}}|}{p^{m(K_n)}}=\frac{|\Gamma_{L,\mathrm{tor}}|}{p^{m(L)}},$$
it suffices to show (1) when $L=K_n$ for sufficiently large $n$. Hence we may assume that 
$n\geqq n(D)$ and $\Gamma_{K_n,\mathrm{tor}}=\{1\}$. 
We set $N_n:=|\Gamma_K/(\Gamma_{K_n}\times \Delta_K)|$. 
Then we can write uniquely $\gamma_K^{N_n}=\gamma_n  g$ where $\gamma_n\in \Gamma_{K_n}$ is a topological generator 
of $\Gamma_{K_n}$ and $g\in \Delta_K$.
 Under this situation, we prove (1) using the isomorphisms 
$\bold{H}^1_{\mathrm{Iw}}(K, D)\isom \varprojlim_m\mathrm{H}^1(C^{\bullet}_{\psi,\gamma_K}(D\widehat{\otimes}_{\mathbb{Q}_p}\widetilde{\Lambda}_{m}^{\iota}))$ and 
$\mathrm{H}^1(K_n, D)\isom \mathrm{H}^1(C^{\bullet}_{\psi, \gamma_n}(D))$.
We define $$\frac{\nabla_0}{\gamma_n-1}:= \frac{1}{\mathrm{log}(\chi(\gamma_n))}\sum^{\infty}_{k=1}\frac{(-1)^{k-1}}{n}(\gamma_n-1)^{k-1}\in 
\Lambda_{\infty}.$$
Let $x\in (\bold{N}^{(n)}_{\mathrm{rig}}(D))^{\psi=1}$ be any element. By the same argument as in the proof of Theorem 2.3 of \cite{Ber03}, we have equalities
$$\frac{\nabla_0}{\gamma_n-1}(T_{K_n}(x))=\frac{1}{\mathrm{log}(\chi(\gamma_n))} T_{K_n}(x)\in \bold{D}^{K_n}_{\mathrm{dR}}(D)$$ and $$\iota_n(\frac{\nabla_0}{\gamma_n-1}(x))
=\frac{\nabla_0}{\gamma_n-1}(\iota_n(x))
= \frac{1}{\mathrm{log}(\chi(\gamma_n))} T_{K_n}(x) +tz \in K_n[[t]]\otimes_{K_n}\bold{D}^{K_n}_{\mathrm{dR}}(D)$$ for some $z\in K_n[[t]]\otimes_{K_n}\bold{D}^{K_n}_{\mathrm{dR}}(D)$.
 Hence, if we define $$\widetilde{x}:=\nabla_{h-1}\cdot\nabla_{h-2}\cdots\nabla_1\cdot \frac{\nabla_0}{\gamma_n-1}(x)\in (D^{(n)}[1/t])^{\psi=1},$$ then we obtain

\[
\begin{array}{ll}
\iota_n(\widetilde{x})&=
\nabla_{h-1}\cdot\nabla_{h-2}\cdots \nabla_1\cdot \frac{\nabla_0}{\gamma_n-1}(\iota_n(x))\\
&=\nabla_{h-1}\cdot\nabla_{h-2}\cdots\nabla_1(\frac{1}{\mathrm{log}(\chi(\gamma_n))} T_{K_n}(x) +tz)\\
&\equiv 
\frac{(-1)^{h-1}(h-1)!}{\mathrm{log}(\chi(\gamma_n))} T_{K_n}(x)\, (\mathrm{mod}\, t^hK_n[[t]]\otimes_{K_n}\bold{D}^{K_n}_{\mathrm{dR}}(D)).
\end{array}
\]

Next, we claim that we have 
$$\iota_{m}(\widetilde{x})
\equiv \iota_n(\widetilde{x})\,(\text{ mod }\,t^h K_m[[t]]\otimes_{K_m}\bold{D}^{K_m}_{\mathrm{dR}}(D))$$ 
for any $m\geqq n$. 
To prove this claim, it suffices to show that
$$\iota_{m+1}(\widetilde{x})-\iota_{m}(\widetilde{x})\in t^hK_{m+1}[[t]]\otimes_{K_{m+1}}\bold{D}^{K_{m+1}}_{\mathrm{dR}}(D)$$
for any $m\geqq n$. Moreover, since we have $\iota_{m+1}((\varphi-1)z)=\iota_m(z)-\iota_{m+1}(z)$ and $\bold{D}_{\mathrm{dif},m}(t^h\bold{N}_{\mathrm{rig}}(D))=t^hK_m[[t]]
\otimes_{K_m}\bold{D}^{K_m}_{\mathrm{dR}}(D)$, 
it suffices to show that
$$(\varphi-1)\widetilde{x}\in t^h\bold{N}^{(n+1)}_{\mathrm{rig}}(D).$$
 Since $x\in \bold{N}_{\mathrm{rig}}(D)^{\psi=1}$, we have 
$\varphi(x)-x\in \bold{N}_{\mathrm{rig}}(D)^{\psi=0}$. By Theorem \ref{2.3},  there exists unique $y_n\in \bold{N}_{\mathrm{rig}}(D)^{\psi=0}$ such that 
$$\varphi(x)-x=(\gamma_n-1)y_n.$$ 
Then we have
\[
\begin{array}{ll}
(\varphi-1)\widetilde{x}&=\nabla_{h-1}\cdots\nabla_1\cdot\frac{\nabla_0}{\gamma_n-1}(\varphi(x)-x))\\
                                                                              &=\nabla_{h-1}\cdots\nabla_1\cdot\frac{\nabla_0}{\gamma_n-1}((\gamma_n-1)y_n)\\
                                                                              &=\nabla_{h-1}\cdots\nabla_1\cdot\nabla_0(y_n)\in t^h\bold{N}^{(n+1)}_{\mathrm{rig}}(D),
\end{array}
\]     
where the last inclusion follows from Lemma \ref{3.6}.                                                                      


By this claim, by Lemma \ref{explicit} (1), by the definition 
of the canonical isomorphism $\mathrm{H}^1(K_n, D)\isom \mathrm{H}^1(C^{\bullet}_{\psi,\gamma_n}(D))$, and 
by the fact that 
$t^hK_m[[t]]\otimes_{K_m}\bold{D}^{K_m}_{\mathrm{dR}}(D)\subseteq \bold{D}^+_{\mathrm{dif},m}(D)$, we obtain 
 
 \[
\begin{array}{ll}
\frac{(-1)^{h-1}(h-1)!}{\mathrm{log}(\chi(\gamma_n))} \mathrm{exp}_{K_n,D}(T_{K_n}(x))&= [(\gamma_n-1)\widetilde{x}, (1-\psi)\widetilde{x}]\\
&=[\nabla_{h-1}\cdot\nabla_{h-2} \cdots\nabla_1\cdot \nabla_0(x),0]\in \mathrm{H}^1(C^{\bullet}_{\psi,\gamma_n}(D)). 
\end{array}
\]
Since the natural projection map 
$\mathrm{pr}_{K_n,D}:D^{\psi=1}\rightarrow \mathrm{H}^1(K_n, D)\isom \mathrm{H}^1(C^{\bullet}_{\psi,\gamma_n}(D))$ is given by 
$\mathrm{pr}_{K_n,D}(y)=[\mathrm{log}_0(\chi(\gamma_n))y, 0]$, we obtain 
\[
\begin{array}{ll}
\mathrm{pr}_{K_n,D}(\mathrm{Exp}_{D,h}(x))&=[\mathrm{log}_0(\chi(\gamma_n))\nabla_{h-1}\cdots \nabla_0(x), 0]\\
                                                                        &=(-1)^{h-1}(h-1)!\frac{\mathrm{log}_0(\chi(\gamma_n))}{\mathrm{log}(\chi(\gamma_n))}\mathrm{exp}_{K_n,D}(T_{K_n}(x))\\
                                                                        &=\frac{(-1)^{h-1}(h-1)!}{p^{m(K_n)}}\mathrm{exp}_{K_n,D}(T_{K_n}(x)), 
                                                                      
\end{array}
\]            
which proves (1).

Next we prove (2). Because we have 
$$\mathrm{Tr}_{K_{n+1}/K_n}( \mathrm{exp}^{*}_{K_{n+1}, D^{\lor}(1)}(x))=\mathrm{exp}^{*}_{K_n,D^{\lor}(1)}(\mathrm{cor}_{K_{n+1}/K_n}(x))$$ 
for any $x\in \mathrm{H}^1(K_{n+1}, D)$, it suffices to show (2) for sufficiently large $n$ as in the 
proof of (1). Moreover, by Lemma \ref{3.9}, it suffices to show (2) for 
$\mathrm{Exp}_{D, 1}$ when $D$ satisfies $\mathrm{Fil}^{-1}\bold{D}^K_{\mathrm{dR}}(D)=\bold{D}^K_{\mathrm{dR}}(D)$. 
For $x\in \bold{N}_{\mathrm{rig}}(D)^{\psi=1}$, we write 
$$\iota_n(x):=\sum_{m=0}t^mx_m \in \bold{D}^+_{\mathrm{dif}, n}(\bold{N}_{\mathrm{rig}}(D))=K_n[[t]]\otimes_{K_n}\bold{D}^{K_n}_{\mathrm{dR}}(D)\,
\,\,\,\,( x_m\in \bold{D}^{K_n}_{\mathrm{dR}}(D)).$$
Because we have the following commutative diagram
$$
\begin{CD}
\bold{N}_{\mathrm{rig}}(D)@>> \iota_n > K_n[[t]]\otimes_{K_n}\bold{D}^{K_n}_{\mathrm{dR}}(D) \\
@VV \widetilde{\partial}^{-k} V @ VV f(t)\otimes x\mapsto (\frac{d}{dt})^{-k}(f(t))\otimes t^{-k}xe_{k} V\\
\bold{N}_{\mathrm{rig}}(D(k))@>> \iota_n > K_n[[t]]\otimes_{K_n}\bold{D}^{K_n}_{\mathrm{dR}}(D(k)),
\end{CD}
$$
for each $k\leqq -1$, we obtain 
\[
\begin{array}{ll}
T_{K_n}(\widetilde{\partial}^{-k}(x))&=\iota_n(\widetilde{\partial}^{-k}(x))|_{t=0}\\
&=(\frac{d}{dt})^{-k}(\sum_{m=0}^{\infty}t^mx_m)|_{t=0}\otimes t^{-k}e_{k}\\
&=(-k)!\cdot x_{-k}\otimes t^{-k}e_{k}\in \bold{D}^{K_n}_{\mathrm{dR}}(D(k))=\bold{D}^{K_n}_{\mathrm{dR}}(D)\otimes t^{-k}e_{k}.
\end{array}
\]
 On the other hand, we have
an equality
 \[
\begin{array}{ll}
\mathrm{pr}_{K_n,D(k)}(\mathrm{Exp}_{D,1}(x))&=\mathrm{pr}_{K_n,D(k)}(\nabla_0(x))\\
&=[\mathrm{log}_0(\chi(\gamma_n))\nabla_0(x)\otimes e_{k}, 0]
\in \mathrm{H}^1(K_n, D(k))=\mathrm{H}^1(C^{\bullet}_{\psi,\gamma_n}(D(k)))
\end{array}
\]
and the natural map 
$$\mathrm{H}^1(C^{\bullet}_{\psi,\gamma_n}(D(k)))\rightarrow \mathrm{H}^1(C^{\bullet}_{\gamma_n}(\bold{D}_{\mathrm{dif}}(D(k))))
=\mathrm{H}^1(C^{\bullet}_{\gamma_n}(K_{\infty}((t))\otimes_{K_n}\bold{D}^{K_n}_{\mathrm{dR}}(D(k))))$$ sends the element 
$[\mathrm{log}_0(\chi(\gamma_n))\nabla_0(x)\otimes e_k, 0]$ to 
 \[
\begin{array}{ll}
[\mathrm{log}_0(\chi(\gamma_n))\nabla_0(\iota_n(x))\otimes e_k]&=[\mathrm{log}_0(\chi(\gamma_n))\nabla_0(\sum_{m=0}^{\infty}t^mx_m)\otimes e_k]\\
&=[\mathrm{log}_0(\chi(\gamma_n))(\sum_{m=1}^{\infty}mt^{m}x_m)\otimes e_k].
\end{array}
\]
Moreover, since we have 
$$[\mathrm{log}_0(\chi(\gamma_n))(\sum_{m=1,m\not=-k}^{\infty}mt^{m}x_m)\otimes e_k]=0\in \mathrm{H}^1(C^{\bullet}_{\gamma_n}(\bold{D}_{\mathrm{dif}}(D(k)))),$$ 
we obtain an equality 
 $$[\mathrm{log}_0(\chi(\gamma_n))(\sum_{m=1}^{\infty}mt^{m}x_m)\otimes e_k]=[\mathrm{log}_0(\chi(\gamma_n))(-k)x_{-k}t^{-k}\otimes e_k]\in \mathrm{H}^1(C^{\bullet}_{\gamma_n}(\bold{D}_{\mathrm{dif}}(D(k)))).$$ 
By these calculations and by the definition of $\mathrm{exp}^{*}_{K_n,D^{\lor}(1-k)}$, we obtain
\[
\begin{array}{ll}
\mathrm{exp}^{*}_{K_n,D^{\lor}(1-k)}(\mathrm{pr}_{K_n,D(k)}(\mathrm{Exp}_{D,1}(x)))&=\mathrm{exp}^{*}_{K_n,D^{\lor}(1-k)}([\mathrm{log}_0(\chi(\gamma_n))\nabla_0(x)\otimes e_k, 0])\\
&=(-k)\frac{\mathrm{log}_0(\chi(\gamma_n))}{\mathrm{log}(\chi(\gamma_n))}x_{-k}\otimes t^{-k}e_k\in 
\bold{D}^{K_n}_{\mathrm{dR}}(D(k))\\
&=\frac{1}{(-1-k)!\cdot p^{m(K_n)}}T_{K_n}(\widetilde{\partial}^{-k}(x))
\end{array}
\]
which proves (2), hence finishes the proof of the theorem.

\end{proof}

\subsection{determinant of $\mathrm{Exp}_{D,h}$: a generalization of Perrin-Riou's $\delta(V)$}
In this subsection, we formulate and prove a theorem which we call $\delta(D)$ concerning the determinant of our big exponential maps,  which says that 
the determinant of our map $\mathrm{Exp}_{D,h}$ can be described by 
the second Iwasawa cohomologies $\bold{H}^2_{\mathrm{Iw}}(K, D)$ and 
$\bold{H}^2_{\mathrm{Iw}}(K, \bold{N}_{\mathrm{rig}}(D))$ and by the ``$\Gamma$-factor"
which is determined by the Hodge-Tate weights of $D$. 


To formulate the theorem $\delta(D)$, we need to recall
 the definition of the characteristic ideal $\mathrm{char}_{\Lambda_{\infty}}(M)\subseteq \Lambda_{\infty}$ for a co-admissible torsion $\Lambda_{\infty}$-module $M$.
 A co-admissible $\Lambda_{\infty}$-module is defined as a
 $\Lambda_{\infty}$-module which is isomorphic to 
 the global section of a coherent sheaf on the rigid analytic space $\cup_{n}\mathrm{Spm}(\Lambda_n)$. See $\S1$ of \cite{Po12b} for  more precise definitions. 
 Let $M$ be a torsion co-admissible $\Lambda_{\infty}$-module. For each $n\geqq 1$,  we put $M_n:=M\otimes_{\Lambda_{\infty}}\Lambda_n$, which 
  is a finite generated torsion 
 $\Lambda_{n}$-module, and $M\isom \varprojlim_nM_n$ by the theorem of Schneider-Teitelbaum. Since 
 $\Lambda_{n}$ is a finite product of P.I.D.s, $M_n$ is a finite length 
 $\Lambda_{n}$-module. Hence,  we can define a unique principal ideal 
 $(f_{M_n})$ of $\Lambda_n$ such that $\mathrm{length}_{(\Lambda_{n})_x}((M_n)_x)=v_x(f_{M_n})$ for each maximal ideal
 $x$ of $\Lambda_n$, where $v_x$ is the normalized valuation of the local ring 
 $(\Lambda_{n})_x$ of $\Lambda_n$ at $x$. By the theorem of Lazard, there exists a unique principal 
 ideal $(f_M)$ of $ \Lambda_{\infty}$ such that $f_M\Lambda_n=(f_{M_n})\subseteq \Lambda_n$ for each
 $n\geqq 1$.  Then, the characteristic ideal $\mathrm{char}_{\Lambda_{\infty}}(M)$ of $M$ is 
 defined by 
 $$\mathrm{char}_{\Lambda_{\infty}}(M):=(f_M)\subseteq \Lambda_{\infty}.$$
 
 Let $\mathrm{Frac}(\Lambda_{\infty})$  be the ring of the total fractions of $\Lambda_{\infty}$.
 Since we have $\Lambda_{\infty}\isom\bigoplus_{\eta\in \widehat{\Gamma}_{K,\mathrm{tor}}}\Lambda_{\infty}\alpha_{\eta}$ and have a non-canonical isomorphism
  $\Lambda_{\infty}\alpha_{\eta}\isom \bold{B}^+_{\mathrm{rig},\mathbb{Q}_p}$ for each $\eta\in \widehat{\Gamma}_{K,\mathrm{tor}}$, we have 
 $\mathrm{Frac}(\Lambda_{\infty})=\bigoplus_{\eta\in \widehat{\Gamma}_{K,\mathrm{tor}}}\mathrm{Frac}(\Lambda_{\infty}\alpha_{\eta})$, where 
 $\mathrm{Frac}(\Lambda_{\infty}\alpha_{\eta})$ is the fraction field of $\Lambda_{\infty}\alpha_{\eta}$. 
 For any principal ideals $(f_1), (f_2)\subseteq \Lambda_{\infty}$ such that 
 $f_i\alpha_{\eta}\not=0$ for any $i=1,2$ and $\eta\in\widehat{\Gamma}_{K,\mathrm{tor}}$, we denote by
 $(f_1)(f_2)^{-1}\subseteq \mathrm{Frac}(\Lambda_{\infty})$ the principal fractional
  ideal of $\mathrm{Frac}(\Lambda_{\infty})$
 generated by $\frac{f_1}{f_2}\in \mathrm{Frac}(\Lambda_{\infty})$.

 Let $M_1$ and $M_2$ be  co-admissible $\Lambda_{\infty}$-modules, 
 $f:M_1\rightarrow M_2$ a $\Lambda_{\infty}$-linear morphism. 
 We assume that $\mathrm{Coker}(f)$ is a torsion $\Lambda_{\infty}$-module and that 
the natural induced map $\alpha_{\eta}\overline{f}:\alpha_{\eta}(M_1/M_{1,\mathrm{tor}})\rightarrow \alpha_{\eta}(M_2/M_{2,\mathrm{tor}})$ is a non-zero injection for each $\eta\in \widehat{\Gamma}_{K,\mathrm{tor}}$. Because we have $(M/M_{\mathrm{tor}})\otimes_{\Lambda_{\infty}}\Lambda_n\isom M_n/M_{n,\mathrm{tor}}$ for any co-admissible $\Lambda_{\infty}$-module $M$ 
and because the latter is a finite projective $\Lambda_n$-module, we can define $\mathrm{det}_{\Lambda_n}(\overline{f}_n):=\mathrm{det}_{\Lambda_n}(\overline{f}_n:M_{1,n}/M_{1,n,\mathrm{tor}}\rightarrow M_{2,n}/M_{2,n,\mathrm{tor}})
\in \Lambda_n$ and 
$\mathrm{det}_{\Lambda_{\infty}}(\overline{f}):=\varprojlim_n \mathrm{det}_{\Lambda_n}(\overline{f}_n)\in \Lambda_{\infty}$, which satisfies that 
$\alpha_{\eta}\mathrm{det}_{\Lambda_{\infty}}(\overline{f})\not=0$ 
for any $\eta\in \widehat{\Gamma}_{K,\mathrm{tor}}$. 
We define a principal  fractional ideal $\mathrm{det}_{\Lambda_{\infty}}(f)\subseteq \mathrm{Frac}(\Lambda_{\infty})$ by 
$$\mathrm{det}_{\Lambda_{\infty}}(f):=(\mathrm{det}_{\Lambda_{\infty}}(\overline{f}))\mathrm{char}_{\Lambda_{\infty}}(M_{2,\mathrm{tor}})
(\mathrm{char}_{\Lambda_{\infty}}(M_{1,\mathrm{tor}}))^{-1}\subseteq \mathrm{Frac}(\Lambda_{\infty}).$$

\begin{lemma}\label{3.11}
$\mathrm{det}_{\Lambda_{\infty}}(-)$ satisfies the following formulae;
\begin{itemize}
\item[(i)]$\mathrm{det}_{\Lambda_{\infty}}(f)=\mathrm{char}_{\Lambda_{\infty}}(\mathrm{Coker}(f))(\mathrm{char}_{\Lambda_{\infty}}(\mathrm{Ker}(f)))^{-1}$.
\item[(ii)]For any $f_1:M_1\rightarrow M_2$ and $f_2:M_2\rightarrow M_3$ as above, we have an equality
$$\mathrm{det}_{\Lambda_{\infty}}(f_2\circ f_1)=\mathrm{det}_{\Lambda_{\infty}}(f_1)\mathrm{det}_{\Lambda_{\infty}}(f_2).$$
\item[(iii)]If we have a commutative diagram
$$
\begin{CD}
0@>>> M_1 @>>> M_1'@ >>> M_1''@ >>> 0 \\
  @. @VV f V    @VV f' V @VV f'' V \\
 0 @>>> M_2@>>> M_2' @>>> M_2'' @>>> 0
  \end{CD}
  $$
  with exact rows, then we have an equality
  $$\mathrm{det}_{\Lambda_{\infty}}(f')=\mathrm{det}_{\Lambda_{\infty}}(f)\mathrm{det}_{\Lambda_{\infty}}(f'').$$

\end{itemize}
\end{lemma}
\begin{proof}
One can prove this by an easy linear algebra argument, so we omit the proof.

\end{proof}

Let $M_1, M_2$ be $\Lambda_{\infty}$-modules, 
$d_i:M_i\rightarrow M_i$ a $\Lambda_{\infty}$-linear endomorphism. 
Denote by $C^{\bullet}_{d_i}(M_i):=[M_i\xrightarrow{d_i} M_i]$  the complex of $\Lambda_{\infty}$-modules concentrated in degree $[1,2]$. We assume that $\mathrm{H}^j(C^{\bullet}_{d_i}(M_i))$
 are co-admissible 
$\Lambda_{\infty}$-modules for any $i,j\in \{1,2\}$.  Let $f:M_1\rightarrow M_2$ be a $\Lambda_{\infty}$-
linear morphism which satisfies that $f\circ d_1=d_2\circ f$. We assume that the induced maps 
$\mathrm{H}^i(f): \mathrm{H}^i(C^{\bullet}_{d_1}(M_1))\rightarrow \mathrm{H}^i(C^{\bullet}_{d_2}(M_2))$ for $i=1,2$ 
satisfy the conditions in the last paragraph. Then we define a principal fractional ideal 
$$\mathrm{det}_{\Lambda_{\infty}}(\mathrm{H}^{\bullet}(f)):=\mathrm{det}_{\Lambda_{\infty}}(\mathrm{H}^1(f))\mathrm{det}_{\Lambda_{\infty}}(\mathrm{H}^2(f))^{-1}.$$ 

\begin{lemma}\label{3.12}
$\mathrm{det}_{\Lambda_{\infty}}(\mathrm{H}^{\bullet}(f))$ satisfies the following;

\begin{itemize}

\item[(iv)]Let $(M_i,d_i)$ ($i=1,2,3$), $f_1:M_1\rightarrow M_2$ and 
$f_2:M_2\rightarrow M_3$ be as above.  Then we have
$$\mathrm{det}_{\Lambda_{\infty}}(\mathrm{H}^{\bullet}(f_2\circ f_1))=\mathrm{det}_{\Lambda_{\infty}}(\mathrm{H}^{\bullet}(f_1))\mathrm{det}_{\Lambda_{\infty}}(\mathrm{H}^{\bullet}(f_2)).$$
\item[(v)]If $\mathrm{Ker}(f)$ and $\mathrm{Coker}(f)$ are both torsion co-admissible $\Lambda_{\infty}$-modules, then we have an 
equality
$$\mathrm{det}_{\Lambda_{\infty}}(\mathrm{H}^{\bullet}(f))=\Lambda_{\infty}.$$
\end{itemize}
\end{lemma}
\begin{proof}
This is also proved by an easy linear algebra argument, so we omit the proof.

\end{proof}

We apply these definitions to the following situation. 
Let $D$ be a de Rham $(\varphi,\Gamma_K)$-module over 
$\bold{B}^{\dagger}_{\mathrm{rig},K}$. Take $h\geqq 1$ such that 
$\mathrm{Fil}^{-h}\bold{D}^K_{\mathrm{dR}}(D)=\bold{D}_{\mathrm{dR}}^K(D)$. 
We want to apply the above definitions to the  maps $\psi-1: D\rightarrow D, \psi-1:\bold{N}_{\mathrm{rig}}(D)\rightarrow \bold{N}_{\mathrm{rig}}(D)$  and 
the map $\nabla_{h-1}\cdots\nabla_0:\bold{N}_{\mathrm{rig}}(D)\rightarrow D$ defined in Lemma \ref{3.6}. 
By Theorem \ref{3.2}, in order to apply the above definition to this setting, we need to show the following lemma.
\begin{lemma}\label{3.13}
The map 
$$\overline{\nabla_{h-1}\cdots\nabla_0}: \bold{N}_{\mathrm{rig}}(D)^{\psi=1}/\bold{N}_{\mathrm{rig}}(D)^{\psi=1}_{\mathrm{tor}}
\rightarrow D^{\psi=1}/D^{\psi=1}_{\mathrm{tor}}$$ which is induced by $\nabla_{h-1}\cdots\nabla_0:\bold{N}_{\mathrm{rig}}(D)^{\psi=1}\rightarrow D^{\psi=1}$ is injective.

\end{lemma}
\begin{proof}
We first note that the map 
$\nabla_{h-1}\cdots\nabla_{0}:\bold{N}_{\mathrm{rig}}(D)^{\psi=1}\rightarrow D^{\psi=1}$ is the composition of 
$\nabla_{h-1}\cdots\nabla_0:\bold{N}_{\mathrm{rig}}(D)^{\psi=1}\rightarrow (t^h\bold{N}_{\mathrm{rig}}(D))^{\psi=1}$ with the natural injection 
$(t^h\bold{N}_{\mathrm{rig}}(D))^{\psi=1}\hookrightarrow D^{\psi=1}$. Since we have 
$$(t^h\bold{N}_{\mathrm{rig}}(D))^{\psi=1}_{\mathrm{tor}}=D^{\psi=1}_{\mathrm{tor}}\cap (t^h\bold{N}_{\mathrm{rig}}(D))^{\psi=1},$$ the map $(t^h\bold{N}_{\mathrm{rig}}(D))^{\psi=1}/(t^h\bold{N}_{\mathrm{rig}}(D))^{\psi=1}_{\mathrm{tor}}
\rightarrow D^{\psi=1}/D^{\psi=1}_{\mathrm{tor}}$ is injective. To show that the map
$$\bold{N}_{\mathrm{rig}}(D)^{\psi=1}/\bold{N}_{\mathrm{rig}}(D)^{\psi=1}_{\mathrm{tor}}\xrightarrow{\overline{\nabla_{h-1}\cdots\nabla_0}}(t^h\bold{N}_{\mathrm{rig}}(D))^{\psi=1}/(t^h\bold{N}_{\mathrm{rig}}(D))^{\psi=1}_{\mathrm{tor}}$$ is injective, it suffices to show that the map
$$\bold{N}_{\mathrm{rig}}(D)^{\psi=1}/\bold{N}_{\mathrm{rig}}(D)^{\psi=1}_{\mathrm{tor}}
\xrightarrow{\overline{\nabla_{h-1}\cdots\nabla_0}} \bold{N}_{\mathrm{rig}}(D)^{\psi=1}/\bold{N}_{\mathrm{rig}}(D)^{\psi=1}_{\mathrm{tor}}$$ is injective. 
Since $\bold{N}_{\mathrm{rig}}(D)^{\psi=1}/\bold{N}_{\mathrm{rig}}(D)^{\psi=1}_{\mathrm{tor}}$ is a finite free $\Lambda_{\infty}$-module by Theorem \ref{3.2}
and $\nabla_{h-1}\cdots \nabla_0\in \Lambda_{\infty}$ is a non zero divisor, the map 
$$\bold{N}_{\mathrm{rig}}(D)^{\psi=1}/\bold{N}_{\mathrm{rig}}(D)^{\psi=1}_{\mathrm{tor}}
\xrightarrow{\overline{\nabla_{h-1}\cdots\nabla_0}} \bold{N}_{\mathrm{rig}}(D)^{\psi=1}/\bold{N}_{\mathrm{rig}}(D)^{\psi=1}_{\mathrm{tor}}$$ is injective, which proves the lemma.

\end{proof}

By this lemma and by  Theorem \ref{3.2}, we can define a 
fractional ideal 
$$\mathrm{det}_{\Lambda_{\infty}}(\mathrm{H}^{\bullet}(\bold{N}_{\mathrm{rig}}(D)\xrightarrow{\nabla_{h-1}\cdots \nabla_0}D))\subseteq \mathrm{Frac}(\Lambda_{\infty}).$$
 By the definition of $\mathrm{det}_{\Lambda_{\infty}}(-)$ and since
$\bold{H}^2_{\mathrm{Iw}}(K, -)$ are co-admissible torsion $\Lambda_{\infty}$-modules by Theorem \ref{3.2}, we have an 
equality 
\begin{multline*}
\mathrm{det}_{\Lambda_{\infty}}(\mathrm{H}^{\bullet}(\bold{N}_{\mathrm{rig}}(D)\xrightarrow{\nabla_{h-1}\cdots \nabla_0}D))=\mathrm{det}_{\Lambda_{\infty}}(\bold{N}_{\mathrm{rig}}(D)^{\psi=1}\xrightarrow{\mathrm{Exp}_{D,h}}\bold{H}^1_{\mathrm{Iw}}(K, D))\cdot\\
 \mathrm{char}_{\Lambda_{\infty}}(\bold{H}^2_{\mathrm{Iw}}(K, \bold{N}_{\mathrm{rig}}(D)))(\mathrm{char}_{\Lambda_{\infty}}
(\bold{H}^2_{\mathrm{Iw}}(K, D)))^{-1}.
\end{multline*}

Concerning  this determinant, we have a following theorem. As we will explain in the next subsection, this theorem can be seen 
as a generalization of the theorem $\delta(V)$ of Perrin-Riou (Conjecture 3.4.7 of \cite{Per94}) and of
 the theorem $\delta(D)$ of Pottharst (Theorem 3.4 of \cite{Po12b}).

\begin{thm}\label{3.14}$(\delta(D))$
Let $D$ be a de Rham $(\varphi,\Gamma_K)$-module over $\bold{B}^{\dagger}_{\mathrm{rig},K}$ of rank  $d$
with Hodge-Tate weights $\{h_1, h_2,\cdots,h_d\}$ $($note that the Hodge-Tate weight of $\mathbb{Q}_p(1)$ is $1$$)$.
For any $h\geqq 1$ such that $\mathrm{Fil}^{-h}\bold{D}^K_{\mathrm{dR}}(D)=\bold{D}^K_{\mathrm{dR}}(D)$,
we have the following equality of principal fractional ideals of $\Lambda_{\infty}$
\begin{multline*}
\frac{1}{(\prod_{i=1}^d \prod_{j_i=0}^{h-h_i-1}
\nabla_{h_i+j_i})^{[K:\mathbb{Q}_p]}}\mathrm{det}_{\Lambda_{\infty}}(\bold{N}_{\mathrm{rig}}(D)^{\psi=1}\xrightarrow{\mathrm{Exp}_{D,h}} \bold{H}^1_{\mathrm{Iw}}(K, D)) \\
=
\mathrm{char}_{\Lambda_{\infty}}(\bold{H}^2_{\mathrm{Iw}}(K, D))(\mathrm{char}_{\Lambda_{\infty}}\bold{H}^2_{\mathrm{Iw}}(K, \bold{N}_{\mathrm{rig}}(D)))^{-1}.
\end{multline*}
In particular, the ideal of the left hand side does not depend on $h$, where 
we define $\prod_{j_i=0}^{h-h_i-1}
\nabla_{h_i+j_i}:=1$ when $h=h_{i}$.

\end{thm}

\begin{proof}
By the definition of $\mathrm{det}_{\Lambda_{\infty}}(\mathrm{H}^{\bullet}(-))$ and because we have an isomorphism 
$\mathrm{H}^{i}([D\xrightarrow{\psi-1}D])\isom 
\bold{H}^i_{\mathrm{Iw}}(K, D)$ by Theorem \ref{3.3}, it suffices to show that 
$$\mathrm{det}_{\Lambda_{\infty}}(\mathrm{H}^{\bullet}(\bold{N}_{\mathrm{rig}}(D)\xrightarrow{\nabla_{h-1}\cdots\nabla_0}D))
=(\prod_{i=1}^d \prod_{j_i=0}^{h-h_i-1}
\nabla_{h_i+j_i})^{[K:\mathbb{Q}_p]}.$$

Moreover, since we have an equality 
 \begin{multline*}
  \mathrm{det}_{\Lambda_{\infty}}(\mathrm{H}^{\bullet}(\bold{N}_{\mathrm{rig}}(D)\xrightarrow{\nabla_{h-1}\cdots\nabla_0}D)) \\
=\prod_{i=0}^{h-1}\mathrm{det}_{\Lambda_{\infty}}(\mathrm{H}^{\bullet}(t^i \bold{N}_{\mathrm{rig}}(D)\xrightarrow{\nabla_i}
t^{i+1}\bold{N}_{\mathrm{rig}}(D))) \cdot \mathrm{det}_{\Lambda_{\infty}}(\mathrm{H}^{\bullet}(t^h\bold{N}_{\mathrm{rig}}(D)
\xrightarrow{\iota} D))
\end{multline*}
by Lemma \ref{3.12} (where $\iota:t^h\bold{N}_{\mathrm{rig}}(D)\hookrightarrow D$ is the canonical inclusion), 
 it suffices to show the following equalities
\begin{itemize}
\item[(1)]$\mathrm{det}_{\Lambda_{\infty}}(\mathrm{H}^{\bullet}(t^{i}\bold{N}_{\mathrm{rig}}(D)\xrightarrow{\nabla_i}
t^{i+1}\bold{N}_{\mathrm{rig}}(D)))=\Lambda_{\infty}$ for each $0\leqq i\leqq h-1$,
\item[(2)]$\mathrm{det}_{\Lambda_{\infty}}(\mathrm{H}^{\bullet}(t^h\bold{N}_{\mathrm{rig}}(D)\xrightarrow{\iota} D))=(\prod_{i=1}^d \prod_{j_i=0}^{h-h_i-1}
\nabla_{h_i+j_i})^{[K:\mathbb{Q}_p]}$.

\end{itemize}

 The claim (1) follows from the property (v) of Lemma \ref{3.12} because $\mathrm{Ker}(\nabla_i:t^i\bold{N}_{\mathrm{rig}}(D)\rightarrow 
 t^{i+1}\bold{N}_{\mathrm{rig}}(D))$ and $\mathrm{Coker}(\nabla_i:t^i\bold{N}_{\mathrm{rig}}(D)\rightarrow 
 t^{i+1}\bold{N}_{\mathrm{rig}}(D))$ are finite dimensional $K_0$-vector spaces by the result of Crew ($\S6$ of \cite{Cr98}) (precisely, his result was under the assumption of the Crew's conjecture, which is now a theorem proved by Andr\'e, Christol-Mebkhout and Kedlaya).
 
 We prove the claim (2) as follows. We first consider the following diagram of short exact sequences
 $$
\begin{CD}
0@>>> t^h\bold{N}_{\mathrm{rig}}(D) @>>> D@ >>> D/t^h\bold{N}_{\mathrm{rig}}(D)@ >>> 0 \\
  @. @VV \psi-1V    @VV \psi-1V @VV \psi-1V \\
 0 @>>> t^h\bold{N}_{\mathrm{rig}}(D)@>>> D@>>> D/t^h\bold{N}_{\mathrm{rig}}(D)@>>> 0. 
  \end{CD}
  $$
  
  Using snake lemma 
  , we obtain the following long exact sequence
  \begin{align*}
  0\rightarrow &\bold{H}^1_{\mathrm{Iw}}(K, t^h\bold{N}_{\mathrm{rig}}(D))
  \rightarrow  \bold{H}^1_{\mathrm{Iw}}(K,D)\rightarrow (D/t^h\bold{N}_{\mathrm{rig}}(D))^{\psi=1} \\
  \rightarrow & \bold{H}^2_{\mathrm{Iw}}(K, t^h\bold{N}_{\mathrm{rig}}(D))
 \rightarrow \bold{H}^2_{\mathrm{Iw}}(K, D)
  \rightarrow (D/t^n\bold{N}_{\mathrm{rig}}(D))/(\psi-1) \rightarrow 0.
  \end{align*}
  
  Since  $(D/t^h\bold{N}_{\mathrm{rig}}(D))^{\psi=1}$ is a torsion co-admissible $\Lambda_{\infty}$-module 
  and $$(D/t^n\bold{N}_{\mathrm{rig}}(D))/(\psi-1)=0$$ by Proposition 2.1 of \cite{Po12b}, 
  we obtain an equality 
  
  $$\mathrm{det}_{\Lambda_{\infty}}(\mathrm{H}^{\bullet}(t^h\bold{N}_{\mathrm{rig}}(D)\xrightarrow{\iota} D))
                 =\mathrm{char}_{\Lambda_{\infty}}((D/t^h\bold{N}_{\mathrm{rig}}(D))^{\psi=1}).$$

 Hence, it suffices to show 
 $$\mathrm{char}_{\Lambda_{\infty}}((D/t^h\bold{N}_{\mathrm{rig}}(D))^{\psi=1})=(\prod_{i=1}^d \prod_{j_i=0}^{h-h_i-1}
\nabla_{h_i+j_i})^{[K:\mathbb{Q}_p]}.$$
 Since we have $\bold{D}^+_{\mathrm{dif}, n}(t^h\bold{N}_{\mathrm{rig}}(D))
 =t^hK_n[[t]]\otimes_{K_n}\bold{D}^{K_n}_{\mathrm{dR}}(D)$ for any sufficiently large $n\gg0$, we 
 have a $\Lambda_{\infty}$-linear isomorphism for each $n\gg0$
 $$D^{(n)}/t^h\bold{N}^{(n)}_{\mathrm{rig}}(D)\isom \prod_{m\geqq n} \bold{D}^+_{\mathrm{dif}, m}(D)/ (t^hK_m[[t]]\otimes_{K_n}\bold{D}^{K_n}_{\mathrm{dR}}(D)): 
 \overline{x}\mapsto (\overline{\iota_m(x)})_{m\geqq n},$$
 where the injection follows from the definition of $\bold{N}_{\mathrm{rig}}^{(n)}(D)$ and the surjection follows by the same 
 proof as Lemma \ref{2.8}.
 If we write $\bold{D}^K_{\mathrm{dR}}(D)=\oplus_{i=1}^dK\beta_i$ such that 
 $\beta_i\in \mathrm{Fil}^{-h_i}\bold{D}^K_{\mathrm{dR}}(D)\setminus \mathrm{Fil}^{-h_i+1}\bold{D}^K_{\mathrm{dR}}(D)$, then 
 we can write 
  \[
  \begin{array}{ll}
  \bold{D}^+_{\mathrm{dif},n}(D)&=\mathrm{Fil}^0(K_n((t))\otimes_{K_n}\bold{D}^{K_n}_{\mathrm{dR}}(D))\\
                                             &=\bigoplus_{i=1}^d K_n[[t]](t^{h_i}\beta_i).
                                             
 \end{array}
 \]
 Since $\Gamma_K$ acts trivially on each $\beta_i$, we obtain a $\Lambda_{\infty}$-linear isomorphism
 $$g_n:D^{(n)}/t^h\bold{N}^{(n)}_{\mathrm{rig}}(D)\isom \bigoplus_{i=1}^d\prod_{m\geqq n}t^{h_i}K_m[[t]]/t^hK_m[[t]].$$
 Since we have the following commutative diagrams
 $$
 \begin{CD}
    D^{(n)}/t^h\bold{N}^{(n)}_{\mathrm{rig}}(D)@>> g_n >   \prod_{m\geqq n}\bigoplus_{i=1}^d t^{h_i}K_m[[t]]/t^hK_m[[t]]  \\
   @VV \overline{x}\mapsto \overline{x} V    @VV (x_m)_{m\geqq n}\mapsto (x_{m})_{m\geqq n+1}V  \\
   D^{(n+1)}/t^h\bold{N}^{(n+1)}_{\mathrm{rig}}(D)@>> g_{n+1} >   \prod_{m\geqq n+1}\bigoplus_{i=1}^d t^{h_i}K_{m}[[t]]/t^hK_{m}[[t]] 
  \end{CD}
 $$
and 
$$
 \begin{CD}
    D^{(n+1)}/t^h\bold{N}^{(n+1)}_{\mathrm{rig}}(D)@>> g_{n+1} >   \prod_{m\geqq n+1}\bigoplus_{i=1}^d t^{h_i}K_m[[t]]/t^hK_m[[t]]  \\
   @VV \psi V    @VV (x_m)_{m\geqq n+1}\mapsto (\frac{1}{p}\mathrm{Tr}_{K_{m+1}/K_m}(x_{m+1}))_{m\geqq n} V  \\
   D^{(n)}/t^h\bold{N}^{(n)}_{\mathrm{rig}}(D)@>> g_{n} >   \prod_{m\geqq n}\bigoplus_{i=1}^d t^{h_i}K_{m}[[t]]/t^hK_{m}[[t]],
  \end{CD}
 $$
 we obtain the following $\Lambda_{\infty}$-isomorphism
  \[
  \begin{array}{ll}
   (D/t^h\bold{N}_{\mathrm{rig}}(D))^{\psi=1}&=\varinjlim_{n\gg0}(D^{(n)}/t^h\bold{N}^{(n)}_{\mathrm{rig}}(D))^{\psi=1}\\
                                                                  &\isom \bigoplus_{i=1}^d \varinjlim_{n\gg0}(\varprojlim_{\frac{1}{p}\mathrm{Tr}_{K_{m+1}/K_m},m\geqq n}
                                                                     t^{h_i}K_m[[t]]/t^hK_m[[t]])\\
                                                                    & \isom \bigoplus_{i=1}^d\varprojlim_{\frac{1}{p}\mathrm{Tr}_{K_{m+1}/K_m}, m\geqq 1}t^{h_i}K_m[[t]]/t^hK_m[[t]].
   \end{array}
   \]
   Since we similarly have the $\Lambda_{\infty}$-isomorphism
   $$
   (t^{h'}\bold{B}^{\dagger}_{\mathrm{rig},K}/t^{h}\bold{B}^{\dagger}_{\mathrm{rig},K})^{\psi=1}
   \isom \varprojlim_{\frac{1}{p}\mathrm{Tr}_{K_{m+1}/K_m},m\geqq 1}t^{h'}K_m[[t]]/t^hK_m[[t]]$$                                                
  for each $h'\leqq h$, it suffices to show 
  $$\mathrm{char}_{\Lambda_{\infty}}((t^{h'}\bold{B}^{\dagger}_{\mathrm{rig},K}/t^{h}\bold{B}^{\dagger}_{\mathrm{rig},K})^{\psi=1})
  =(\nabla_{h'}\nabla_{h'-1}\cdots\nabla_{h-1})^{[K:\mathbb{Q}_p]}.$$
  
  Since we have $(t^{h'}\bold{B}^{\dagger}_{\mathrm{rig},K}/t^h\bold{B}^{\dagger}_{\mathrm{rig},K})/(\psi-1)=0$ for any $h>h'$ by 
  Proposition 2.1of \cite{Po12b}, 
  we obtain the following short exact sequence
  $$0\rightarrow (t^{h'+1}\bold{B}^{\dagger}_{\mathrm{rig},K}/t^h\bold{B}^{\dagger}_{\mathrm{rig},K})^{\psi=1}
  \rightarrow (t^{h'}\bold{B}^{\dagger}_{\mathrm{rig},K}/t^h\bold{B}^{\dagger}_{\mathrm{rig},K})^{\psi=1}
  \rightarrow  (t^{h'}\bold{B}^{\dagger}_{\mathrm{rig},K}/t^{h'+1}\bold{B}^{\dagger}_{\mathrm{rig},K})^{\psi=1}\rightarrow 0$$
  for each $h>h'$, hence we obtain
  $$\mathrm{char}_{\Lambda_{\infty}}((t^{h'}\bold{B}^{\dagger}_{\mathrm{rig},K}/t^h\bold{B}^{\dagger}_{\mathrm{rig},K})^{\psi=1})
  =\prod_{i=0}^{h-h'-1}\mathrm{char}_{\Lambda_{\infty}} (t^{h'+i}\bold{B}^{\dagger}_{\mathrm{rig},K}/t^{h'+i+1}\bold{B}^{\dagger}_{\mathrm{rig},K})^{\psi=1}.$$
  Hence, to prove the claim (2), it suffices to show the following lemma.
  \end{proof}
  
  \begin{lemma}\label{3.15}
  For each $h\in\mathbb{Z}$, we have
  $$\mathrm{char}_{\Lambda_{\infty}}((t^{h}\bold{B}^{\dagger}_{\mathrm{rig},K}/t^{h+1}\bold{B}^{\dagger}_{\mathrm{rig},K})^{\psi=1})=(\nabla_{h}^{[K:\mathbb{Q}_p]}).$$
  \end{lemma}
  
  \begin{proof}
  From the short exact sequence 
  $$0\rightarrow t^{h+1}\bold{B}^{\dagger}_{\mathrm{rig},K}\rightarrow t^h\bold{B}^{\dagger}_{\mathrm{rig},K}\rightarrow 
  t^h\bold{B}^{\dagger}_{\mathrm{rig},K}/t^{h+1}\bold{B}^{\dagger}_{\mathrm{rig},K}\rightarrow 0$$
  and from the fact that $(t^h\bold{B}^{\dagger}_{\mathrm{rig},K}/t^{h+1}\bold{B}^{\dagger}_{\mathrm{rig},K})/(\psi-1)=0$, 
  we obtain the following exact sequence
   \[
  \begin{array}{ll}
 0& \rightarrow \bold{H}^1_{\mathrm{Iw}}(K, t^{h+1}\bold{B}^{\dagger}_{\mathrm{rig},K})\rightarrow \bold{H}^1_{\mathrm{Iw}}(K, t^h\bold{B}^{\dagger}_{\mathrm{rig},K})
 \rightarrow (t^h\bold{B}^{\dagger}_{\mathrm{rig},K}/t^{h+1}\bold{B}^{\dagger}_{\mathrm{rig},K})^{\psi=1} \\
 &\rightarrow \bold{H}^2_{\mathrm{Iw}}(K, t^{h+1}\bold{B}^{\dagger}_{\mathrm{rig},K})\rightarrow \bold{H}^2_{\mathrm{Iw}}(K, t^{h}\bold{B}^{\dagger}_{\mathrm{rig},K})
 \rightarrow 0.
 \end{array}
 \]
 Hence, we obtain an equality
 $$\mathrm{char}_{\Lambda_{\infty}}((t^h\bold{B}^{\dagger}_{\mathrm{rig},K}/t^{h+1}\bold{B}^{\dagger}_{\mathrm{rig},K})^{\psi=1})
 =\mathrm{det}_{\Lambda_{\infty}}(\mathrm{H}^{\bullet}(t^{h+1}\bold{B}^{\dagger}_{\mathrm{rig},K}\rightarrow t^h\bold{B}^{\dagger}_{\mathrm{rig},K})). $$
 If we apply (iv) of Lemma \ref{3.12} to the composition of the maps 
 $$t^{h}\bold{B}^{\dagger}_{\mathrm{rig},K}\xrightarrow{\nabla_{h}} t^{h+1}\bold{B}^{\dagger}_{\mathrm{rig},K}
 \hookrightarrow t^h\bold{B}^{\dagger}_{\mathrm{rig},K},$$
 we obtain an equality 
 \begin{multline*}
  \mathrm{det}_{\Lambda_{\infty}}(\mathrm{H}^{\bullet}(t^{h+1}\bold{B}^{\dagger}_{\mathrm{rig},K}\hookrightarrow 
 t^h\bold{B}^{\dagger}_{\mathrm{rig},K}))\\
 =\mathrm{det}_{\Lambda_{\infty}}(\mathrm{H}^{\bullet}(t^h\bold{B}^{\dagger}_{\mathrm{rig},K}\xrightarrow{\nabla_h}
 t^h\bold{B}^{\dagger}_{\mathrm{rig},K}))(\mathrm{det}_{\Lambda_{\infty}}(\mathrm{H}^{\bullet}(t^h\bold{B}^{\dagger}_{\mathrm{rig},K}\xrightarrow{\nabla_h}
 t^{h+1}\bold{B}^{\dagger}_{\mathrm{rig},K})))^{-1}.
 \end{multline*}
 Since we have $\mathrm{det}_{\Lambda_{\infty}}(\mathrm{H}^{\bullet}(t^h\bold{B}^{\dagger}_{\mathrm{rig},K}\xrightarrow{\nabla_h}
 t^{h+1}\bold{B}^{\dagger}_{\mathrm{rig},K}))=\Lambda_{\infty}$ by the claim (1), 
 we obtain
 $$\mathrm{det}_{\Lambda_{\infty}}(\mathrm{H}^{\bullet}(t^{h+1}\bold{B}^{\dagger}_{\mathrm{rig},K}\hookrightarrow 
 t^h\bold{B}^{\dagger}_{\mathrm{rig},K}))=\mathrm{det}_{\Lambda_{\infty}}(\mathrm{H}^{\bullet}(t^h\bold{B}^{\dagger}_{\mathrm{rig},K}\xrightarrow{\nabla_h}
 t^h\bold{B}^{\dagger}_{\mathrm{rig},K})).$$
 Finally, because $t^h\bold{B}^{\dagger}_{\mathrm{rig},K}/(\psi-1)(t^h\bold{B}^{\dagger}_{\mathrm{rig},K})$ is 
 a co-admissible  torsion $\Lambda_{\infty}$-module and  the $\Lambda_{\infty}$-free rank of $(t^h\bold{B}^{\dagger}_{\mathrm{rig},K})^{\psi=1}$ is $[K:\mathbb{Q}_p]$ by 
 Theorem 3.2, we obtain
 $$\mathrm{det}_{\Lambda_{\infty}}(\mathrm{H}^{\bullet}(t^h\bold{B}^{\dagger}_{\mathrm{rig},K}\xrightarrow{\nabla_h}
 t^h\bold{B}^{\dagger}_{\mathrm{rig},K}))=(\nabla_h^{[K:\mathbb{Q}_p]}).$$
  Combining all these equalities, we obtain the equality
  $$\mathrm{char}_{\Lambda_{\infty}}((t^h\bold{B}^{\dagger}_{\mathrm{rig},K}/t^{h+1}\bold{B}^{\dagger}_{\mathrm{rig},K})^{\psi=1})
  =(\nabla_h^{[K:\mathbb{Q}_p]}),$$ which proves the lemma, hence proves the theorem.

  \end{proof}

  \subsection{crystalline case}
  In this final subsection, we compare our results obtained in the last two subsections with 
  the previous results of Perrin-Riou when 
  $K$ is unramified over $\mathbb{Q}_p$ and $D$ is potentially crystalline such that $D|_{K_n}$ is crystalline for some $n\geqq 0$.
  After some preliminaries on the theory of $p$-adic Fourier transform, we recall the Berger's formula of Perrin-Riou's big exponential map $\Omega_{D,h}$ 
  (\cite{Ber03}), which is a map 
  from a $\Lambda_{\infty}$-submodule of $\Lambda_{\infty}\otimes_{\mathbb{Q}_p}\bold{D}^{K_n}_{\mathrm{crys}}(D)$ to 
  $\bold{H}^1_{\mathrm{Iw}}(K, D)/\bold{H}^1_{\mathrm{Iw}}(K, D)_{\mathrm{tor}}$.
  We next recall the statements of Perrin-Riou's $\delta(V)$. 
  Finally, we compare our exponential map $\mathrm{Exp}_{D,h}$ with Perrin-Riou's big exponential map.  In particular, 
 we show that our $\delta(D)$ is equivalent to Perrin-Riou's $\delta(V)$ in the unramified and crystalline case.

  If $K$ is unramified,  the cyclotomic character gives an isomorphism 
  $\chi:\Gamma_K\isom \mathbb{Z}^{\times}_p$. If we set $T:=[\varepsilon]-1$, then $\bold{B}^{\dagger}_{\mathrm{rig},K}=\cup_{r>0}\bold{B}^{\dagger,r}_{\mathrm{rig},K}$ 
   can be written as 
   $$\bold{B}^{\dagger, r}_{\mathrm{rig}, K}:=\{f(T):=\sum_{n\in \mathbb{Z}}a_nT^{n}| \,a_n\in K\, \text{and}\, f(T) \,\text{is convergent on }\, 
p^{-1/r}\leqq |T|_p<1\}.$$
   and the actions of 
  $\varphi$ and $\gamma\in\Gamma_K$ are given by the formula 
  $$\varphi(\sum_{n\in \mathbb{Z}}a_nT^n)
  :=\sum_{n\in \mathbb{Z}}\varphi(a_n)((1+T)^p-1)^n, \,\,\gamma(\sum_{n\in \mathbb{Z}}a_nT^n)
  :=\sum_{n\in\mathbb{Z}}a_n((1+T)^{\chi(\gamma)}-1)^n.$$ 
  We define a $\varphi$ and $\Gamma_K$-stable subring $\bold{B}^+_{\mathrm{rig},K}$ of $\bold{B}^{\dagger}_{\mathrm{rig},K}$ by 
  $$\bold{B}^+_{\mathrm{rig},K}:=\{f(T)=\sum_{n=0}^{+\infty}a_nT^n| \,a_n\in K\, \text{and}\, f(T) \text{ is convergent on } 
   0\leqq |T|_p<1\}.$$ 
  We have natural $\varphi$- and $\Gamma_K\isom \Gamma_{\mathbb{Q}_p}$-equivariant isomorphisms
  $$\bold{B}^{\dagger}_{\mathrm{rig}, \mathbb{Q}_p}\otimes_{\mathbb{Q}_p}K\isom \bold{B}^{\dagger}_{\mathrm{rig},K}, \,\,\,
  \bold{B}^+_{\mathrm{rig},\mathbb{Q}_p}\otimes_{\mathbb{Q}_p}K\isom \bold{B}^+_{\mathrm{rig},K}:f(T)\otimes a\mapsto af(T).$$
 One has a $\Lambda_{\infty}$-linear isomorphism defined  by 
  $$\Lambda_{\infty}\isom (\bold{B}^+_{\mathrm{rig},\mathbb{Q}_p})^{\psi=0}:\lambda\mapsto 
  \lambda\cdot(1+T).$$
 We remark that the definition of this isomorphism depends on the choice of $T$, i.e the choice of $\{\zeta_{p^n}\}_{n\geqq 1}$.
  In this subsection, we consider potentially crystalline $(\varphi,\Gamma)$-modules $D$ over $\bold{B}^{\dagger}_{\mathrm{rig},K}$ 
  such that $D|_{K_n}$ are crystalline  for some $n\geqq 0$.
  
  We first need to study the relationship between 
  $\bold{N}_{\mathrm{rig}}(D)^{\psi=1}$ and $\Lambda_{\infty}\otimes_{\mathbb{Q}_p}\bold{D}^{K_n}_{\mathrm{crys}}(D)$.
  
 \begin{lemma}\label{3.16}
 Let  $D$ be a potentially crystalline $(\varphi,\Gamma_K)$-module over $\bold{B}^{\dagger}_{\mathrm{rig},K}$ such that 
 $D|_{K_n}$ is crystalline  for some $n\geqq 0$. Then there exists an isomorphism of $(\varphi,\Gamma_K)$-modules
  over $\bold{B}^{\dagger}_{\mathrm{rig},K}$
  $$\bold{N}_{\mathrm{rig}}(D)\isom \bold{B}^{\dagger}_{\mathrm{rig},K}\otimes_K \bold{D}^{K_n}_{\mathrm{crys}}(D),$$
  where, on the right hand side, $\varphi$ and $\Gamma_K$ act diagonally.
  \end{lemma}
  \begin{proof}
  Since
  the natural map 
  $$\bold{B}^{\dagger}_{\mathrm{rig},K}[1/t]\otimes_{K}\bold{D}^{K_n}_{\mathrm{crys}}(D)\rightarrow D[1/t]: f(T)\otimes x\mapsto f(T)x$$
  is isomorphism, the natural map 
  $$\bold{B}^{\dagger}_{\mathrm{rig},K}\otimes_K \bold{D}^{K_n}_{\mathrm{crys}}(D)\rightarrow D[1/t]:f(T)\otimes x\mapsto f(T)x$$ is injective. 
  Then, it is easy to see that $\bold{B}^{\dagger}_{\mathrm{rig},K}\otimes_K \bold{D}^{K_n}_{\mathrm{crys}}(D)\subseteq D[1/t]$ satisfies the conditions 
  (1) and (2) of Theorem \ref{3.5}. Hence $\bold{B}^{\dagger}_{\mathrm{rig},K}\otimes_K \bold{D}^{K_n}_{\mathrm{crys}}(D)\isom \bold{N}_{\mathrm{rig}}(D)$ by 
  the uniqueness of $\bold{N}_{\mathrm{rig}}(D)$.
  \end{proof}
  
  By this lemma, $\bold{B}^+_{\mathrm{rig},K}\otimes_K \bold{D}^{K_n}_{\mathrm{crys}}(D)$ can be seen as a $\varphi$ and $\Gamma_K$ stable 
  submodule of $\bold{N}_{\mathrm{rig}}(D)$. Since we have an isomorphism $$\Lambda_{\infty}\otimes_{\mathbb{Q}_p}\bold{D}^{K_n}_{\mathrm{crys}}(D)
  \isom (\bold{B}^+_{\mathrm{rig},\mathbb{Q}_p})^{\psi=0}\otimes_{\mathbb{Q}_p}\bold{D}^{K_n}_{\mathrm{crys}}(D)\isom 
  (\bold{B}^+_{\mathrm{rig},K}\otimes_K \bold{D}^{K_n}_{\mathrm{crys}}(D))^{\psi=0}$$ and
  the map  $(\varphi-1)$ sends $(\bold{B}^+_{\mathrm{rig},K}\otimes_K\bold{D}^{K_n}_{\mathrm{crys}}(D))^{\psi=1}$ to 
  $(\bold{B}^+_{\mathrm{rig},K}\otimes_K \bold{D}^{K_n}_{\mathrm{crys}}(D))^{\psi=0}$, to study the relationship between 
  $\Lambda_{\infty}\otimes_{\mathbb{Q}_p}\bold{D}^{K_n}_{\mathrm{crys}}(D)$ and 
  $\bold{N}_{\mathrm{rig}}(D)^{\psi=1}$, we need to study the inclusion $(\bold{B}^+_{\mathrm{rig},K}\otimes_K\bold{D}^{K_n}_{\mathrm{crys}}(D))^{\psi=1}
  \hookrightarrow \bold{N}_{\mathrm{rig}}(D)^{\psi=1}$ and the map $$\varphi-1: (\bold{B}^+_{\mathrm{rig},K}\otimes_K\bold{D}^{K_n}_{\mathrm{crys}}(D))^{\psi=1}
  \rightarrow (\bold{B}^+_{\mathrm{rig},K}\otimes_K\bold{D}^{K_n}_{\mathrm{crys}}(D))^{\psi=0}.$$

 
 Before studying these maps, we recall some facts concerning $p$-adic Fourier transform (see $\S$ 2.6 of \cite{Ch12}).
 Let $f:\mathbb{Z}_p\rightarrow \mathbb{Q}_p$ be a map and $h\in \mathbb{Z}_{\geqq 0}$. We say that $f$ is locally $h$-analytic if, 
 for each $x\in \mathbb{Z}_p$, there exists $\{a_n(x)\}_{n\geqq 0}\subseteq \mathbb{Q}_p$ 
 such that $f(x+p^hy)=\sum_{n=0}^{\infty}a_n(x)y^h$ for any $y\in \mathbb{Z}_p$. 
 We define
 $$\mathrm{LA}_h(\mathbb{Z}_p, \mathbb{Q}_p):=\{f:\mathbb{Z}_p\rightarrow \mathbb{Q}_p| f\,\text{ is locally }\,h\text{-analytic} \}$$
  and 
 $$\mathrm{LA}(\mathbb{Z}_p,\mathbb{Q}_p):=\varinjlim_{h}\mathrm{LA}_h(\mathbb{Z}_p,\mathbb{Q}_p).$$ 
 $\mathrm{LA}_{h}(\mathbb{Z}_p, \mathbb{Q}_p)$ is a $\mathbb{Q}_p$-Banach space whose norm $|-|_h$ is defined by 
 $$|f|_h:=\mathrm{sup}_{x\in \mathbb{Z}_p, n\geqq 0}|a_n(x)|_p.$$ We define the actions of $\varphi,\psi$ and  $\gamma\in\Gamma_{K}\isom \Gamma_{\mathbb{Q}_p}$ on $\mathrm{LA}(\mathbb{Z}_p,\mathbb{Q}_p)$ 
 by 
  $$\varphi(f)(x):=
  \begin{cases}
      0 & ( \text{if $ x\in \mathbb{Z}_p^{\times}$}) \\
      f(\frac{x}{p}) &(\text{if $ x\in p\mathbb{Z}_p$}) ,
      \end{cases}
     $$
 $$\psi(f)(x):=f(px), \,\,\,\gamma(f)(x):=\frac{1}{\chi(\gamma)}f(\frac{x}{\chi(\gamma)}). $$
 We define a map $\mathrm{Col}:\bold{B}^{\dagger}_{\mathrm{rig}, \mathbb{Q}_p}\rightarrow \mathrm{LA}(\mathbb{Z}_p,\mathbb{Q}_p)$, which we call Colmez transform, 
  by 
 $$\mathrm{Col}(f)(x):=\mathrm{Res}((1+T)^xf(T)\frac{dT}{1+T}) \,\, \text{ for each } x\in \mathbb{Z}_p,$$
 where $\mathrm{Res}:\bold{B}^{\dagger}_{\mathrm{rig},\mathbb{Q}_p}\rightarrow \mathbb{Q}_p$ is the residue map 
 defined by $$\mathrm{Res}(\sum_{n\in \mathbb{Z}}a_nT^n):=a_{-1}.$$
  The map $\mathrm{Col}$  commutes with the actions of
 $\psi,\varphi$ and $\Gamma_{\mathbb{Q}_p}$ and we have $\mathrm{Ker}(\mathrm{Col})=\bold{B}^+_{\mathrm{rig},\mathbb{Q}_p}$. Hence we obtain 
 the following short exact sequence 
 $$0\rightarrow \bold{B}^+_{\mathrm{rig},\mathbb{Q}_p}\rightarrow \bold{B}^{\dagger}_{\mathrm{rig},\mathbb{Q}_p}\xrightarrow{\mathrm{Col}}\mathrm{LA}(\mathbb{Z}_p,\mathbb{Q}_p)
 \rightarrow 0.$$
 For each $k\in \mathbb{Z}_{\geqq 0}$, we define a locally analytic function $x^k:\mathbb{Z}_p\rightarrow \mathbb{Q}_p:y\mapsto y^k$. This function 
 satisfies that $$\psi(x^k)=p^kx^k \,\,\text{and}\,\,\gamma(x^k)=\chi(\gamma)^{-(k+1)}x^k.$$

 \begin{lemma}\label{3.17}
 Let $D_0$ be a $\varphi$-module over $\mathbb{Q}_p$, i.e. $D_0$ is a finite dimensional $\mathbb{Q}_p$-vector space with a $\mathbb{Q}_p$-linear 
 automorphism $\varphi:D_0\isom D_0$. Then,  for sufficiently large $k_0\gg0$, 
 we have the following equalities;
 \begin{itemize}
 \item[(1)]$\bigoplus_{k=0}^{k_0}(t^k \otimes D_0)^{\varphi=1}=\bigoplus_{k=0}^{\infty}(t^k\otimes D_0)^{\varphi=1}=(\bold{B}^+_{\mathrm{rig},\mathbb{Q}_p}\otimes_{\mathbb{Q}_p} D_0)^{\varphi=1}$,
 \item[(2)]
 \begin{multline*}
 \bigoplus_{k= 0}^{k_0} (t^k\otimes D_0)/(1-\varphi)(t^k\otimes D_0)=\bigoplus_{k= 0}^{\infty} (t^k\otimes D_0)/(1-\varphi)(t^k\otimes D_0)\\
   \isom (\bold{B}^+_{\mathrm{rig}, \mathbb{Q}_p}\otimes_{\mathbb{Q}_p} D_0)/(1-\varphi) (\bold{B}^+_{\mathrm{rig}, \mathbb{Q}_p}\otimes_{\mathbb{Q}_p} D_0),
   \end{multline*}
 \item[(3)]$(\bold{B}^+_{\mathrm{rig},\mathbb{Q}_p}\otimes_{\mathbb{Q}_p}D_0)/(1-\psi)(\bold{B}^+_{\mathrm{rig},\mathbb{Q}_p}\otimes_{\mathbb{Q}_p}D_0)=0$,
 \item[(4)]$\bigoplus_{k=0}^{k_0}(x^k\otimes D_0)^{\psi=1}=\bigoplus_{k=0}^{\infty}(x^k\otimes D_0)^{\psi=1}=(\mathrm{LA}(\mathbb{Z}_p,\mathbb{Q}_p)\otimes_{\mathbb{Q}_p}D_0)^{\psi=1}$, 
  \item[(5)]
 \begin{multline*} \bigoplus_{k=0}^{k_0}(x^k\otimes D_0)/(1-\psi)(x^k\otimes D_0)=\bigoplus_{k=0}^{\infty}(x^k\otimes D_0)/(1-\psi)(x^k\otimes D_0)\\
 \isom 
 (\mathrm{LA}(\mathbb{Z}_p,\mathbb{Q}_p)\otimes_{\mathbb{Q}_p}D_0)/(1-\psi) (\mathrm{LA}(\mathbb{Z}_p,\mathbb{Q}_p)\otimes_{\mathbb{Q}_p}D_0),
 \end{multline*}
 
 \end{itemize}
 where we define $t^k\otimes D_0:=\mathbb{Q}_pt^k\otimes_{\mathbb{Q}_p}D_0$ and 
 $x^k\otimes D_0:=\mathbb{Q}_px^k\otimes_{\mathbb{Q}_p}D_0$  for each $k\geqq 0$.

 \end{lemma}
 \begin{proof}
 When $D_0$ is one dimensional, then all these properties are  proved in $\S$ 2 of \cite{Ch12}. In the general case, this lemma 
 can be proved in the same way, so we omit the proof.
 
 \end{proof}
 We go back to our situation. Let $D$ be a potentially crystalline $(\varphi,\Gamma_K)$-module 
 over $\bold{B}^{\dagger}_{\mathrm{rig},K}$ such that $D|_{K_n}$ is crystalline for some $n\geqq 0$.
 We define a $\Lambda_{\infty}$-linear morphism 
 
 $$\widetilde{\Delta}:(\bold{B}^+_{\mathrm{rig},K}\otimes_{K}\bold{D}^{K_n}_{\mathrm{crys}}(D))^{\psi=0} 
 \rightarrow \bigoplus_{k= 0 }^{\infty}t^k\otimes \bold{D}^{K_n}_{\mathrm{crys}}(D)/(1-\varphi)(t^k\otimes \bold{D}^{K_n}_{\mathrm{crys}}(D))$$
 by 
 $$\widetilde{\Delta}(\sum_{i=1}^m f_i(T)\otimes z_i):=( \overline{t^k\otimes (\sum_{i=1}^m\partial^k(f_i)(0)\cdot z_i)})_{k\geqq 0},$$
 where we recall that $\partial(f)(T)=(1+T)\frac{df(T)}{dT}$.

 The following lemma was proved in $\S$ 2.2 of \cite{Per94}, but here we re-prove it using the above lemma. 
  \begin{lemma}\label{3.18}
 There exists a following exact sequence of $\Lambda_{\infty}$-modules

\begin{multline*}
  0\rightarrow  \bigoplus_{k=0}^{\infty}(t^k\otimes\bold{D}^{K_n}_{\mathrm{crys}}(D))^{\varphi=1}\rightarrow(\bold{B}^+_{\mathrm{rig},K}\otimes_K\bold{D}^{K_n}_{\mathrm{crys}}(D))^{\psi=1}\\
 \xrightarrow{\varphi-1}
  (\bold{B}^+_{\mathrm{rig},K}\otimes_K\bold{D}^{K_n}_{\mathrm{crys}}(D))^{\psi=0}\xrightarrow{\widetilde{\Delta}} \bigoplus_{k=0}^{\infty}
  (t^k\otimes \bold{D}^{K_n}_{\mathrm{crys}}(D))/(1-\varphi)(t^k\otimes\bold{D}^{K_n}_{\mathrm{crys}}(D))\rightarrow  0.
  \end{multline*}

  \end{lemma}
  \begin{proof}
  Since we have an inclusion 
  $$(1-\varphi)(\bold{B}^+_{\mathrm{rig},K}\otimes_K\bold{D}^{K_n}_{\mathrm{crys}}(D))^{\psi=1}\subseteq 
  (\bold{B}^+_{\mathrm{rig},K}\otimes_K\bold{D}^{K_n}_{\mathrm{crys}}(D))^{\psi=0},$$ we have the following exact sequence
  \begin{multline*}
0\rightarrow  \bigoplus_{k=0}^{\infty}(t^k\otimes\bold{D}^{K_n}_{\mathrm{crys}}(D))^{\varphi=1}
  \rightarrow (\bold{B}^+_{\mathrm{rig},K}\otimes_K\bold{D}^{K_n}_{\mathrm{crys}}(D))^{\psi=1}\\
  \xrightarrow{1-\varphi}(\bold{B}^+_{\mathrm{rig},K}\otimes_K\bold{D}^{K_n}_{\mathrm{crys}}(D))^{\psi=0}
  \rightarrow (\bold{B}^+_{\mathrm{rig},K}\otimes_K\bold{D}^{K_n}_{\mathrm{crys}}(D))/(1-\varphi)(\bold{B}^+_{\mathrm{rig},K}\otimes_K\bold{D}^{K_n}_{\mathrm{crys}}(D)),
  \end{multline*}
  where the exactness at the second arrow  follows from the equality 
  $$\bigoplus_{k=0}^{\infty}(t^k\otimes\bold{D}^{K_n}_{\mathrm{crys}}(D))^{\varphi=1}=(\bold{B}^+_{\mathrm{rig},K}\otimes_K\bold{D}^{K_n}_{\mathrm{crys}}(D))^{\varphi=1}$$
   which is proved in  (1) of Lemma \ref{3.17}.
  We show that the natural map 
   $$(\bold{B}^+_{\mathrm{rig},K}\otimes_K \bold{D}^{K_n}_{\mathrm{crys}}(D))^{\psi=0}\rightarrow (\bold{B}^+_{\mathrm{rig},K}\otimes_K\bold{D}^{K_n}_{\mathrm{crys}}(D))/(1-\varphi)
   (\bold{B}^+_{\mathrm{rig},K}\otimes_K\bold{D}^{K_n}_{\mathrm{crys}}(D)):z\mapsto \overline{z}$$ 
   is a surjection. To prove this claim, let $z$ be an element of $\bold{B}^+_{\mathrm{rig},K}\otimes_K \bold{D}^{K_n}_{\mathrm{crys}}(D)$.
   Then it suffices to show that there exists $y\in \bold{B}^+_{\mathrm{rig},K}\otimes_K \bold{D}^{K_n}_{\mathrm{crys}}(D)$ such that 
   $\psi(z-(1-\varphi)y)=0$. Because we have $\psi(z-(1-\varphi)y)=\psi(z)-(\psi-1)y$, such $y$ exists by (3) of Lemma \ref{3.17}.

  By this claim and because we have a natural  isomorphism 
  $$\bigoplus_{k=0}^{\infty} t^k\otimes\bold{D}^{K_n}_{\mathrm{crys}}(D)/(1-\varphi)(t^k\otimes \bold{D}^{K_n}_{\mathrm{crys}}(D))\isom(\bold{B}^+_{\mathrm{rig},K}\otimes_K\bold{D}^{K_n}_{\mathrm{crys}}(D))/(1-\varphi)(\bold{B}^+_{\mathrm{rig},K}\otimes_K\bold{D}^{K_n}_{\mathrm{crys}}(D))
 $$ by Lemma \ref{3.17}, we obtain the surjection  
 $$(\bold{B}^+_{\mathrm{rig},K}\otimes_K\bold{D}^{K_n}_{\mathrm{crys}}(D))^{\psi=0}\rightarrow \bigoplus_{k=0}^{\infty}
  (t^k\otimes \bold{D}^{K_n}_{\mathrm{crys}}(D))/(1-\varphi)(t^k\otimes\bold{D}^{K_n}_{\mathrm{crys}}(D))$$ which is explicitly defined by $$\sum_{i=1}^m f_i(T)\otimes x_i\mapsto 
  ( \frac{1}{k!}\overline{t^k\otimes (\sum_{i=1}^m\partial^k(f_i)(0)\cdot x_i)})_{k\geqq 0}.$$
 
   Since this map and $\widetilde{\Delta}$ are only 
  differ by a factor of $k!$ at each $k$-th component, their kernels and images are equal. Hence we 
  finish to prove the exactness of the sequence in this lemma.

  \end{proof}

  The following definition is  Berger's formula for Perrin-Riou's big exponential map. More precisely, Berger defined 
  Perrin-Riou's map for crystalline $p$-adic representations and the following definition is just the direct generalization of 
  his formula for potentially crystalline $(\varphi,\Gamma)$-modules.
  
  \begin{defn}
  Let $D$ be a potentially crystalline $(\varphi,\Gamma_K)$-module over $\bold{B}^{\dagger}_{\mathrm{rig},K}$ such that 
  $D|_{K_n}$ is crystalline for some $n\geqq 0$ and let $h\geqq 1$ be an integer such that 
  $\mathrm{Fil}^{-k}\bold{D}^K_{\mathrm{dR}}(D))=D^{K}_{\mathrm{dR}}(D)$. 
  Then, we define a $\Lambda_{\infty}$-linear map 
  $$\Omega_{D,h}: (\Lambda_{\infty}\otimes_{\mathbb{Q}_p}\bold{D}^{K_n}_{\mathrm{crys}}(D))^{\widetilde{\Delta}=0}
  \rightarrow \bold{H}^1_{\mathrm{Iw}}(K, D)/\bold{H}^1_{\mathrm{Iw}}(K, D)_{\mathrm{tor}}$$ 
  as the composition of the isomorphism
  $$(\varphi-1)^{-1}: (\Lambda_{\infty}\otimes_{\mathbb{Q}_p}\bold{D}^{K_n}_{\mathrm{crys}}(D))^{\widetilde{\Delta}=0}
  \isom (\bold{B}^+_{\mathrm{rig},K}\otimes_K \bold{D}^{K_n}_{\mathrm{crys}}(D))^{\psi=1}/(\bold{B}^+_{\mathrm{rig},K}\otimes_K\bold{D}^{K_n}_{\mathrm{crys}}(D))^{\varphi=1}$$
  with the natural inclusion
  $$(\bold{B}^+_{\mathrm{rig},K}\otimes_K\bold{D}^{K_n}_{\mathrm{crys}}(D))^{\psi=1}/(\bold{B}^+_{\mathrm{rig},K}\otimes_K \bold{D}^{K_n}_{\mathrm{crys}}(D))^{\varphi=1}
  \hookrightarrow \bold{N}_{\mathrm{rig}}(D)^{\psi=1}/\bold{N}_{\mathrm{rig}}(D)^{\varphi=1}$$
   and with the injection proved in Lemma \ref{3.13}
   $$\overline{\mathrm{Exp}}_{D,h}: \bold{N}_{\mathrm{rig}}(D)^{\psi=1}/\bold{N}_{\mathrm{rig}}(D)^{\varphi=1}\hookrightarrow \bold{H}^1_{\mathrm{Iw}}(K,D)/\bold{H}^1_{\mathrm{Iw}}(K,D)_{\mathrm{tor}}.$$
  \end{defn}
  
  \begin{rem}
  Let $V$ be a crystalline representation of $G_K$ and let $D(V)$ be the $(\varphi,\Gamma_K)$-module 
  over $\bold{B}^{\dagger}_{\mathrm{rig},K}$ associated to $V$. If we admit the natural isomorphisms 
  $\Lambda_{\infty}\otimes_{\Lambda}\bold{H}^1_{\mathrm{Iw}}(K, V)\isom \bold{H}^1_{\mathrm{Iw}}(K, D)$ (see $\S$ 2 of \cite{Po12b}) and 
  $\bold{D}^K_{\mathrm{crys}}(V)\isom \bold{D}^K_{\mathrm{crys}}(D(V))$, 
   Berger proved that the map 
   \begin{multline*}
\Omega_{V,h}:(\Lambda_{\infty}\otimes_{\mathbb{Q}_p}\bold{D}^K_{\mathrm{crys}}(V))^{\widetilde{\Delta}=0}\isom 
  (\Lambda_{\infty}\otimes_{\mathbb{Q}_p}\bold{D}^K_{\mathrm{crys}}(D(V)))^{\widetilde{\Delta}=0}\\
  \xrightarrow{\Omega_{D,h}}\bold{H}^1_{\mathrm{Iw}}(K, D(V))/\bold{H}^1_{\mathrm{Iw}}(K, D(V))_{\mathrm{tor}}\isom 
  \Lambda_{\infty}\otimes_{\Lambda}(\bold{H}^1_{\mathrm{Iw}}(K, V)/\bold{H}^1_{\mathrm{Iw}}(K, V)_{\mathrm{tor}})
  \end{multline*}
  coincides with Perrin-Riou's original map defined in \cite{Per94} (see Theorem 2.13 of \cite{Ber03}).
  
  \end{rem}
  
  To state Perrin-Riou's $\delta(V)$, we slightly generalize the definition of $\mathrm{det}_{\Lambda_{\infty}}(-)$ to the following situation.
   Let $M_1$ and $M_2$ be co-admissible $\Lambda_{\infty}$-modules.
  We assume that there exist co-admissible $\Lambda_{\infty}$-submodules $M_1'\subseteq M_1$ and $M_2'\subseteq M_2$ such that 
  $M_1/M_1'$ and $M_2'$ are torsion $\Lambda_{\infty}$-modules and that there exists a $\Lambda_{\infty}$-linear 
  map $f:M_1'\rightarrow M_2/M_2'$ for which we can define $\mathrm{det}_{\Lambda_{\infty}}(f)$. Under this situation, we define a fractional ideal 
  $\mathrm{det}_{\Lambda_{\infty}}(f:M_1\rightarrow M_2)\subseteq \mathrm{Frac}(\Lambda_{\infty})$ by 
  $$\mathrm{det}_{\Lambda_{\infty}}(f:M_1\rightarrow M_2):=
  \mathrm{det}_{\Lambda_{\infty}}(f:M_1'\rightarrow M_2/M_2')\mathrm{char}_{\Lambda_{\infty}}(M_1/M_1')^{-1}\mathrm{char}_{\Lambda_{\infty}}(M_2').$$
  
  We apply this definition to the map 
  $$\Omega_{D,h}:(\Lambda_{\infty}\otimes_{\mathbb{Q}_p}\bold{D}^{K_n}_{\mathrm{crys}}(D))^{\widetilde{\Delta}=0}
  \rightarrow \bold{H}^1_{\mathrm{Iw}}(K, D)/\bold{H}^1_{\mathrm{Iw}}(K, D)_{\mathrm{tor}},$$
   i.e, we define the principal fractional ideal
  $$\mathrm{det}_{\Lambda_{\infty}}(\Omega_{D,h}: \Lambda_{\infty}\otimes_{\mathbb{Q}_p}\bold{D}^{K_n}_{\mathrm{crys}}(D)\rightarrow 
  \bold{H}^1_{\mathrm{Iw}}(K, D))$$ 
  by the product
  \begin{multline*}
  \mathrm{det}_{\Lambda_{\infty}}(\Omega_{D,h}:(\Lambda_{\infty}\otimes_{\mathbb{Q}_p}\bold{D}^{K_n}_{\mathrm{crys}}(D))^{\widetilde{\Delta}=0}
  \rightarrow \bold{H}^1_{\mathrm{Iw}}(K, D)/\bold{H}^1_{\mathrm{Iw}}(K, D)_{\mathrm{tor}})\cdot\\
 \mathrm{char}_{\Lambda_{\infty}}(\Lambda_{\infty}\otimes_{\mathbb{Q}_p}\bold{D}^{K_n}_{\mathrm{crys}}(D)/(\Lambda_{\infty}\otimes_{\mathbb{Q}_p}
  \bold{D}^{K_n}_{\mathrm{crys}}(D))^{\widetilde{\Delta}=0})^{-1}\cdot\mathrm{char}_{\Lambda_{\infty}}(\bold{H}^1_{\mathrm{Iw}}(K, D)_{\mathrm{tor}}).
  \end{multline*}
  
  Using  this definition,  Perrin-Riou's $\delta(V)$-theorem can be stated as follows. 
  More precisely, the following is the direct generalization of Perrin-Riou's $\delta(V)$ to
  any slope crystalline $D$.
  In Proposition \ref{3.23} below, we will prove the theorem by proving that  the theorem 
  is equivalent to Theorem \ref{3.14}.
   \begin{thm}\label{3.21}
  Let $D$ be a potentially crystalline $(\varphi,\Gamma_K)$-module over $\bold{B}^{\dagger}_{\mathrm{rig},K}$ such that 
  $D|_{K_n}$ is crystalline. Let $\{h_1,\cdots, h_d\}$ be the set of Hodge-Tate weights of $D$ and let $h\geqq 1$  be an integer 
  such that $\mathrm{Fil}^{-h}\bold{D}^K_{\mathrm{dR}}(D)=\bold{D}^K_{\mathrm{dR}}(D)$. 
  Then, we have an equality of fractional ideals of $\mathrm{Frac}(\Lambda_{\infty})$
   \begin{multline*}
\mathrm{det}(\Omega_{D,h}:\Lambda_{\infty}\otimes_{\mathbb{Q}_p}\bold{D}^{K_n}_{\mathrm{crys}}(D)\rightarrow 
  \bold{H}^1_{\mathrm{Iw}}(K, D))\\
  =(\prod_{1\leqq i\leqq d}\nabla_{h_{i}}\nabla_{h_{i}+1}\cdots \nabla_{h-1})^{[K:\mathbb{Q}_p]}\cdot\mathrm{char}_{\Lambda_{\infty}}(\bold{H}^2_{\mathrm{Iw}}(K, D)).
  \end{multline*}

  \end{thm}
  \begin{rem}
  On the other hand, for any slope potentially crystalline $D$, Pottharst defined 
  the ``inverse"  map 
  $$\mathrm{Log}_D:\bold{H}^1_{\mathrm{Iw}}(K, D)\rightarrow \mathrm{Frac}(\Lambda_{\infty})\otimes_{\mathbb{Q}_p}\bold{D}^{K_n}_{\mathrm{crys}}(D)$$ 
  of $\Omega_{D,h}$ using the theory of Wach modules. Using $\mathrm{Log}_D$, he
   also proved his $\delta(D)$-theorem (Theorem 3.4 of \cite{Po12b}) by reducing to Perrin-Riou's 
   $\delta(V)$ using a slope filtration argument. It is easy to check that the theorem above is equivalent to his $\delta(D)$.
    \end{rem}
  
  The next proposition is the main result of this subsection, which says that, when $D$ is as above, our Theorem \ref{3.14} is equivalent to the above Theorem 
  \ref{3.21}. 
  
   \begin{prop}\label{3.23}
  We have an equality 
   \begin{multline*}
  \mathrm{det}_{\Lambda_{\infty}}(\bold{N}_{\mathrm{rig}}(D)^{\psi=1}\xrightarrow{\mathrm{Exp}_{D,h}}
  \bold{H}^1_{\mathrm{Iw}}(K, D))\cdot\mathrm{char}_{\Lambda_{\infty}}(\bold{H}^2_{\mathrm{Iw}}(K, \bold{N}_{\mathrm{rig}}(D)))\\
  =\mathrm{det}_{\Lambda_{\infty}}(\Lambda_{\infty}\otimes_{\mathbb{Q}_p}\bold{D}^{K_n}_{\mathrm{crys}}(D)\xrightarrow{\Omega_{D,h}} \bold{H}^1_{\mathrm{Iw}}(K, D)).
  \end{multline*}
  In particular, Theorem \ref{3.14} is equivalent to Theorem \ref{3.21}.
  
  \end{prop}
  \begin{proof}
  Since we have $\bold{N}_{\mathrm{rig}}(D)=\bold{B}^{\dagger}_{\mathrm{rig},K}\otimes_K\bold{D}^{K_n}_{\mathrm{crys}}(D)$ by Lemma \ref{3.16}, the principal fractional ideal
  $$
  \mathrm{det}_{\Lambda_{\infty}}(\bold{N}_{\mathrm{rig}}(D)^{\psi=1}\xrightarrow{\mathrm{Exp}_{D,h}}
  \bold{H}^1_{\mathrm{Iw}}(K, D))\cdot\mathrm{char}_{\Lambda_{\infty}}(\bold{H}^2_{\mathrm{Iw}}(K, \bold{N}_{\mathrm{rig}}(D)))$$ 
  is equal to the product
  \begin{multline*}
  \mathrm{det}_{\Lambda_{\infty}}((\bold{B}^+_{\mathrm{rig},K}\otimes_K \bold{D}^{K_n}_{\mathrm{crys}}(D))^{\psi=1}\xrightarrow{\mathrm{Exp}_{D,h}|_{(\bold{B}^+_{\mathrm{rig},K}\otimes_K\bold{D}^{K_n}_{\mathrm{crys}}(D))^{\psi=1}}}\bold{H}^1_{\mathrm{Iw}}(K, D))\cdot\\
 \mathrm{char}_{\Lambda_{\infty}}((\bold{B}^{\dagger}_{\mathrm{rig},K}\otimes_K\bold{D}^{K_n}_{\mathrm{crys}}(D))^{\psi=1}/(\bold{B}^+_{\mathrm{rig},K}\otimes_K \bold{D}^{K_n}_{\mathrm{crys}}(D))^{\psi=1})^{-1}\cdot
  \mathrm{char}_{\Lambda_{\infty}}(\bold{H}^2_{\mathrm{Iw}}(K, \bold{N}_{\mathrm{rig}}(D))).
  \end{multline*}
  Since we have 
  $$(\bold{B}^+_{\mathrm{rig},K}\otimes_K\bold{D}^{K_n}_{\mathrm{crys}}(D))/(\psi-1)(\bold{B}^+_{\mathrm{rig},K}\otimes_K\bold{D}^{K_n}_{\mathrm{crys}}(D))=0$$ by (3) of Lemma \ref{3.17}, 
  using the snake lemma, we obtain the following isomorphisms
  
  \begin{align*}
  (\bold{B}^{\dagger}_{\mathrm{rig}}\otimes_K\bold{D}^{K_n}_{\mathrm{crys}}(D))^{\psi=1}/(\bold{B}^+_{\mathrm{rig},K}\otimes_K \bold{D}^{K_n}_{\mathrm{crys}}(D))^{\psi=1}
   &\isom (\mathrm{LA}(\mathbb{Z}_p,\mathbb{Q}_p)\otimes_{\mathbb{Q}_p}\bold{D}^{K_n}_{\mathrm{crys}}(D))^{\psi=1}\\
   &\isom \bigoplus_{k=0}^{k_0} (x^k\otimes\bold{D}^{K_n}_{\mathrm{crys}}(D))^{\psi=1},
  \end{align*}
 where the last isomorphism is  (4) of Lemma \ref{3.17} for sufficiently large  $k_0\gg0$. We similarly obtain an isomorphism 
 $$\bold{H}^2_{\mathrm{Iw}}(K, \bold{N}_{\mathrm{rig}}(D))\isom \bigoplus_{k=0}^{k_0} (x^k\otimes\bold{D}^{K_n}_{\mathrm{crys}}(D))/(1-\psi)(x^k\otimes\bold{D}^{K_n}_{\mathrm{crys}}(D)).$$
Hence, we obtain 
 
  \begin{multline*}
  \mathrm{char}_{\Lambda_{\infty}}((\bold{B}^{\dagger}_{\mathrm{rig}}\otimes_K\bold{D}^{K_n}_{\mathrm{crys}}(D))^{\psi=1}
  /(\bold{B}^+_{\mathrm{rig},K}\otimes_K \bold{D}^{K_n}_{\mathrm{crys}}(D))^{\psi=1})^{-1}\cdot
  \mathrm{char}_{\Lambda_{\infty}}(\bold{H}^2_{\mathrm{Iw}}(K, \bold{N}_{\mathrm{rig}}(D)))\\
  =\mathrm{det}_{\Lambda_{\infty}}(\bigoplus_{k=0}^{k_0}x^k\otimes\bold{D}^{K_n}_{\mathrm{crys}}(D)\xrightarrow{\psi-1} (\bigoplus_{k=0}^{k_0} x^k\otimes\bold{D}^{K_n}_{\mathrm{crys}}(D))
  =\Lambda_{\infty}.
  \end{multline*}
  where the last equality follows from (v) of Lemma \ref{3.12}.
  Hence, we obtain an equality
  \begin{multline*}
  \mathrm{det}_{\Lambda_{\infty}}(\bold{N}_{\mathrm{rig}}(D)^{\psi=1}\xrightarrow{\mathrm{Exp}_{D,h}} 
  \bold{H}^1_{\mathrm{Iw}}(K, D))\cdot\mathrm{char}_{\Lambda_{\infty}}(\bold{H}^2_{\mathrm{Iw}}(K, \bold{N}_{\mathrm{rig}}(D)))\\
  =\mathrm{det}_{\Lambda_{\infty}}((\bold{B}^+_{\mathrm{rig},K}\otimes_K \bold{D}^{K_n}_{\mathrm{crys}}(D))^{\psi=1}\xrightarrow{\mathrm{Exp}_{D,h}|_{(\bold{B}^+_{\mathrm{rig},K}\otimes_K\bold{D}^{K_n}_{\mathrm{crys}}(D))^{\psi=1}}}\bold{H}^1_{\mathrm{Iw}}(K, D)).
  \end{multline*}
  Next, we calculate the right hand side of the proposition.
  
  First, by the definition of $\Omega_{D,h}$ and by the property of $\mathrm{det}_{\Lambda_{\infty}}(-)$, the fractional ideal $$\mathrm{det}_{\Lambda_{\infty}}(\Lambda_{\infty}\otimes_{\mathbb{Q}_p}\bold{D}^{K_n}_{\mathrm{crys}}(D)\xrightarrow{\Omega_{D,h}} \bold{H}^1_{\mathrm{Iw}}(K, D))$$
  is equal to the product
  \begin{multline*}
  \mathrm{det}_{\Lambda_{\infty}}((\bold{B}^+_{\mathrm{rig},K}\otimes_K \bold{D}^{K_n}_{\mathrm{crys}}(D))^{\psi=1}\xrightarrow{\mathrm{Exp}_{D,h}|_{(\bold{B}^+_{\mathrm{rig},K}\otimes_K\bold{D}^{K_n}_{\mathrm{crys}}(D))^{\psi=1}}}\bold{H}^1_{\mathrm{Iw}}(K, D))\cdot\\
    \,\, \mathrm{det}_{\Lambda_{\infty}}( ((\bold{B}^+_{\mathrm{rig},K}\otimes_K \bold{D}^{K_n}_{\mathrm{crys}}(D))^{\psi=1}
    \xrightarrow{1-\varphi} \Lambda_{\infty}\otimes_{\mathbb{Q}_p}\bold{D}^{K_n}_{\mathrm{crys}}(D)))^{-1}.
  \end{multline*}
  By Lemma \ref{3.18}, we have 
   \begin{multline*}
  \mathrm{det}_{\Lambda_{\infty}}((\bold{B}^+_{\mathrm{rig},K}\otimes_K \bold{D}^{K_n}_{\mathrm{crys}}(D))^{\psi=1}\xrightarrow{1-\varphi} \Lambda_{\infty}\otimes_{\mathbb{Q}_p}\bold{D}^{K_n}_{\mathrm{crys}}(D))\\
  =\mathrm{det}_{\Lambda_{\infty}}( \bigoplus_{k=0}^{k_0}t^k\otimes\bold{D}^{K_n}_{\mathrm{crys}}(D)\xrightarrow{1-\varphi} \bigoplus_{k=0}^{k_0}t^k\otimes\bold{D}^{K_n}_{\mathrm{crys}}(D))=\Lambda_{\infty}.
  \end{multline*}
  Hence, we also obtain an equality
   \begin{multline*}
 \mathrm{det}_{\Lambda_{\infty}}(\Lambda_{\infty}\otimes_{\mathbb{Q}_p}\bold{D}^{K_n}_{\mathrm{crys}}(D)\xrightarrow{\Omega_{D,h}} \bold{H}^1_{\mathrm{Iw}}(K, D))\\
 =\mathrm{det}_{\Lambda_{\infty}}((\bold{B}^+_{\mathrm{rig},K}\otimes_K \bold{D}^{K_n}_{\mathrm{crys}}(D))^{\psi=1}\xrightarrow{\mathrm{Exp}_{D,h}|_{(\bold{B}^+_{\mathrm{rig},K}\otimes_K\bold{D}^{K_n}_{\mathrm{crys}}(D))^{\psi=1}}}\bold{H}^1_{\mathrm{Iw}}(K, D)),
 \end{multline*}
  
   which proves the proposition.

  \end{proof}
  \section*{List of notation}
  Here is a list of the main notation of the article, in the order of the section in which it appears.
  \begin{itemize}
  \item[$\S 1.1:$]$\mathrm{exp}_{K, V}$, $\mathrm{exp}^*_{K, V^{\lor}(1)}$.
  \item[$\S1.2:$]$\Lambda$, $\bold{H}^q_{\mathrm{Iw}}(K, V)$, $\Omega_{V,h}$.
  \item[Notation:]$p$, $K$, $K_0$, $\overline{K}$, $\mathbb{C}_p$, $v_p$, $|-|_p$, 
  $G_K$, $\{\zeta_{p^n}\}_{n\geqq 0}$, $K_n$, $K_{\infty}$, $\chi$,
  $\Gamma_K$, $e_1$, $e_k$, $|G|$.
  \item[$\S2.1:$]$\widetilde{\bold{E}}^+$, $v_{\widetilde{\bold{E}}^+}$, $\widetilde{\bold{E}}$, 
  $\varepsilon$, $\widetilde{p}$, $\widetilde{\bold{A}}^+$, $\widetilde{\bold{A}}$, $\theta$, $\bold{B}^+_{\mathrm{dR}}$, $t$, $\bold{B}_{\mathrm{dR}}$, $\widetilde{\bold{B}}^{\dagger}_{\mathrm{rig}}$,
   $\widetilde{\bold{A}}^{[r,s]}$, $\widetilde{\bold{B}}^{[r,s]}$, $\bold{B}^+_{\mathrm{max}}$, $\widetilde{\bold{B}}^{\dagger,r}_{\mathrm{rig}}$, $\widetilde{\bold{B}}^{\dagger}_{\mathrm{rig}}$, $r_n$, $\iota_n:\widetilde{\bold{B}}^{\dagger,r_n}_{\mathrm{rig}}\hookrightarrow \bold{B}^+_{\mathrm{dR}}$, $\bold{B}_{\mathrm{max}}$, $\bold{B}_e$, $T$, $\bold{B}^{\dagger,r}_{\mathrm{rig}, F}$, 
   $\bold{B}^{\dagger}_{\mathrm{rig}, F}$, $e_K$, $K_0'$, $r(K)$, $\pi_K$, $\bold{B}^{\dagger,r}_{\mathrm{rig},K}$, $\bold{B}^{\dagger}_{\mathrm{rig},K}$, $\psi$, $n(K)$, $\iota_n:\bold{B}^{\dagger,r_n}_{\mathrm{rig},K}\hookrightarrow K_n[[t]]$, 
   $\frac{1}{p}\mathrm{Tr}_{K_{n+1}/K_n}$, $D|_L$, $D^{\lor}$, $D_1\otimes D_2$, $n(D)$, 
   $D^{(n)}$, $\bold{D}^+_{\mathrm{dif}}(D)$, $\bold{D}_{\mathrm{dif}, n}(D)$, $\bold{D}^+_{\mathrm{dif}}(D)$, $\bold{D}_{\mathrm{dif}}(D)$, $K_{\infty}[[t]]$, 
   $K_{\infty}((t))$, $\iota_n:D^{(n)}\hookrightarrow \bold{D}_{\mathrm{dif}, n}(D)$.
   \item[$\S2.2:$]$\Delta_K$, $\gamma_K$, $M^{\Delta_K}$, $C^{\bullet}_{\gamma_K}(M)$, 
   $C^{\bullet}_{\varphi,\gamma_K}(M)$, $\mathrm{H}^q(K, D)$, $\mathrm{H}^q(K, D[1/t])$, 
   $\mathrm{H}^q(K, \bold{D}^+_{\mathrm{dif}}(D))$, $\mathrm{H}^q(K, \bold{D}_{\mathrm{dif}}(D))$, 
   $\cup$, $<,>$, $\mathrm{ev}$, $f_{\mathrm{tr}}$, $f'_{\mathrm{tr}}$, $\kappa$, $\mathrm{rec}_{\mathbb{Q}_p}$, $C^{\bullet}_{\psi,\gamma_K}(D)$, $\bold{D}^K_{\mathrm{crys}}(D)$, 
   $\bold{D}^K_{\mathrm{dR}}(D)$, $\mathrm{Fil}^i\bold{D}_{\mathrm{dR}}(D)$.
   \item[$\S 2.3:$]$\delta_{1,D}$, $\delta_{2,D}$, $\widetilde{C}^{\bullet}_{\varphi,\gamma_K}(D^{(n)})$, $\widetilde{C}^{\bullet}_{\varphi,\gamma_K}(D^{(n)}[1/t])$, $\widetilde{C}^{\bullet}_{\varphi,\gamma_K}(\bold{D}^+_{\mathrm{dif},n}(D))$, $\widetilde{C}^{\bullet}_{\varphi,\gamma_K}(\bold{D}_{\mathrm{dif},n}(D))$, $\mathrm{exp}_{K, D}$. 
   \item[$\S 2.4:$]$\cup_{\mathrm{dif}}$, $g_D$, $\mathrm{log}(\chi)$, 
   $<,>_{\mathrm{dif}}$, $\mathrm{exp}^*_{K, D^{\lor}(1)}$, $[-,-]_{\mathrm{dR}}$.
   \item[$\S2.5:$]$W$, $W_e$, $W_{\mathrm{dR}}^+$, $W_{\mathrm{dR}}$, $W(V)$, $W_e(D)$, 
   $W_{\mathrm{dR}}(D)$, $W^+_{\mathrm{dR}}(D)$, $W(D)$, $\widetilde{D}^{(n)}(W)$, 
   $\widetilde{D}(W)$, $D(W)$, $C^q(G_K, M)$, $\delta_q$, $C^{\bullet}(G_K, M)$, $C^{\bullet}(G_K, W)$, $\mathrm{H}^1(K, W)$, $\delta_{1, W}$, $\delta_{2, W}$, $\bold{D}^K_{\mathrm{dR}}(W)$, 
   $\mathrm{exp}_{K, W}$.
   \item[$\S 3.1:$]$\Gamma_{K, \mathrm{tor}}$, $\Gamma_{K, \mathrm{free}}$, $\Lambda_n$, 
   $\Lambda_{\infty}$, $\bold{B}^{+}_{\mathrm{rig}, \mathbb{Q}_p}$, $\widetilde{\Lambda}_n$, 
   $\widetilde{\Lambda}_n^{\iota}$, $D\widehat{\otimes}_{\mathbb{Q}_p}\widetilde{\Lambda}_n^{\iota}$, $\bold{H}^q_{\mathrm{Iw}}(K, D)$, $\widehat{\Gamma}_{K,\mathrm{tor}}$, $\eta$, $\alpha_{\eta}$, 
   $M_{\mathrm{tor}}$, $A(\delta)$, $f_{\delta}$, $\mathrm{pr}_{L, D(k)}$, $\delta_L$, 
   $\mathbb{Q}_p[\widetilde{\Gamma_K/\Gamma_L}]^{\iota}$, $f_{D, k}$, $f_k$, $C^{\bullet}_{\psi}(D)$, $\iota_D$, $p_{\Delta_K}$, $\mathrm{log}_0(-)$.
   \item[$\S 3.2:$]$\nabla_0$, $\widehat{\Omega}_{\bold{B}^{\dagger}_{\mathrm{rig}, K}/K_0'}$, 
   $\bold{N}^{(n)}_{\mathrm{rig}}(D)$, $\bold{N}_{\mathrm{rig}}(D)$, $\partial$, $\widetilde{\partial}$. 
   \item[$\S3.3:$]$\nabla_i$, $\mathrm{Exp}_{D, h}$, $T_L$, $m(L)$.
   \item[$\S3.4:$]$\mathrm{char}_{\Lambda_{\infty}}(M)$, $\mathrm{det}_{\Lambda_{\infty}}(f)$, 
   $\mathrm{det}_{\Lambda_{\infty}}(\mathrm{H}^{\bullet}(f))$. 
   \item[$\S3.5:$]$\bold{B}^+_{K,\mathrm{rig}}$, $\mathrm{LA}_h(\mathbb{Z}_p, \mathbb{Q}_p)$, 
   $\mathrm{LA}(\mathbb{Z}_p, \mathbb{Q}_p)$, $|-|_h$, $\mathrm{Col}$, $\mathrm{Res}$, $x^k$, 
   $\widetilde{\Delta}$, $\Omega_{D,h}$.

  \end{itemize}


\begin{thebibliography}{99}
\bibitem[Ben00]{Ben00}
D. Benois,  On Iwasawa theory of crystalline representations. Duke Math. J. 104 (2000) 211-267.

\bibitem[Ber02]{Ber02}
L.Berger, Repr\'esentations $p$-adiques et \'equations diff\'erentielles, Invent. Math. 148 (2002), 219-284.
\bibitem[Ber03]{Ber03}
 L. Berger. Bloch and Kato's exponential map: three explicit formulas. Doc. Math, 99-129 (electronic), 2003. Kazuya Kato's fiftieth birthday.
\bibitem[Ber08a]{Ber08a}
L.Berger, Construction de ($\varphi,\Gamma$)-modules: repr\'esentations $p$-adiques et $B$-paires,  Algebra and Number Theory, 2 (2008), no. 1, 91--120.
\bibitem[Ber08b]{Ber08b}
L. Berger, \'Equations diff\'erentielles $p$-adiques et ($\varphi$,N)-modules filtres, Asterisque (2008), no. 319, 13-38, Repr\'esentations $p$-adiques de groupes $p$-adiques. I. Repr\'esentations galoisiennes et ($\varphi$,N)-modules.
\bibitem[BK90]{BK90}
S. Bloch, K.Kato, $L$-functions and Tamagawa numbers of motives. The Grothendieck Festschrift, Vol. I, 333-400, Progr. Math. 86, Birkh\"auser Boston, Boston, MA 1990.
\bibitem[CC98]{CC98}
F. Cherbonnier,P. Colmez, Repr\'esentations $p$-adiques surconvergentes. Invent. Math. 133 (1998), 581-611.
\bibitem[CC99]{CC99}
F.Cherbonnier,P. Colmez, Th\'eorie d'Iwasawa des repr\'esentations $p$-adiques dun corps local. J. Amer. Math. Soc. 12 (1999), 241-268.
\bibitem[Ch12]{Ch12}
G.Chenevier, Sur la densit\'e des representations cristallines de $\mathrm{Gal}(\overline{\mathbb{Q}}_p/\mathbb{Q}_p)$, Math. Ann. (2012).
\bibitem[Col98]{Col98}
P. Colmez,Th\'eorie d'Iwasawa des repr\'esentations de de Rham d'un corps local. Ann. of Math. 148 (1998), 485-571.
\bibitem[Cr98]{Cr98}
R. Crew, Finiteness theorems for the cohomology of an overconvergent isocrystal on a curve, Ann. Sci. \'Ecole Norm. Sup. (4) 31(6) (1998) 717?763.

\bibitem[Fo90]{Fo90}
J.-M. Fontaine, Repr\'esentations $p$-adiques des corps locaux I. The Grothendieck Festschrift, Vol. 2, Progr. Math. 87, Birkh\"auser, Boston, 1990, 249-309.
\bibitem[Fo94]{Fo94}
J.-M. Fontaine, Le corps des p\'eriodes $p$-adiques,  Ast\'erisque 223 (1994), 59-111.
\bibitem[Fo03]{Fo03}
J.-M. Fontaine,  Presque $\mathbb{C}_p$-repr\'esentations. Kazuya Kato's fifties birthday. Doc. 
Math. 2003, Extra Vol., 285-385 (electronic). 

\bibitem[Her98]{Her98}
 L.Herr, Sur la cohomologie galoisienne des corps $p$-adiques, Bull. Soc. Math. France 126 (1998), no. 4, 563-600.
\bibitem[Her01]{Her01}
 L.Herr,Une approche nouvelle de la dualit\'e locale de Tate, Math. Ann. 320 (2001), no. 2, 307-337.

\bibitem[Ka93a]{Ka93a}
K. Kato, Lectures on the approach to Iwasawa theory for Hasse-Weil $L$-functions via $\bold{B}_{\mathrm{dR}}$. Arithmetic algebraic geometry, Lecture Notes in Mathematics 1553, Springer-
Verlag, Berlin, 1993, 50-63.
\bibitem[Ka93b]{Ka93b}
K.Kato, Lectures on the approach to Iwasawa theory for Hasse-Weil $L$-functions via $\bold{B}_{\mathrm{dR}}$. II, preprint (1993). 
\bibitem[Ka04]{Ka04}
K.Kato, $p$-adic Hodge theory and values of zeta functions of modular forms, Ast\'erisque (2004), no. 295, ix, 117-290, Cohomologies p-adiques et applications arithm\'etiques. III.
\bibitem[KKT96]{KKT96}
K. Kato, M.Kurihara, T.Tsuji, Local Iwasawa theory of Perrin-Riou and syntomic complexes. Preprint, 1996.
\bibitem[Ke04]{Ke04}
K.Kedlaya, A $p$-adic local monodromy theorem, Ann. of Math. (2) 160 (2004), 93-184.
\bibitem[Li08]{Li08}
 R. Liu, Cohomology and duality for ($\varphi,\Gamma$)-modules over the Robba ring, Int. Math. Res. Not. IMRN (2008), no. 3.
 \bibitem[Na09]{Na09}
K.Nakamura, Classification of two dimensional split trianguline representations of $p$-adic fields, Compositio Math. 145 (2009), 865-914.
\bibitem[Na10]{Na10}
 K.Nakamura, Deformations of trianguline $B$-pairs and Zariski density of two dimensional
crystalline representations, preprint arXiv:1006.4891 [math.NT].
\bibitem[Na12]{Na12}
K.Nakamura, A generalization of Kato's local $\varepsilon$-conjecture for 
$(\varphi,\Gamma)$-modules over the Robba ring, in preparation.
\bibitem[Per92]{Per92}
B.Perrin-Riou,Th\'eorie d'Iwasawa et hauteurs $p$-adiques. Invent. Math. 109, 137-185 (1992).
\bibitem[Per94]{Per94}
B. Perrin-Riou, Th\'eorie d'Iwasawa des repr\'esentations $p$-adiques sur un corps
local. Invent. Math. 115 (1994) 81-161.
\bibitem[Per95]{Per95}
B. Perrin-Riou, Fonctions $L$ $p$-adiques des repr\'esentations $p$-adiques. Ast\'erisque
No. 229 (1995), 198 pp.

\bibitem[Po12a]{Po12a}
J. Pottharst, Analytic families of finite-slope selmer groups, preprint on his web page.
\bibitem[Po12b]{Po12b}
J. Pottharst, Cyclotomic Iwasawa theory of motives, preprint on his web page.





\end{thebibliography}
\end{document}